\documentclass[a4paper, 10pt, DIV = 13]{scrartcl}
\usepackage[utf8]{inputenc}
\usepackage[english]{babel}

\usepackage{amsmath,amsfonts,amssymb,amsthm,mathrsfs}
\usepackage{amsfonts}
\usepackage{amssymb}
\usepackage{esint}
\usepackage{wasysym}
\usepackage{lmodern}
\usepackage{hyperref}
\usepackage{xcolor}
\usepackage{anyfontsize}
\usepackage{hyperref}
\hypersetup{breaklinks=true}
\usepackage{tikz-cd}

\usepackage{mathtools}
\usepackage{bm}

\definecolor{darkred}{rgb}{0.5,0,0}
\definecolor{darkgreen}{rgb}{0,0.5,0}
\definecolor{darkblue}{rgb}{0,0,0.5}
\hypersetup{colorlinks,
linkcolor=darkblue,
filecolor=darkgreen,
urlcolor=darkred,
citecolor=darkblue}

\numberwithin{equation}{section}

\usepackage[shortlabels,inline]{enumitem}
\setlist{nosep}
\setlist{noitemsep}
\setlist{leftmargin=*}

\usepackage{mathtools}

\newtheorem{theorem}{Theorem}[section]
\newtheorem*{theorem*}{Theorem}

\newtheorem{proposition}[theorem]{Proposition}
\newtheorem{lemma}[theorem]{Lemma}
\newtheorem{corollary}[theorem]{Corollary}

\newtheorem{claim}[theorem]{Claim}

\theoremstyle{definition}
\newtheorem{definition}[theorem]{Definition}
\newtheorem{remark}[theorem]{Remark}
\newtheorem*{assumption}{Assumption}

\newcommand{\R}{\mathbb{R}}

\newcommand{\Cc}{\mathtt{C}} 
\renewcommand{\O}{\mathcal{O}} 
\renewcommand{\div}{\mathrm{div}}

\newcommand{\hal}{\frac{1}{2}}

\renewcommand{\epsilon}{\varepsilon}

\newcommand{\z}{\mathsf{0}}
\renewcommand{\u}{\mathsf{1}}
\newcommand{\dd}{\mathtt{d}}
\newcommand{\id}{\mathrm{Id}}

\def\Xint#1{\mathchoice
   {\XXint\displaystyle\textstyle{#1}}%
   {\XXint\textstyle\scriptstyle{#1}}%
   {\XXint\scriptstyle\scriptscriptstyle{#1}}%
   {\XXint\scriptscriptstyle\scriptscriptstyle{#1}}%
   \!\int}
\def\XXint#1#2#3{{\setbox0=\hbox{$#1{#2#3}{\int}$}
     \vcenter{\hbox{$#2#3$}}\kern-.5\wd0}}
\def\dashint{\Xint-}

\newcommand{\sig}{\bx}

\newcommand{\Wn}{\mathcal{W}_n}

\renewcommand{\P}{\mathrm{P}}
\newcommand{\p}{\mathrm{p}}

\renewcommand{\d}{\mathsf{d}}
\newcommand{\Zd}{\mathbb{Z}^{\d}}
\newcommand{\La}{\Lambda}
\newcommand{\Lan}{{\La_n}}

\newcommand{\M}{M}

\newcommand{\Conf}{\mathsf{Conf}}
\newcommand{\Confn}{{\mathsf{Conf}_n}}

\newcommand{\tu}{\theta_u}
\newcommand{\tus}{\left(\tu\right)_\star}
\newcommand{\HLa}{\mathsf{H}_\La}

\newcommand{\HLan}{\mathsf{H}_{\Lan}}

\newcommand{\HLam}{\mathsf{H}_{\La_m}}

\newcommand{\HH}{\mathsf{H}}
\newcommand{\Pp}{\mathscr{P}}
\newcommand{\Ppn}{\mathscr{P}_{n}}
\newcommand{\Ppi}{\Pp(\Conf)}
\newcommand{\Esp}{\mathbb{E}}

\newcommand{\Ppis}{\Pp^{\mathrm{s}}}

\newcommand{\sigz}{\sig^\z}
\newcommand{\sigu}{\sig^\u}
\newcommand{\di}{\mathsf{d}}
\newcommand{\dn}{\di_n}
\newcommand{\Pz}{\P^\z}
\newcommand{\Pu}{\P^\u}
\newcommand{\Wi}{\mathcal{W}}
\newcommand{\Ee}{\mathcal{E}}
\newcommand{\Een}{\mathcal{E}_n}
\newcommand{\Eei}{\mathcal{E}}
\newcommand{\Hh}{\mathcal{H}}
\newcommand{\Hhn}{\mathcal{H}_n}
\newcommand{\Hhi}{\mathcal{H}}

\newcommand{\Stat}{\mathtt{Stat}}

\newcommand{\Error}{\mathsf{Error}}

\newcommand{\gL}{g^{\Lambda}}

\DeclareMathOperator*{\argmin}{arg\,min}

\usepackage{xargs}          
\usepackage{xparse}         
\usepackage{csquotes}       
\usepackage{microtype}

\DeclareMathOperator{\vol}{\mathsf{vol}}

\newcommand{\given}[1][]{%
  \nonscript\:#1\vert
  \allowbreak
  \nonscript\:
\mathopen{}}

\DeclareMathOperator{\prob}{\mathbb{P}}

\DeclareMathOperator{\var}{\mathbb{Var}}
\DeclarePairedDelimiterXPP{\Prob}[1]{\prob}[]{}{#1}
\DeclarePairedDelimiterXPP{\Var}[1]{\var}[]{}{\renewcommand\given{\nonscript\:\delimsize\vert\nonscript\:\mathopen{}} #1}
\DeclarePairedDelimiter{\paren}{\lparen}{\rparen}
\DeclarePairedDelimiter{\bracket}{\lbrack}{\rbrack}
\DeclarePairedDelimiter{\set}{\lbrace}{\rbrace}
\DeclarePairedDelimiter{\abs}{\lvert}{\rvert}
\DeclarePairedDelimiter{\norm}{\lVert}{\rVert}
\DeclarePairedDelimiterX{\hsp}[2]{\langle}{\rangle}{#1, #2}
\DeclarePairedDelimiterX{\pairing}[2]{\langle}{\rangle}{#1 \mid #2}
\DeclarePairedDelimiterX{\hpairing}[2]{\lcurvyangle}{\rcurvyangle}{#1 \mid #2}
\DeclarePairedDelimiterXPP{\Inf}[1]{\inf}{\lbrace}{\rbrace}{}{#1}
\DeclarePairedDelimiterXPP{\Exp}[1]{\exp}{\lparen}{\rparen}{}{#1}

\usepackage[capitalise,sort]{cleveref}

\newcommand{\Fbeta}{\mathcal{F}^\beta}
\newcommand{\Fn}{\mathcal{F}^\beta_n}
\newcommand{\fbeta}{\Fbeta}

\newcommand{\nn}{\mathsf{n}}
\newcommand{\Ricc}{\mathrm{Ricc}_M}
\newcommand{\V}{\mathrm{U}}
\newcommand{\W}{\Psi}
\newcommand{\tHLaD}{\widetilde{\mathsf{H}}_\La}
\newcommand{\tHLanD}{\widetilde{\mathsf{H}}_{\Lan}}
\newcommand{\dist}{\mathrm{dist}}
\newcommand{\T}{\mathrm{T}}

\newcommand{\bx}{\bm{x}}
\renewcommand{\epsilon}{\varepsilon}

\begin{document}
\author{\large{Ronan Herry\thanks{IRMAR, Université de Rennes,
263 avenue du Général Leclerc, 35042 Rennes Cedex, France \texttt{ronan.herry@univ-rennes.fr}} , Thomas Leblé\thanks{ Université de Paris-Cité, CNRS, MAP5 UMR 8145, F-75006 Paris, France \texttt{thomas.leble@math.cnrs.fr} } }}
\title{\LARGE{Gradient flow of the infinite-volume free energy for lattice systems of continuous spins}}
\date{\small{\today}}
\maketitle
\vspace{-1cm}

\begin{abstract}
We consider an infinite lattice system of interacting spins living on a smooth compact manifold, with short- but not necessarily finite-range pairwise interactions. We construct the gradient flow of the infinite-volume free energy on the space of translation-invariant spin measures, using an adaptation of the variational approach in Wasserstein space pioneered by Jordan, Kinderlehrer, and Otto in \cite{Jordan_1998}.
We also construct the infinite-volume diffusion corresponding to the so-called overdamped Langevin dynamics of the spins under the effect of the interactions and of thermal agitation.

We show that the trajectories of the gradient flow and of the law of the spins under this diffusion both satisfy, in a weak sense, the same hierarchy of coupled parabolic PDE's, which we interpret as an infinite-volume Fokker--Planck--Kolmogorov equation. We prove regularity of weak solutions and derive an Evolution Variational Inequality for regular solutions, which implies uniqueness. Thus, in particular, the trajectories of the gradient flow coincide with those obtained from the Langevin dynamics.

Concerning the long-time evolution, we check that the free energy is always non-increasing along the flow and that moreover, if the Ricci curvature of the spin space is uniformly positive, then at high enough temperature the dynamics converges exponentially, in free energy and in specific Wasserstein distance, to the unique minimizer of the infinite-volume free energy.
\end{abstract}

\section{Introduction}

\emph{Gradient flows in the Wasserstein space} provide powerful tools to study asymptotic properties of dynamical systems arising from physics.
Initially developed by the seminal contribution \cite{Jordan_1998} for the \emph{Fokker--Planck equation}, this approach is now implemented for numerous dynamics, notably the \emph{porous media equation} \cite{OttoPorous2}, \emph{granular flows} \cite{Carrillo_2003}, and the \emph{Boltzmann equation} \cite{ErbarBoltzmann}.
See, for instance \cite[Chap.~15]{VillaniOldNew}, or \cite[Chap.~18]{ABSOptimal}, for broader introductions to this topic.

Recent works \cite{ErbarHuesmannConfiguration,DSHSRicciPoisson,erbar2023optimal} have further extended this approach to \emph{infinite-volume non-interacting particle systems}.
In the present work, we develop a similar approach for another class of models coming from statistical physics: \emph{infinite-volume interacting spin systems}, and show that their natural evolution under the presence of thermal agitation coincides with the \emph{gradient flow} of the corresponding infinite-volume \emph{free energy} functional.

\subsection{General setting}
Let $\d \geq 1$ (the “lattice dimension”), let $\nn \geq 1$ (the “spin dimension”) and let the \emph{single-spin state space} $M$ be any smooth compact connected manifold of dimension $\nn$ without boundary, e.g. the $\nn$-dimensional sphere. We define an \emph{infinite spin configuration} as a vector $(\bx_i)_{i\in \Zd} \in M^{\Zd}$, and an \emph{infinite spin measure} as the law of a random infinite spin configuration. 

As usual in statistical mechanics, we consider a certain interaction energy between the spins, and we fix the value of the inverse temperature parameter $\beta \in (0, + \infty)$, which represents thermal agitation: this forms a \emph{spin system}. 

The corresponding infinite-volume \emph{free energy} $\fbeta$ is a functional defined on the space of translation-invariant infinite spin measures. Studying this functional is relevant because, under normal physical evolution of such a spin system, one expects (see e.g. \cite[§ 15]{landau2013statistical}) that the free energy \emph{decreases with time} and that the random state of the spins converges to a \emph{minimizer of the free energy}.

In this paper, we use ideas and techniques coming from optimal transportation of measures and the theory of gradient flows in metric spaces to study three objects related to the evolution of those systems: 
\begin{itemize}
  \item The \emph{gradient flow} of the infinite-volume free energy functional $\fbeta$, which forms a trajectory in the space of translation-invariant infinite spin measures.

  \item The \emph{infinite-volume diffusion} associated to the Langevin dynamics of the spin system, which gives rise to random paths $t \mapsto \bx(t)$ among spin configurations. At each time $t$, the law of $\bx(t)$ is a spin measure.

  \item The \emph{infinite-volume Fokker-Planck-Kolmogorov equations}, which describe a time evolution in the space of infinite spin measures. Solutions can be understood in several senses (dual, weak or strong).
\end{itemize}

\subsection{Summary of the results}
We give here an informal presentation of our results, with references to the precise statements.

\paragraph{Definition of the objects.} The first task consists in giving a meaning to the gradient flow, and in ensuring that the infinite-volume diffusion is well-defined.
\begin{enumerate}
  \item The gradient flow is constructed by adapting the discrete “JKO scheme” of \cite{Jordan_1998} to the infinite-volume setting. We show that our infinite-volume adaptation of the JKO scheme (presented in Section \ref{sec:discrete_scheme}) has a limit (as the discretization step-size tends to $0$), which we call the gradient flow of the free energy (Theorem \ref{prop:Convergence}). 
  \item We show, in Theorem \ref{th:diffusion:solution}, that the infinite-volume diffusion associated to the infinite spin system is well-defined.
    To do so, we embed the configuration space $\mathsf{Conf}$ in a suitable Hilbert space, and verify that the gradient of the interaction energy is Lipschitz for this structure.
    This diffusion can also be understood as the limit of finite-volume diffusions (Theorem \ref{th:diffusion:cauchy}).
\end{enumerate}

\paragraph{Link between gradient flow and dynamics via Fokker-Planck equations.} Our main goal is then to justify that these two objects coincide. The connection is established through the infinite Fokker-Planck-Kolmogorov equations. We prove that:
\begin{enumerate}
  \item The trajectories of the gradient flow and the law of the infinite diffusion are both solutions to the infinite-volume Fokker-Planck-Kolmogorov equations (introduced in Section \ref{sec:IFP}) in a dual sense (Theorem \ref{prop:plalFP} and Theorem \ref{th:diffusion:fokker-planck}).
  \item All dual solutions have local densities which are strong solutions (Theorem \ref{theo:regul}).
  \item Strong solutions are unique (Corollary \ref{th:uniqueness}), which follows from the EVI characterization of the flow (see below). 
\end{enumerate}
As a consequence, the law of the infinite-volume diffusion and the trajectories of the gradient flow coincide.

\paragraph{EVI characterization of the gradient flow.} 
Next, to support the idea that our trajectories are indeed gradient flows and to understand fine properties of such dynamics, we prove, in Theorem \ref{th:evi}, that every solution to the Fokker--Planck--Kolmogorov equation satisfies an \emph{Evolution Variational Inequality (EVI)}. In an abstract metric setting (see e.g. \cite{ambrosio2005gradient}), EVI's are used to characterize gradient flows and to obtain structural properties of the dynamics.

We refer to Section \ref{sec:evi} for a precise definition of EVI gradient flows. Let us simply point out here that this inequality involves: the dynamics, the free energy, the \emph{specific Wasserstein distance} (see Section \ref{sec:Wasserstein}), and a convexity / curvature constant $K_{\beta}$ which depends explicitly on $\beta$, the interaction, and the Ricci curvature of $M$.

\paragraph{Long-time behavior.} 
It is physically expected, and mathematically proven for a large class of models, that the free energy is always non-increasing along the trajectories. However, the long time behavior is not clear in general, as is often the case for gradient descents within complex “energy landscapes”.

Using the theory of displacement convexity together with the EVI formulation of the gradient flow, we prove that if the spin space $M$ has Ricci curvature bounded from below by a positive constant, and if the inverse temperature $\beta$ is small enough, then:
\begin{enumerate}
  \item The free energy functional $\mathcal{F}_{\beta}$ has a unique minimizer $\mathsf{P}^{*}$ (Theorem \ref{theo:uniqueness}).
  \item The solution to the Fokker--Planck--Kolmogorov equations converges exponentially fast to $\mathsf{P}^{*}$ as $t \to \infty$ (Corollary \ref{th:long-time-behaviour}).
    This exponential convergence holds with respect to both the specific Wasserstein distance and the free energy.
\end{enumerate}

\paragraph{Comments on our results.} 
 The construction of the gradient flow for an infinite-volume free energy functional is new. The only related example that we are aware of is the recent identification in \cite[Thm. 1.5.]{erbar2023optimal} of the gradient flow for the specific (i.e. infinite-volume) relative entropy with respect to the Poisson point process, with techniques which can be considered similar in spirit as the ones we use here. Point processes are technically more challenging than spin systems, but the free energy that we consider contains an interaction term in addition to the specific entropy.

Infinite-volume Fokker-Planck-Kolmogorov equations have been considered before, see e.g. \cite[Chap.~10]{Bogachev_2015} and the references therein. Our regularity result (Theorem \ref{theo:regul}) is new and might be of independent interest. The few existing uniqueness results do not apply to our case, and we obtain uniqueness of solutions thanks to our infinite-volume EVI formulation, which is also new in this context.

Uniqueness of the free energy minimizer at high temperature follows, \emph{without any curvature assumption}, from Dobrushin's uniqueness criterion combined with the Gibbs variational principle (see \cite[Cor. 6.37 \& Thm. 6.82]{FriedliVelenik}), so this part of our statement is weaker than existing ones. However, we give here a new proof using a form of \emph{infinite-volume displacement convexity}. 

Monotonicity of the free energy is not surprising (see for instance \cite[Theorem 4.25]{HolleyStroockDiffusionTorus} or \cite{Wick_1982}) and exponential convergence was known in certain finite-range models (\cite[Thm. 1, item 2]{Wick_1981}) in a weak sense: the law $\P(t)$ of the system at time $t$ converges exponentially fast (as $t\to \infty$) to $\mathsf{P}^{*}$ \emph{in a certain dual norm}. We provide here a unified treatment in greater generality. Moreover, our notion of convergence \emph{in specific Wasserstein distance} (and in free energy) is, by essence, uniform and thus stronger than the dual one from \cite{Wick_1981}. 

\newcommand{\CC}{\mathscr{C}}
\subsection{Choice of the interactions}
Recall that we model an infinite configuration of spins as an element $\bx = (\bx_i)_{i \in \Zd}$ of $\Conf := M^{\Zd}$. In order to encode the interaction energy of the spins, we fix:
\begin{itemize}
  \item A \emph{spin-spin interaction potential} $\W \in \CC^3(\M \times \M, \R)$, symmetric with respect to the two variables.
  \item A collection of \emph{spin-spin coupling coefficients}: for each $(i,j) \in \Zd \times \Zd$ we choose a real number $J_{i,j}$. 
\end{itemize}
Our assumptions on the coupling coefficients are as follows:
\begin{itemize}
  \item (Symmetry) For all $i,j$ we have $J_{i,j} = J_{j,i}$.
  \item (Translation-invariance) For all $i,j$, we have $J_{i,j} = J_{0, j-i}$.  
  \item (Short-range) The spin-spin couplings have \emph{short-range} in the following sense:
  \begin{equation}
  \label{eq:short-range}
  \|J\|_{\ell^1} := \sum_{i \in \Zd} |J_{0, i}| < + \infty.
  \end{equation}
\end{itemize}
In particular, $\|J\|_{\ell^\infty} := \sup_{i \in \Zd} |J_{0, i}|$ is also finite and, by translation invariance
\begin{equation*}
\sup_{j \in \Zd} \sum_{i \in \Zd} |J_{i,j}| < \infty.
\end{equation*}

\paragraph{Energy interaction and gradient of the energy} \newcommand{\TT}{\mathrm{T}}
One would like to define the energy of a given spin configuration $\bx$ as $\HH(\bx) \overset{??}{:=} \sum_{i,j \in \Zd} J_{i,j} \W(\bx_i, \bx_j)$, but this quantity is typically infinite. We explain in Section \ref{sec:freeenergy} how to define an energy per unit volume, see \eqref{def:HLa} and \eqref{finiteH}.

On the other hand, thanks to the short-range assumption \eqref{eq:short-range}, for any spin configuration $\bx$ we can define the \emph{gradient of the energy} $\nabla \HH(\bx)$ as:
\begin{equation*}
\nabla \HH(\bx) := \left(\sum_{j \in \Zd} J_{i,j} \partial_{1} \W(\bx_{i}, \bx_{j}) \right)_{i \in \Zd}
\end{equation*}
where $\partial_{1} \W$ is the gradient of $\W$ with respect to the first coordinate. For $\bx \in M^{\Zd}$, $\nabla \HH(\bx)$ belongs to $\Pi_{i \in \Zd} \left(\TT_{\bx_i} M \right)$, where $\TT_x M$ is the tangent space of $M$ at $x$.
Each component of the vector $\nabla \HH$ is always finite and bounded by $\|J\|_{\ell^1} \times \|\partial_1 \W\|_{\mathscr{L}^{\infty}}$.

\subsection{Discussion of the model and examples}
Imposing symmetry and translation-invariance on the couplings coefficients is a standard assumption, especially when dealing with stationary processes.
Many models of lattice spin systems simply set $J_{i,j} = 1$, if $i$ and $j$ are neighbors in $\Zd$, and $0$ otherwise. However, there is a significant interest in considering non-nearest neighbor situations. Two natural examples of \eqref{eq:short-range} are:
\begin{itemize}
	\item When the spin-spin couplings have \emph{finite range} i.e. when $i \mapsto J_{0,i}$ is compactly supported in $\Zd$. Many results in the literature are stated under this assumption.
	\item When the spin-spin couplings have a power law of the form $\left|J_{0,i}\right| \leq \frac{C}{(1+|i|)^{\d + s}}$ with $s > 0$. This has been considered in the physics literature since the 1960's, see e.g. \cite{joyce1966spherical,sak1973recursion,fisher1972critical} or \cite{aizenman1988critical} for a mathematical treatment. In the physics literature, this case is rather called “long-range” by opposition to “finite range”, but we prefer the terminology \emph{short-range}, in accordance to what would be used for interacting particle systems.
\end{itemize}

Regarding the nature of the interactions themselves, the case of a pure one-body interaction, that is setting $J_{i,j} = 0$ for all $i \neq j$ is much simpler to analyze.
In this case, the model has a product structure, and all the analysis boils down to understanding the dynamics (or the gradient flow) for a single spin.
On the other hand, the techniques of the present paper should extend readily to the case of a $k$-body interaction with $k \geq 2$, as long as the coupling coefficients satisfy some form of summability in the spirit of \eqref{eq:short-range} --- we stick to $k = 2$ for simplicity and because this is the main case of interest in the literature.

\subparagraph{$O(\nn)$ and Stochastic Heisenberg model.}
The $O(\nn)$ models form a well-studied class of lattice spin systems, for which $M$ is the $\nn$-dimensional unit sphere.
The interactions are given by a scalar product $\W(x, x') := x \cdot x'$, where $x$ and $x'$ are seen as vectors of $\mathbb{R}^{\mathsf{n}+1}$, and the spin-spin coupling is set to $J_{i,j} \coloneq 1$  if $i, j$ are lattice neighbors and $0$ otherwise.
This model is a continuous spin generalization of the celebrated Ising model, where $M \coloneq \{\pm1\}$, which informally corresponds to setting $\nn \coloneq 0$.
When $\nn \coloneq 1$, it is called the $XY$ model, and for $\nn \coloneq 2$ the Heisenberg model, we refer to \cite[Chap. 9]{FriedliVelenik} or \cite{peled2019lectures} for a presentation. 

The corresponding infinite-volume diffusion is introduced under the name “Stochastic Heisenberg model” in \cite{Faris_1979} and is further studied in \cite{Wick_1981,Wick_1982}. The setting there is actually more general: $M$ can be a general smooth compact manifold, and interactions are finite range but not necessarily nearest-neighbor.

\subparagraph{Non-compact spin space.} Despite an important part of the literature being devoted to compact spin spaces, the case of unbounded spin systems also receive some attention, for instance \cite{BodineauHelffer,LedouxUnbounded}, as well as the more recent works \cite{BauerschmidtBodineauPolchinski,BauerschmidtDagallier} in relation to Euclidean field theory.

In order to streamline the arguments, we limit our presentation to compact Riemannian manifolds.
We stress, however, that most of our reasoning applies verbatim to complete non-compact manifolds endowed with a sufficiently log-concave probability measure and whose geometry is suitably controlled. We discuss the non-compact case further in Section \ref{s:non-compact}.

\subsection{Overview of the three dynamics}
\label{sec:overview}
We informally present the three dynamics on spin configurations or spin measures mentioned above.

\paragraph{The infinite-volume diffusion.}
In Section \ref{sec:sde}, we construct the \emph{infinite-volume overdamped Langevin dynamics}, which is a stochastic process $t \mapsto \bx_t$ on the space $\Conf$ of spin configurations, corresponding to a Brownian motion on $M^{\Zd}$ with drift $- \beta \nabla \HH$. Informally, we consider the following SDE:
\begin{equation}
\label{eq:langevin}
  \dd X_{t} = \sqrt{2} \dd B_{t} - \beta \nabla \HH(X_{t}) \dd t,
\end{equation}
where $B_{t} = (B^{i}_{t})_{i \in \Zd}$ is a family of independent Brownian motions on $M$, possibly with a drift coming from $U$, indexed by $\Zd$. Since $\Conf$ is not a manifold, the well-posedness of \eqref{eq:langevin} does not follow from the standard theory.

A solution of \eqref{eq:langevin} is a random trajectory in the space of spin \emph{configurations}, and its law follows the corresponding trajectory in the space of spin \emph{measures}.

\paragraph{Infinite-volume Fokker--Planck--Kolmogorov equations.} \newcommand{\oL}{\omega_\La}
We start with the Kolmogorov equation, which is an equation describing a trajectory $t \mapsto \P(t)$ in the space of infinite spin \emph{measures}. It is expressed by saying that for all smooth functions $(t, \bx) \mapsto f(t, \bx)$ depending only on finitely many coordinates of $\bx$:
\begin{equation}
\label{Kolmo1}
\partial_t \Esp_{\P(t)}[f] = \Esp_{\P(t)} \left[ \partial_t f + \Delta f - \beta \nabla \HH \cdot \nabla f \right].
\end{equation}
We refer to \eqref{Kolmo1} as the \emph{dual} formulation, see Section \ref{sec:formFP}. Next, when the restriction of $\P(t)$ to each finite subset $\La \Subset \Zd$ has a density $\p_\La$ with respect to a certain reference measure $\oL$ on $M^\La$, we can recast \eqref{Kolmo1} as an evolution equation satisfied by $\p_\La$.
We obtain a family (indexed by $\La \Subset \Zd$) of Fokker--Planck equations, which can be considered either in a weak form:
\begin{equation*}
\partial_t \int_{M^\La} f \p_\La \dd \omega_\La = \int_{M^\La} \left( \partial_t f + \Delta f - \beta \nabla f \cdot \Esp_{\P(t)}\left[\nabla \HH | \La \right] \right) \p_\La  \dd \omega_\La,
\end{equation*}
where $f$ is a test function as above, or in a strong sense, when the densities are smooth enough:
\begin{equation*}
\partial_t \p_\La = \Delta \p_\La + \beta \div\left( \p_\La \cdot \Esp_{\P(t)} \left[\nabla \HH | \La \right]  \right).
\end{equation*}
In these last two formulations, the conditional expectation $\Esp_{\P(t)} \left[\nabla \HH | \La \right]$ of $\nabla \HH$ with respect to the configuration in $\La$ (see Section \ref{sec:formFP} for a precise definition) involves the entire spin measure $\P(t)$ and not just its restriction to $\La$. The Fokker--Planck equations are thus coupled together.


\paragraph{Gradient flow.}
\newcommand{\F}{\mathcal{F}}
\newcommand{\tP}{\tilde{\P}}
The infinite-volume free energy functional, denoted by $\Fbeta$, is obtained (see Section \ref{sec:freeenergy}) as a limit as $n \to \infty$ of finite-volume functionals $\Fn$ defined for configurations living on large boxes $\Lan := \{-n, \dots, n\}^\d \Subset \Zd$.
To construct a \emph{Minimizing Movement Scheme}, we fix a step-size $h > 0$ and proceed as follows:
\begin{enumerate}
	\item Start with some initial condition $\Pz$ chosen among stationary spin measures.
	\item Assume that $\P^{k}$ has been constructed for $k \geq 0$, then:
	\begin{enumerate}
	\item (Variational scheme). Take $n$ large and solve:
	\begin{equation}
	\label{MMS2}
  \bar{\mathsf{P}}^{k+1}_n = \argmin_{\P \text{ spin measure on $\Lan$}} \left(\frac{1}{2} \Wn^2(\P^k, \P) + h\Fn(\P)\right),
	\end{equation}
	where $\Wn$ is the $2$-Wasserstein distance between spin measures on $\Lan$ (see Section \ref{sec:Wasserstein} for reminders).
  This variational step yields a \emph{finite-volume} spin measure.
\item (Compatibilization step).
  Turn the finite-volume spin measure $\bar{\mathsf{P}}^{k+1}_n$ into an infinite-volume, \emph{stationary} spin measure using the construction given in Section \ref{sec:stationarisation}, and choose this as the next iterate $\P^{k+1}$.
\end{enumerate}
\end{enumerate}
The variational step is as in \cite{Jordan_1998} but the compatibilization step is new and specific to our context. We then define a \emph{discrete-time gradient flow} with step-size $h$ by setting:
\begin{equation*}
  \mathrm{P}_{h}(t) := \mathrm{P}^{k} \text{ for } t \in [kh, (k+1)h),
\end{equation*}
and we send $h \to 0$ to obtain the trajectory of the \emph{gradient flow} of $\Fbeta$ starting from $\Pz$.

\subsection{Connection with the literature}
\label{sec:literature}

\paragraph{Gradient flows.}
The link between solutions to the Fokker--Planck equation and gradient flows of the free energy starts with the seminal work \cite{Jordan_1998} by Jordan, Kinderlehrer, and Otto (JKO). They consider the usual Fokker--Planck equation on $\mathbb{R}^\nn$, namely:
\begin{equation*}
  \partial_{t} \rho = \Delta \rho + \nabla \cdot (\rho \nabla \V),
\end{equation*}
which describes the evolution of the law of overdamped Langevin dynamics on $\mathbb{R}^{\nn}$, given by:
\begin{equation*}
  \mathtt{d} X_{t} = \sqrt{2} \mathtt{d} B_{t} - \nabla \V(X_{t}) \mathtt{d} t,
\end{equation*}
where $(B_{t})$ is the Brownian motion on $\mathbb{R}^{\nn}$ and $\V$ is some smooth potential. They connect these equations to a gradient descent on the space of probability measures endowed with the $2$-Wasserstein distance, by showing that the law of $X_t$ follows a steepest descent with respect to the \emph{free energy} functional
\begin{equation*}
\mu \mapsto  \mathcal{F}(\mu) :=  \int \mu \log \mu \ \mathtt{d} x + \int \V \mathtt{d} \mu,
\end{equation*}
Note that this free energy can also be written as a \emph{relative entropy} with respect to $\nu := \mathrm{e}^{-U} \mathtt{d} x$: 
\begin{equation}
\label{def:Fmu}
\mathcal{F}(\mu) = \int \log \frac{\mathtt{d} \mu}{\mathtt{d} \nu} \mathtt{d} \mu, 
\end{equation}
Taking $\V \equiv 0$ gives the heat equation, which can thus be interpreted as the gradient flow of the Boltzmann entropy $\mu \mapsto \int \mu \log \mu \mathtt{d} x$.

Continuing this approach in order to handle non-linear PDEs, Otto \cite{OttoPorous,OttoPorous2} has developed a Riemannian-like structure on the Wasserstein space over $\mathbb{R}^{\nn}$. This Riemannian interpretation makes the notion of gradient and gradient flows on the Wasserstein space rigorous. 

Since then, a considerable literature has been devoted to the study of gradient flows in Wasserstein space and in more general metric spaces, and which we do not attempt to review it  here --- see the monographs \cite{AGSGradient,santambrogio2017euclidean,FigalliGlaudoInvitation,ABSOptimal}, as well as \cite{VillaniOldNew} for details.
Regarding manifolds let us simply mention that an adaptation of the scheme for the heat flow on, possibly non-compact, Riemannian manifolds can be found in \cite{Zhang_2007}, and that \cite{Erbar_2010} shows that the heat flow satisfies an \emph{Evolution Variational Inequality} with an explicit constant connected to the \emph{Ricci curvature} of the manifold. 

Our work relates to this line of research by showing that the infinite-volume Langevin dynamics on spins is indeed the gradient flow of the free energy $\fbeta$.
Contrarily to the case of a Riemannian manifold, we are working with infinitely many coordinates, these coordinates are interacting, and the free-energy is not a relative entropy, but a limit of relative entropies per volume.
Other recent works have studied gradient flows for infinite-volume systems coming from statistical physics:
\begin{itemize}
	\item \cite{ErbarHuesmannConfiguration} shows that, on a possibly non-compact Riemannian manifold with Ricci curvature bounded from below, the gradient flow, for the Wasserstein distance, of the relative entropy with respect to the Poisson point process is given by a family of non-interacting Brownian motions.
	\item \cite{DSHSRicciPoisson} shows that the infinite-volume birth-and-death process corresponds to the gradient flow, for a non-local transport distance, of the relative entropy with respect to the Poisson point process.
  \item \cite{erbar2023optimal} shows that the gradient flow, for the \emph{specific} Wasserstein distance, of the \emph{specific entropy} with respect to the Poisson point process is given by a family of non-interacting Brownian motions on $\mathbb{R}^{\nn}$.
  \item \cite{SuzukiCurvatureDyson} shows that the Dyson Brownian motion is the gradient flow, for the Wasserstein distance, of the relative entropy with respect to the $\mathsf{Sine}_{\beta}$ point process of Valkó-Virág. It also derives a curvature lower bound which is $0$ for all $\beta$.
\end{itemize}
The works \cite{ErbarHuesmannConfiguration,DSHSRicciPoisson,erbar2023optimal} consider \emph{non-interacting} evolutions on point processes, whereas \cite{SuzukiCurvatureDyson} deals with a long-range interaction coming from the one-dimensional logarithmic potential.

The references \cite{ErbarHuesmannConfiguration,DSHSRicciPoisson,SuzukiCurvatureDyson} work in a non-stationary setting and do not consider specific quantities. In particular, the starting point of gradient flow must belong to the domain of the relative entropy, which does not contain any stationary measure other than the reference measure. In contrast, \cite{erbar2023optimal} deals with stationary objects, as in the present work.

\paragraph{The Bakry--Émery condition.}
The ergodic and contractive properties of a gradient flow are often related to the convexity of the free energy functional.
Using the notion of \emph{displacement convexity} introduced by McCann \cite{McCann_1997}, \cite{cordero2001riemannian,vonRenesseSturm} show that convexity of the relative entropy connects to Ricci curvature lower bound on the manifold (see also \cite[Part II]{VillaniOldNew}).
Namely, on the Riemannian manifold $M$, the functional $\mathcal{F}$, from \eqref{def:Fmu}, is $\kappa$-convex along geodesics, in the sense of optimal transport, if and only if the \emph{Bakry--Émery condition} $\mathrm{Ric}_{M} + \nabla^{2} \V \geq \kappa$ holds, where $\mathrm{Ric}_{M}$ is the Ricci tensor.

This condition, introduced in \cite{BakryEmery} to study diffusions, is equivalent to certain functional inequalities, such as the celebrated \emph{logarithmic Sobolev inequality} or the \emph{$\Gamma_{2}$-criterion}.
These inequalities can typically be considered without referring explicitly to the Riemannian structure of the base space, using only the diffusion, and as consequence the point of view introduced in \cite{BakryEmery} is well-suited to investigate diffusions on abstract non-Riemannian spaces, such as infinite-dimensional diffusions. Numerous works have followed this path:
\begin{itemize}
  \item \cite{CarlenStroock} is the first paper to give a sufficient condition for infinite-volume spin systems to satisfy the Bakry--Émery condition;
  \item \cite{StroockZegarlinskiLogSob,StroockZegarlinskiMixing,Laroche} where the equivalence of the Bakry--Émery condition with the \emph{Dobrushin--Sloshman mixing} condition is established;
  \item \cite{HelfferUnbounded,BodineauHelffer,LedouxUnbounded} study the case of infinite-volume spin systems with non-compact continuous state space (see also \cite{Yoshida} for the discrete case);
  \item \cite{BartheMilman} has revisited this literature to provide effective bounds on the constant in the logarithmic Sobolev inequality for spin systems;
  \item \cite{BauerschmidtBodineau} derives a new spectral condition to obtain logarithmic Sobolev inequalities for spin models.
\end{itemize}
All these results provide, at sufficiently high temperature, a control uniform in the size of the box $\Lambda$, on the constant in the logarithmic Sobolev inequality.
Such a control is equivalent to the Dobrushin--Sloshman condition, which in particular implies:
\begin{enumerate}[(i)]
  \item the uniqueness of the Gibbs measure;
  \item exponential stabilisation in relative entropy or in $L^{2}$ with respect to the Gibbs measures.
\end{enumerate}
We refer to the monographs \cite{Royer,GuionnetZegarlinski} for a broader introduction on the subject.

All the above works provide information about the model only at high enough temperature.
Moreover, the results about stabilisation hold for \emph{near-equilibrium} initial conditions, that is for measures that are absolutely continuous with respect to the Gibbs measure.
For instance, \cite[Thm., p.~346]{CarlenStroock} gives exponential convergence in $L^{2}$ of a Gibbs measure and thus allows to only handle convergence of measures having a density with respect to the said Gibbs measure, which is a strong limitation especially in the stationary framework.
It is not clear whether this approach based on logarithmic Sobolev inequalities could yield results for non absolutely continuous measures, using the specific entropy or another suitable notion.
At least informally, dividing by $|\Lambda|$ and taking the limit as $\Lambda \to \Zd$ in this family of logarithmic Sobolev inequalities should yield a \emph{specific logarithmic Sobolev inequality}.
However, turning this intuition into precise results is beyond the scope of our paper.

In this article, we take another route and work from the point view of gradient flows of the specific entropy.
Apart from providing exponential rates of convergence in specific entropy in the high temperature regime, this alternative approach also shows that the free energy decays along the dynamics for \emph{every} temperature.

\paragraph{Infinite-volume FPK equation.}  The infinite-volume Fokker--Planck--Kolmogorov equation is discussed in \cite[Chap.~10]{Bogachev_2015} and the references therein.
The theory regarding existence, uniqueness, and regularity of solutions is not as fully developed as in the finite-volume case --- in particular, as explained in \cite[Chap.~10]{Bogachev_2015}, uniqueness is significantly harder to prove for infinite-volume Fokker--Planck--Kolmogorov equations.
Some conditions, related to the existence of good finite-dimensional approximations, appear in the literature, for instance \cite{bogachev2013analytic}, but they do not apply to our case.

Existence of solutions is typically proven by constructing the infinite-volume diffusion, as done for general \emph{finite-range} interactions in \cite{Faris_1979,royer1979processus,HolleyStroockDiffusionTorus,fritz1982infinite}.
Some of those papers include considerations on the regularity of the restriction to finite boxes $\Lambda \Subset \Zd$ of the infinite-volume solutions, through Malliavin calculus.

Monotonicity of the free energy along those diffusions is known \cite{Wick_1981,HolleyStroockDiffusionTorus}.
\cite{Wick_1981} also proves exponential convergence to equilibrium in a weak sense (against each fixed test function) when $M$ is a \emph{homogeneous} manifold, which allows for Fourier-analytic techniques to be used. Our notion of convergence in specific Wasserstein distance, and in free energy, is much stronger, and holds for general compact manifolds.

Existence of weak solutions to the infinite-volume Fokker-Planck equation is proven (by stochastic analysis methods)  in \cite{lemle2013uniqueness} for lattice systems of continuous spins living on a compact manifold, with \emph{finite-range} interaction. They also claim to prove uniqueness of solutions, however we cannot follow\footnote{They seem to use the fact that the spectrum of an elliptic operator of the type $f \mapsto - \Delta f + \nabla \HH \cdot \nabla f$
(say $\HH$ is smooth) is always $\subset (-\infty, 0]$, which is clearly false. The upper bound on the spectrum will depend on $\HH$ and thus, in their context, on the box $\Lambda_n$ that they consider, hence it does not seem possible to choose $\lambda$ independent of $n$, which is crucial for them. They also rely on some total variation bound for the solution that is never truly explained.} the proof of their \cite[Lemma 3.1]{lemle2013uniqueness}. 

Our method provides a new approach to proving existence of solutions through the (limit of the) JKO scheme. Our uniqueness result is new, and derives from the EVI characterization. 

\paragraph{Infinite-volume dynamics.}
Our construction of the dynamics is valid for interactions satisfying the short-range assumption \eqref{eq:short-range}, which are not necessarily finite-range interactions. On the other hand, it is fair to note that several works from this period deal with $k$-body interactions for arbitrary $k$, whereas for simplicity we stick to $k =2$.

Concerning the construction of the dynamics, and the existence and uniqueness of solutions of the associated SDE, we closely follow the method developed by \cite{HolleyStroockDiffusionTorus} on the torus with finite-range interaction, and by \cite{LehaRitter} on $\mathbb{R}^{\mathsf{n}}$ with interaction satisfying \eqref{eq:short-range}. These two references proceed by first constructing a solution to the infinite-volume stochastic differential equation, and then using the martingale method.

\begin{remark}
Following \cite{LehaRitter}, \cite{ADK97,ADK03} develops a comprehensive theory for stochastic differential equations on infinite products of compact manifolds, and use it to derive abstract existence result for general diffusions.
The existence of our Langevin dynamic should follow from their analysis, but for completeness, we prefer to give a self-contained argument in our simpler setting.
\end{remark}

\paragraph{Optimal transportation for infinite-volume objects.} The recent paper \cite{erbar2023optimal} develops a framework for optimal transportation and gradient flows for stationary point processes directly at the infinite-volume level.
In short, translating their ideas from the point process setting --- which is arguably more challenging --- to spin systems, they proceed as follows:
\begin{enumerate}
  \item Define a \emph{cost per unit volume} directly at the level of infinite spin configurations, by setting:
  \begin{equation*}
    \di_\infty^2(\bx^0, \bx^1) := \limsup_{n \to \infty} \frac{1}{|\Lambda_{n}|} \dn^2\left(\bx^0_{|\Lambda_{n}}, \bx^1_{|\Lambda_{n}}\right), \qquad \bx^{0}, \bx^{1} \in \mathsf{Conf},
  \end{equation*}
  where $\dn(\cdot, \cdot)$ is the Riemannian distance on $M^{\Lambda_{n}}$.
   The limit might not exist in general, while, by compactness, the $\limsup$ is always finite.
\item Given two stationary spin measures $\mathrm{P}_{0}, \mathrm{P}_{1}$, define the \emph{stationary Wasserstein distance} by minimizing the associated transportation cost.
  Namely, define
  \begin{equation*}
    \mathcal{W}_{\infty}(\mathrm{P}_{0}, \mathrm{P}_{1}) := \inf_{\Pi} \paren*{\iint_{\mathsf{Conf} \times \mathsf{Conf}} \di_\infty^2(\bx^0, \bx^1) \mathtt{d} \Pi(\bx^0, \bx^1)}^{1/2},
  \end{equation*}
  where the infimum runs over all couplings $\Pi$ between $\mathrm{P}_{0}, \mathrm{P}_{1}$ with good invariance properties with respect to lattice shifts
\item Show that optimal couplings are actually \emph{matchings}, \emph{i.e.} are induced by an optimal transportation map $T \colon \mathsf{Conf} \to \mathsf{Conf}$. Define the corresponding notion of displacement interpolation.
  \item Introduce the relevant notion of displacement convexity, show that the specific relative entropy is displacement convex, etc.
\end{enumerate}
This program could likely be implemented for spin measures, presumably with fewer technical difficulties than in \cite{erbar2023optimal}. We believe that $\mathcal{W}_{\infty}$ would coincide with our specific Wasserstein distance $\Wi$, defined in \eqref{def:Wi}.

One could then seek to construct the gradient flow of $\fbeta$ by considering the corresponding minimizing movement variational scheme, that is, choosing for each iterate (cf \eqref{MMS2}):
\begin{equation}
\label{MMS3}
\mathrm{P}^{k+1} \in \arg\min \left( \hal \mathcal{W}^{2}_{\infty}(\mathrm{P}_k, \cdot) + h \fbeta(\cdot)  \right),
\end{equation}
the minimization being among stationary spin measures. A first difficulty here is that now the minimizer might not be unique.

This approach would avoid going through the stationarization step described in Section \ref{sec:JKO}.
However, deriving the Fokker--Planck equation from such as scheme requires to study the variational problem associated to the minimization in \eqref{MMS3} --- which involves perturbing infinitesimally the minimizers to derive Euler--Lagrange equations, see Section \ref{sec:JKO}.
It is not easy to proceed directly at the level of stationary spin measures, since one would need to perform global perturbations in a stationary way.
We would probably have to work in finite boxes and undergo some stationarization procedure.

\subparagraph{Uniqueness of minimizers and displacement convexity.}
The “infinite-volume displacement convexity” argument used here to prove uniqueness of minimizers of the free energy at high temperature in the case of positive curvature (Theorem \ref{theo:uniqueness}) is similar in spirit to the one of \cite{erbar2021one}, with two differences: 1) \cite{erbar2021one} deals with a \emph{point process} and 2) in \cite{erbar2021one}, the convexity comes from the interaction term in the free energy, whereas we get convexity from the entropy term.

\section{Setting and preliminaries}

\subsection{Spin configurations and spin measures}
\label{sec:single_spin}
At each site of the lattice $\Zd$, we place a spin with value in $M$. The dimensions $\nn$ of $M$ and $\d$ of the lattice play almost no role.

\paragraph{Single-spin space.}
We choose our single-spin space $M$ to be a connected $\nn$-dimensional smooth compact Riemannian manifold without boundary. We endow $M$ with the topology associated with its Riemannian metric and with the corresponding Borel $\sigma$-algebra. We let $\vol$ be the corresponding volume measure on $\M$, normalized to be a probability measure. We denote by $\di(\cdot, \cdot)$ the Riemannian distance on $M$.

We fix a smooth \emph{single-spin potential} $\V$ on $M$ such that $\int_M e^{-\V} \dd \omega = 1$ and we consider the measure $\omega$ on $M$ with density 
\begin{equation}
\label{eq:defomega}
\dd \omega := e^{-\V} \dd \vol.
\end{equation}
We can think of $e^{-\V}$ as a weight. Taking $\V \equiv 0$ gives back $\omega = \vol$.

We let $\nabla$ be the \emph{Levi--Civita connection}, $\Delta$ be the \emph{Laplace--Beltrami operator}, $\operatorname{div}_{\V} := (\operatorname{div} - \nabla \V \cdot)$ be the adjoint of $\nabla$ in $\mathscr{L}^{2}(\omega)$, where $\operatorname{div}$ is the usual divergence, and $\Delta_{\V} := - \operatorname{div}_{\V} \nabla = \Delta - \nabla \V \cdot \nabla$ for the weigthed Laplace--Beltrami operator.

By construction, $-\Delta_{\V}$ is a non-negative self-adjoint operator on $\mathscr{L}^{2}(\omega)$.

\subparagraph{Gradients and derivatives.}
For $x \in M$, we denote the tangent space at $x$ by $\mathrm{T}_{x}M$.

If $f : M \to \R$ is a $\CC^1$ map, for $x \in M$ we denote by $\nabla f(x)$ the gradient of $f$ at $x$, which is a vector in the tangent space $\T_x M$, and we write $\|\nabla f(x)\|^2 \coloneq g_{x}(\nabla f(x), \nabla f(x))$ for its norm computed in the tangent space through the inner product corresponding to the Riemannian structure on $M$.

If $\La \Subset \Zd$ is a finite subset and $f: M^\La \to \R$ is a $\CC^1$ function, at each point $\bx \in M^\La$ we see the gradient $\nabla f(\sig)$ as a family of tangent vectors indexed by $\La$, with 
\begin{equation*}
  \left(\nabla f(\sig)\right)_i = \nabla_i f(\sig) \in \T_{\sig_i}M, \qquad i \in \Lambda,
\end{equation*}
where $\nabla_i f(\sig)$ is the partial derivative of $f$ at $\sig$ with respect to the $i$-coordinate. We define $\|\nabla f(\sig)\|$ as the $\ell^2$ norm of this family, namely:
\begin{equation*}
\|\nabla f(\sig)\| := \left(\sum_{i \in \La} \|\nabla_i f(\sig)\|^2 \right)^\hal,
\end{equation*}
where each individual norm is computed in the tangent space of a single copy of $M$. 

If $f: M^\La \to \R$ is a $\CC^2$ function, we write:
\begin{equation}
\label{secondpartial}
\|\nabla^2 f(\sig)\| := \left(\sum_{i,j \in \La} \|\nabla^2_{ij} f(\sig)\|^2 \right)^\hal,
\end{equation}
Here $\nabla^2_{ij} f(\sig)$ corresponds to the second partial derivative of $f$ at $\sig$, which is a linear map from $\T_j M \to \T_i M$.

\subparagraph{Curvature.}
We denote by $\Ricc$ the Ricci curvature tensor of $M$: at every point of $M$, it gives a quadratic form on the tangent space, see \cite[Chap.~14]{VillaniOldNew}.
Following \cite{BakryEmery}, we will use the following “Bakry--Émery curvature lower bound” on the weighted manifold $(M, \omega)$:
\begin{equation}
\tag{Bakry-Émery}
\label{PosCurv}
\kappa \coloneq \sup \set*{ \kappa' \in \mathbb{R} : \mathrm{Ricc}_{M} + \nabla^{2} \mathrm{U} \geq \kappa g }.
\end{equation}
Here, $\nabla^2 \V$ denotes the Hessian of $\V$ as in \eqref{eq:defomega}.
By compactness, the constant $\kappa$ is finite. Its exact value of $\kappa$ plays no role in the definition of the gradient flow or for its connexion with Langevin dynamics, but the assumption $\kappa > 0$ will be crucial in order to analyse the long-time behavior of the trajectories.

\newcommand{\by}{\bm{y}}

\paragraph{Spin configurations.}
When $\La$ is a subset of the lattice $\Zd$, we let $\Conf(\La) \coloneq M^{\Lambda}$ be the set of \emph{spin configurations on $\La$}, namely vectors $\bx = (\bm{x}_{i})_{i \in \Lambda} \in M^{\Lambda}$. We identify $\Conf(\La)$ to the (at most countable) product $\M^\La$ as a topological and measurable space.

When $\La = \Zd$ itself, we simply write $\Conf := \Conf(\Zd)$.
Every configuration on $\Zd$ yields a configuration on $\La \subset \Zd$ for all $\La$ by restriction. If $\bx$ is a spin configuration and $i$ a lattice point, we denote the value of the spin at $i$ by $\bx_i$. 

The lattice $\Zd$ acts on $\Conf$ by translation: for $u \in \Zd$ and $\sig \in \Conf$ we denote by $\tu \cdot \sig$ the spin configuration such that $\left(\tu \cdot \sig\right)_i := \sig_{i+u}$ for all $i \in \Zd$. More generally, we use $\tu \cdot$ to denote a translation by $u$.

\subparagraph{Structure on $\M^\La$.}
When $\La$ is finite (we write $\La \Subset \Zd$) the product $\M^\La$ inherits a Riemannian manifold structure, as well as the corresponding distance $\d_\La^2(\bx, \by) := \sum_{i \in \La} \d^2(\bx_i, \by_i)$. It also inherits the product measure $\oL$:
\begin{equation}
\label{def:oL}
\dd \oL(\bx) =  e^{- \sum_{i \in \La} \V(\bx_i)} \prod_{i \in \La} \dd \vol(\bx_i)
\end{equation}

However, $\Conf = \Conf(\Zd)$ is not a manifold --- not even modelled on a infinite-dimensional vector space. Nevertheless, we define its “tangent space” at $\bx \in \Conf$ as:
\begin{equation*}
  \mathrm{T}_{\bx}\mathsf{Conf} := \prod_{i \in \Zd} \mathrm{T}_{\bx_{i}} M,
\end{equation*}
and we let $\mathrm{T} \mathsf{Conf}$ be the associated “\emph{tangent bundle}”. We then call a \emph{vector field} on $\mathsf{Conf}$ a section of $\mathrm{T}\mathsf{Conf}$.

If a function $\varphi \colon \Conf \to \mathbb{R}$ is differentiable at $\bx \in \mathsf{Conf}$ with respect to the site $i \in \Zd$, we write $\nabla_{i} \varphi(\bx) \in \mathrm{T}_{\bx_{i}} M$ for its gradient with respect to $i$. When $\varphi$ is differentiable at $\bx \in \Conf$ with respect to every site $i \in \Zd$, we write
\begin{equation*}
  \nabla \varphi(\bx) := \set*{ \nabla_{i} \varphi(\bx) : i \in \Zd } \in \mathrm{T}_{\bx}\mathsf{Conf}.
\end{equation*}
Finally, if $\varphi$ is differentiable at every $\bx \in \Conf$, we can consider its gradient vector field given by $\bx \mapsto \nabla \varphi(\bx)$.

We can give similar definitions for vector fields or tensor fields.
In particular, provided it makes sense, the second derivative $\nabla^{2}_{ij}\varphi(\bx) := \nabla_{j} \nabla_{i} \varphi(\bx)$ is an element of $\mathrm{T}_{\bx_{i}}M \otimes \mathrm{T}_{\bx_{j}} M$.
We also set $\Delta_{i}\varphi := \operatorname{Tr} \nabla^{2}_{ii}\varphi$, and $\Delta := \sum_{i \in \Zd} \Delta_{i}$. 
The infinite-volume version of the weighted Laplace--Beltrami operator is given by:
\begin{equation*}
  \Delta_{\V}f(\bm{x}) \coloneq \sum_{i \in \Zd} (\Delta_{i}f(\bm{x}) - \nabla \V(\bm{x}_{i}) \cdot \nabla_{i} f(\bm{x})).
\end{equation*}
 
\paragraph{Finite lattice boxes.}
For all $n \geq 1$, we introduce the box $\La_n := \{-n, \dots, n\}^{\d}$. 
We often use the subscript $n$ for objects restricted to $\La_n$, for example we write $\mathsf{Conf}_{n} := \mathsf{Conf}(\Lambda_{n})$, $\omega_{n} := \omega_{\Lambda_{n}}$, $\mathrm{dist}_{n} := \mathrm{dist}_{\Lambda_{n}}$, and so on.
 We also extend the definitions of $\Delta_{\V}$ and $\operatorname{div}_{\V}$ to $\mathsf{Conf}_{n}$ by acting coordinate-wise.

We say that a function $f \colon \mathsf{Conf} \to \mathbb{R}$, is \emph{$\Lan$-local} provided $f(\bx) = f(\bx')$ for all $\bx$ and $\bx' \in \mathsf{Conf}$ coinciding on $\Lambda_{n}$.
In particular $f$ can then be seen as a function on $\mathsf{Conf}_{n}$. More generally, we say that $f$ is \emph{local} provided it is $\Lan$-local for some $n \geq 1$.

\paragraph{Spin measures.} 
We define \emph{infinite spin measures} on $\Zd$ as elements of $\Pp(\Conf)$, the space of probability measures on $\Conf$, and for $n \geq 1$, we let $\Ppn := \Pp(\Confn)$. We call elements of $\Ppn$ \emph{finite} spin measures. There are two natural topologies on $\Pp(\Conf)$:
\begin{itemize}
   \item The weak topology, defined as as the coarsest topology such that $\P \mapsto \Esp_{\P}[f]$ is continuous for all functions $f$ which are continuous on $\Conf$ with respect to the product topology.
   \item The local topology, defined as the coarsest topology such that $\P \mapsto \Esp_{\P}[f]$ is continuous for all functions $f$ which are bounded and \emph{local}.
 \end{itemize} 
We introduce below another topology, induced by our specific Wasserstein distance, see \eqref{def:Wi}.

\subparagraph{Stationary measures.}
We let $\Ppis$ be the set of infinite spin measures that are invariant under all lattice shifts, namely:
 \begin{equation}
 \label{def:stati}
\Ppis := \left\lbrace \P \in \Ppi, \ \tus \P = \P \text{ for all } u \in \Zd \right\rbrace,
\end{equation} 
where $\tus$ is the push-forward by the translation $\tu$. When $\P \in \Ppis$, we say that $\P$ is a \emph{stationary} spin measure.

\subsection{The free energy functionals.} \label{sec:freeenergy}
\newcommand{\Pn}{\P_n}
\newcommand{\pn}{\p_n}
\paragraph{Interaction energy within a finite box.}
For any finite subset $\La \Subset \Zd$ and for $\bx \in \Conf(\La)$ we define \emph{the interaction energy $\HLa(\bx)$ of $\bx$ in $\La$} as:
\begin{equation}
\label{def:HLa}
\HLa(\bx) := \sum_{i \in \La, j \in \La} J_{i,j} \W(\bx_i, \bx_j).
\end{equation}

From the short-range assumption \eqref{eq:short-range} and the fact that $\W$ is a continuous function on a compact manifold and thus bounded, there exists a constant $\Cc$, depending on $M$ and $\W$, such that for all $\La \Subset \Zd$ and all $\bx$ in $\Conf$:
\begin{equation}
\label{eq:HLaStable}
- \Cc  \leq \frac{\HLa(\bx)}{|\La|} \leq \Cc.
\end{equation}

\begin{remark}[Two conventions for the spin-spin interaction in $\La$]
\label{rem:tHLaD}
In our definition \eqref{def:HLa} of the spin-spin interaction in $\La$, we only consider couples of spins that are \emph{both contained in $\La$}. This differs from the usual choice (see e.g. \cite[Section 6]{FriedliVelenik} or \cite[Chapter 2]{GeorgiiGibbs}), which consists in setting
\begin{equation}
\label{def:tH}
\tHLaD (\sig) := \sum_{i \in \La, j \in \Zd} J_{i,j} \W(\sig_i, \sig_j).
\end{equation}
Our convention turns out to be convenient later. In fact, under the short-range assumption, one would get the same infinite-volume functional. Indeed one has, as $n \to \infty$:
\begin{equation}
\label{HLaDvstHLaD}
\lim_{n \to + \infty} \sup_{\sig \in \Conf} \frac{1}{|\Lan|} \left| \HH_{\Lan}(\sig) - \tHLanD (\sig) \right| = 0.
\end{equation}
To prove \eqref{HLaDvstHLaD}, let $R_L := \|\W\|_{\mathscr{L}^\infty(M^\La \times M^\La)} \times \sum_{|i| \geq L} |J_{0,i}|$, and observe that this tends to $0$ as $L \to \infty$ by the short-range assumption \eqref{eq:short-range}. We can directly bound $\frac{1}{|\Lan|} \left| \HH_{\Lan}(\sig) - \tHLanD (\sig) \right|$ by $\frac{1}{|\Lan|} \sum_{i \in \Lan} R_{\dist(i, \partial \Lan)}$, which tends to $0$ as $n \to \infty$, since most points in $\Lan$ are far from $\partial \Lan$.
\end{remark}

\paragraph{Interaction of a finite box with the exterior.}
For $\La \Subset \Zd$, and $\bx \in \Conf$, we let $\HH(\Zd \rightarrow \La)(\bx)$ be 
\begin{equation}
\label{finiteH}
\HH(\Zd \rightarrow \La)(\bx) := \sum_{i \in \La, j \in \Zd} J_{i,j} \W(\bx_i, \bx_j),
\end{equation}
which is always bounded by $\Cc \times |\La|$, and we define $\nabla \HH(\Zd \rightarrow \La)(\bx)$ as the following family of tangent vectors, indexed by $\La$: 
\begin{equation}
\label{nablaHZdLa}
\nabla \HH(\Zd \rightarrow \La)(\bx) := \left(\sum_{j \in \Zd} J_{i,j} \partial_1 \W(\bx_i, \bx_j) \right)_{i \in \La}.
\end{equation}
Each component of $\nabla \HH(\Zd \rightarrow \La)$ is bounded, uniformly in $\La$, by a constant depending only the parameters of the model (here $\|J\|_{\ell^1}$ and $\|\partial_1 \W\|_{\mathscr{L}^\infty}$).

\paragraph{Finite-volume free energy.} For $n \geq 1$ and $\Pn \in  \Ppn$, we introduce two quantities $\Ee_n(\Pn)$ and $\Hh_n(\Pn)$.
\begin{itemize}
\item The \emph{relative entropy}: if $\pn$ is the Radon-Nykodym density of $\Pn$ with respect to $\omega_n$, we let:
\begin{equation*}
\Een(\Pn) := \int_{\M^{\Lan}} \log \pn \ \dd \pn.
\end{equation*} 
If $\Pn$ is not absolutely continuous with respect to $\omega_n$, we set $\Een(\Pn) = + \infty$.
Since $\omega_{n}$ is a probability measure, by Jensen's inequality, $\Een(\Pn)$ is always $\geq 0$, and vanishes if and only if $\Pn$ coincides with $\omega_n$.

\item The \emph{average interaction energy}, where $\HLan$ is as in \eqref{def:HLa}:
\begin{equation*}
\Hhn(\Pn) := \Esp_{\Pn} \left[ \HLan(\sig) \right].
\end{equation*}
\end{itemize}
We then define the finite volume free energy $\Fn$ as: 
\begin{equation}
\label{def:Fn}
\Fn(\Pn) := \Een(\Pn) + \beta \Hhn(\Pn).
\end{equation}

\begin{lemma}
\label{finitevolLSC}
Both $\Een$ and $\Hhn$ are lower semi-continuous on $\Ppn$, and thus so is $\Fn$.
\end{lemma}
\begin{proof}
Lower semi-continuity of the relative entropy is a classical fact. Moreover $\HLan$ is continuous and bounded on $M^\La$, thus $\Hhn = \Esp_{\Pn} \left[ \HLan(\sig) \right]$ is in fact continuous on $\Ppn$. 
\end{proof}

\paragraph{Infinite-volume free energy.} Next, when $\P$ is a stationary spin measure, we define, with $\P_{|\Lan}$ denoting the restriction to $\Lan$:
\begin{itemize}
	\item The \emph{specific relative entropy:}
	\begin{equation}
	\label{eq:Eei}
 \Eei(\P) := \lim_{n + \infty}  \frac{1}{|\Lan|}\ \Een(\P_{|\Lan}) = \sup_{n \geq 1}\ \frac{1}{|\Lan|} \ \Een(\P_{|\Lan}).
	\end{equation}
\item The \emph{interaction energy density:}
	\begin{equation}
	\label{eq:Hhi}
	\Hhi(\P) := \lim_{n + \infty} \frac{1}{|\Lan|}\  \Hhn(\P_{|\Lan}) = \lim_{n + \infty} \frac{1}{|\Lan|}\ \Esp_{\P_{|\Lan}} \left[ \HLan(\sig) \right].
	\end{equation}
\end{itemize}
The fact that the limit in \eqref{eq:Eei} exists and is non-decreasing (hence coincides with the $\sup$) is proven e.g. in \cite[Prop~6.75]{FriedliVelenik} (the quantity $s$ there is in fact the \emph{opposite of} our $\Eei$). Similarly, existence of a limit in \eqref{eq:Hhi} is shown in \cite[Prop~6.78]{FriedliVelenik} using a different convention ($\tHLanD$ instead of $\HLan$) for the spin-spin interactions, as mentioned in Remark \ref{rem:tHLaD}, but in view of \eqref{HLaDvstHLaD} this does not matter. Both results are classical and rely on some kind of super-additive behavior and Fekete's lemma.

We then define the infinite-volume energy functional as:
\begin{equation}
\label{def:Fbeta}
\Fbeta(\P) := \Eei(\P) + \beta \Hhi(\P) = \lim_{n \to \infty} \frac{1}{|\Lan|}\ \Fn(\P_{|\Lan})
\end{equation}

\begin{lemma}
\label{lem:PropertiesOfTheLimits}
The following holds:
\begin{enumerate}
  \item Both $\Eei$ and $\Hhi$ are lower semi-continuous on $\Ppis$, and thus so is $\Fbeta$.
  \item The limit $\frac{1}{|\Lan|} \ \Een \to \Eei$ is non-decreasing.
  \item The limit $\frac{1}{|\Lan|} \ \Hhn \to \Hhi$ is uniform.
\end{enumerate}
In particular, there exists a sequence $n \mapsto h(n)$ such that $h(n) \to 0$ as $n \to \infty$ and for all $\P$ in $\Ppis$ we have:
\begin{equation}
\label{error-free-energy}
\Fbeta(\P) \geq \frac{1}{|\Lan|} \ \Fn(\P) + h(n).
\end{equation}
\end{lemma}

\begin{proof}
The first item follows from the last two together with the semi-continuity of the finite-volume functionals stated in Lemma \ref{finitevolLSC}. The fact that $\frac{1}{|\Lan|} \ \Een$ is non-decreasing is also mentioned above. It remains to prove that $\frac{1}{|\Lan|} \ \Hhn $ converges to its limit as $n \to \infty$ uniformly on $\Ppis$. For this, let $n \geq 100$ and let $m \geq 100 n$ be a multiple of $n$. Paving $\La_m$ by disjoint copies of $\La_n$, we write:
\begin{equation*}
\Hh_m(\P) = \frac{|\La_m|}{|\La_n|} \Hhn(\P) + \frac{|\La_m|}{|\La_n|} \times  o(n^\d),
\end{equation*}
where the first term in the right-hand side corresponds to the interactions within each copy, and the second error term corresponds to the interaction of each copy with the rest of $\La_m$, which is $o\left( n^{\d} \right)$ uniformly on $\Ppis$ by the short-range assumption (see Remark \ref{rem:tHLaD}).
We thus have:
\begin{equation*}
\frac{1}{|\La_m|} \Hh_m(\P) = \frac{1}{|\La_n|} \Hhn(\P) + o_n(1).
\end{equation*}
Sending $m \to \infty$ we get $\left|\Hhi(\P) - \frac{1}{|\La_n|} \Hhn(\P)\right| = o_n(1)$ uniformly on $\Ppis$.
\end{proof}

\subsection{Stationary version of a finite-volume measure}
\newcommand{\Lal}{{\La_\ell}}
\label{sec:stationarisation}
Let $n \geq 1$ and let $\P_n$ be a probability measure on $\Confn$. We define the \emph{stationarized version of $\P_n$ with averages in $\Lan$}, denoted by $\Stat_n(\P_n)$, as follows: we pave $\Zd$ by disjoint copies of $\Lan$, namely we identify $\Zd$ with $\bigsqcup_{u \in n\Zd} \tu \cdot \Lan$, and on each copy of $\Lan$ we place an independent sample of $\P_n$, and finally we average the resulting random spin configuration over translations in~$\Lan$. Formally speaking:
	\begin{enumerate}
		\item We take a family $\left(\sig^{(u)}\right)_{u \in n \Zd}$ of i.i.d. $\Confn$-valued random variables with common distribution $\P_n$.
		\item We define a random spin configuration $\bar{\sig}$ by setting (for any given $i \in \Zd$) $\bar{\sig}_i := \sig^{(u)}_{i-u}$ for the unique $u \in \Zd$ such that $i \in \tu \cdot \Lan$.
		\item We define another random spin configuration $\hat{\sig}$ by introducing a uniform random variable $v$ on $\Lan$, independent from $\bar{\sig}$ and letting $\hat{\sig} := \theta_v \cdot \bar{\sig}$. 
    \item We let $\Stat_n(\P_n)$ be the law of $\hat{\sig}$.
	\end{enumerate}
This stationarization serves as a compatibilization procedure, allowing us to work in large but finite boxes and yet still produce a stationary object in the end.
The following lemma shows that the stationarization procedure preserves, at least approximately, some important properties.

\begin{lemma}[Properties of the stationarization]
\label{lem:ppy_stationary}
Let $n \geq 1$, let $\P_n$ and $\Stat_n(\P_n)$ be as above. 
\begin{enumerate}
	\item $\Stat_n(\P_n)$ is indeed a stationary spin measure.
	\item If $f : \Conf \to \R$ is a bounded function which is $\Lal$-local with $\ell < n$, then:
	\begin{equation}
	\label{fLalLanStat}
	\Esp_{\Stat_n(\P_n)}[f] = \frac{1}{|\La_{n-\ell}|} \int_{\La_{n-\ell}} \Esp_{\P_n}[ f \circ \theta_u ] \dd u + \O\left( \frac{\ell}{n}  \right) \|f\|_\infty.
	\end{equation}
	\item The specific relative entropy of $\Stat_n(\P_n)$ satisfies:
	\begin{equation}
	\label{stat_specific}
	\Eei(\Stat_n(\P_n)) \leq \frac{1}{|\Lan|} \ \Een(\P_n).  
	\end{equation}
	\item The energy density of $\Stat_n(\P_n)$ satisfies:
	\begin{equation}
	\label{stat_energy}
	\Hhi(\Stat_n(\P_n)) = \frac{1}{|\Lan|} \ \Hhn(\P_n) + o_n(1),
	\end{equation}
	with an error term that is uniform with respect to $\P_n$.
\end{enumerate}
In particular, there exists a sequence $n \mapsto h(n)$ such that $h(n) \to 0$ as $n \to \infty$ and for all $\P_n$ we have:
\begin{equation}
\label{hnP}
\Fbeta(\Stat_n(\P_n)) \leq \frac{1}{|\Lan|} \ \Fn(\P_n) + h(n).
\end{equation}
\end{lemma}
\begin{proof}[Proof of Lemma \ref{lem:ppy_stationary}]
\newcommand{\hP}{\hat{\P}}
\newcommand{\bP}{\bar{\P}}
Denote by $\bP_n$ the law of the spin configuration $\bar{\sig}$ obtained after the second step of the construction of $\Stat_n(\P_n)$. Also, for convenience, let us write $\hP_n$ instead of $\Stat_n(\P_n)$, which is the law of $\hat{\sig}$.
By definition, in particular the third step of the construction, $\hP_n$ is given by the mixture:
\begin{equation}
\label{hpasmixture}
\hP_n :=  \dashint_{\La_n} \left( \left(\theta_v\right)_\star \bP_n \right) \dd v
\end{equation}
of the push-forward of $\bP_n$ by translations $\theta_v$, for $v$ in $\La_n$.
\medskip

\emph{1. Stationarity.} For all bounded measurable functions $f$ on $\Conf$ and all $u$ in $\Zd$ we have:
\begin{equation*}
\Esp_{\hP_n}[f(\theta_u \cdot \sig)] = \dashint_{\La_n} \Esp_{\bP_n} [ f(\theta_u \cdot (\theta_v \cdot \sig))] \dd v = \dashint_{\La_n} \Esp_{\bP_n} [ f(\theta_{u+v} \cdot \sig)] \dd v = \dashint_{\La_n} \Esp_{\bP_n} [ f(\theta_{v} \cdot \sig)] \dd v,
\end{equation*}
because $v \mapsto u+v \mod \La_n$ is a measure-preserving bijection of $\La_n$, thus $\Esp_{\hP_n}[f(\theta_u \cdot \sig)] = \Esp_{\hP_n}[f(\sig)]$ for all $f$ and $u$, so $\hP_n$ is indeed stationary.

\medskip

\emph{2. Local statistics.} Assume that $f$ is $\Lal$-local. We have:
\begin{equation*}
\Esp_{\hP_n}[f(\sig)] = \dashint_{\La_n} \Esp_{\bP_n} [ f \circ \theta_v] \dd v = \dashint_{\La_{n-\ell}} \Esp_{\bP_n} [ f \circ \theta_v ] \dd v + \O\left( \frac{\ell}{n}  \right) \|f\|_\infty,
\end{equation*}
first by definition \eqref{hpasmixture} and then because averages over $\Lan$ or over $\La_{n-\ell}$ differ by $\O\left( \frac{\ell}{n}  \right)$. If $v \in \La_{n-\ell}$, then $f \circ \theta_v$ remains $\Lan$-local, and thus by  construction $\Esp_{\bP_n} [ f \circ \theta_v ] = \Esp_{\P_n}[ f \circ \theta_v]$, which yields \eqref{fLalLanStat}.

\medskip

\emph{3. Specific entropy.}
Take $m \geq 100 n$. By \eqref{hpasmixture} and convexity of the relative entropy we have: 
\begin{equation*}
\Ee_m [ \hP_n ] \leq  \dashint_{\La_n} \Ee_m [ \left(\theta_v\right)_\star \bP_n ] \dd v.
\end{equation*}
Recall that by construction, the measure $\bP_n$ has a product structure with respect to the decomposition $\Zd = \prod_{u \in n\Zd} \tu \cdot \Lan$, and thus for all fixed $v \in \La_n$, the restriction of $\left(\theta_v\right)_\star \bP_n$ to $\La_m$ has a product structure with respect to $\left(\prod_{u \in n\Zd} \left(\theta_{u+v} \cdot \Lan\right) \cap \La_m  \right)$. By additivity of the relative entropy for product measures:
\begin{equation*}
\Ee_m [ \left(\theta_v\right)_\star \bP_n ] = \sum_{u \in n\Zd, \left(\theta_{u+v} \cdot \Lan \right) \cap \La_m \neq \emptyset} \Ee_{\left(\theta_{u+v} \cdot \Lan\right) \cap \La_m} [\left(\theta_{u+v} \right)_\star \P_n],
\end{equation*}
with a slight abuse of notation.

There are $\frac{|\La_m|}{|\La_n|} - o(|\La_m|)$ indices $u$ in the sum above for which the corresponding copy of $\Lan$ is inside $\La_m$, i.e. $\theta_{u+v} \cdot \Lan \subset \La_m$, and for each of them the entropy contribution is exactly given by 
\begin{equation*}
\Ee_{\left(\theta_{u+v} \cdot \Lan\right) \cap \La_m} [\left(\theta_{u+v} \right)_\star \P_n] = \Ee_{\Lan} [\P_n].
\end{equation*} 
The contribution of the remaining terms, corresponding to the copies of $\Lan$ that intersect the boundary of the large box $\La_m$, is bounded (by $\Ee_{\Lan}(\P)$) independently of $u$ and $m$. We obtain:
\begin{equation*}
\Ee_m [ \left(\theta_v\right)_\star \bP_n ] = \frac{|\La_m|}{|\La_n|} \times \Ee_{\Lan} [\P_n] + o_m(|\La_m|).
\end{equation*}
We thus have $\frac{1}{|\La_m|} \Ee_m [ \hP_n ] \leq \frac{1}{|\La_n|} \Ee_{\Lan} [\P_n] + o_m(1)$, and sending $m \to \infty$ yields \eqref{stat_specific}.

\medskip
\newcommand{\HLaDL}{\mathsf{H}_{\La, L}}
\newcommand{\HLamDL}{\mathsf{H}_{\La_m, L}}
\newcommand{\simv}{\sim_v}

\emph{4. Energy density.} Let $\delta > 0$ be fixed. By the short-range assumption \eqref{eq:short-range} we know that for some $L \geq 1$ large enough, depending only on $\delta$ and $\W$, we have $\sum_{\|i\| > L} |J_{0,i}| \times \|\W\|_\infty \leq \delta$ and thus also, by translation-invariance:
\begin{equation}
\label{EffectTrunc}
\sup_{\bx \in \Conf} \sup_{i \in \Zd} \sum_{j \in \Zd, \|i-j\| > L} \left| J_{i,j} \W(\bx_i, \bx_j) \right| \leq \delta.
\end{equation}
We may thus consider a truncated version of $\HLa$ (see \eqref{def:HLa}) defined by 
\begin{equation*}
\HLaDL : \bx \mapsto \sum_{i \in \La, j \in \La, |i-j| \leq L} J_{i,j} \W(\bx_i,\bx_j),
\end{equation*}
and by \eqref{EffectTrunc} we have $\sup_{\bx \in \Conf} \left|\HLa(\bx) - \HLaDL(\bx)\right| \leq \delta |\La|$. In particular, for $m \geq n$ we can write:
\begin{equation*}
\left| \Esp_{\hP_n} [\HLam] - \Esp_{\hP_n} [\HLamDL] \right| \leq \delta |\La_m|.
\end{equation*}
Next, by construction (see \eqref{hpasmixture}) we have: 
\begin{equation*}
\Esp_{\hP_n}[\HLamDL] = \dashint_{\Lan} \Esp_{\bP_n}[\HLamDL \circ \theta_v ] \dd v.
\end{equation*}
We now fix some $v \in \La_n$. For $i,j$ in $\Zd$ we write $i \simv j$ when there exists $u \in n\Zd$ such that $i, j \in \theta_{u+v} \Lan$, which means that $i$ and $j$ belong to the same copy of $\Lan$ shifted by $v$. We have, for all $\sig \in \Conf$:
\begin{equation*}
\HLamDL(\sig) = \sum_{i \in \La_m, j \in \La_m, |i-j| \leq L} J_{i,j} \W(\bx_i,\bx_j) 
= \sum_{i,j \in \La_m, |i-j| \leq L, i \simv j} J_{i,j} \W(\bx_i,\bx_j) \pm \Cc  \frac{|\La_m|}{|\La_n|} L^\d n^{\d-1},
\end{equation*}
where $\Cc$ depends on the parameters of the model.
Indeed, between the first and the second sum we are throwing away interactions between sites $i$ and $j$ for $i \in \La_m$, $j \in \La_m$, with $|i-j| \leq L$, and $i \nsim_v j$. There are $\O(L^\d n^{\d-1})$ such couples for each translated copy of $\La_n$ and $\frac{|\La_m|}{|\La_n|}(1 + o_m(1))$ copies (to see this, fix such a copy: for $i$ in this copy, since $L$ is fixed but $n$ and $m$ are large, if there exists $j \in \Lambda_{m}$ such that $i \not\sim j$ and $|i-j| \leq L$, then necessarily $i$ is at distance less than $L$ from $\partial \Lambda_{n}$. Thus, there is $\O(n^{\d-1} L)$ possibilities for $i$, and for any fixed $i$ there is $\O(L^{\d})$ indices $j$ at distance $\leq L$). 

Taking the expectation, we obtain:
\begin{multline*}
\Esp_{\bP_n}[\HLamDL \circ \theta_v ] = \Esp_{\bP_n} \sum_{i \in \La_m, j \in \La_m, |i-j| \leq L, i \simv j} J_{i,j} \W(\bx_{i+v},\bx_{j+v}) \pm \Cc   \frac{|\La_m|}{|\La_n|} L^\d n^{\d-1} \\
= \Esp_{\bP_n} \sum_{i \in \theta_v \cdot \La_m , j \in \theta_v \cdot \La_m, |i-j| \leq L, i \sim j} J_{i,j} \W(\bx_{i},\bx_{j})  \pm \Cc  \frac{|\La_m|}{|\La_n|} L^\d n^{\d-1}.
\end{multline*}
We can then write, using \eqref{EffectTrunc} again in the second line:
\begin{multline*}
\Esp_{\bP_n} \sum_{i \in \theta_v \cdot \La_m , j \in \theta_v \cdot \La_m, |i-j| \leq L, i \sim j} J_{i,j} \W(\bx_{i},\bx_{j})  =  \frac{|\La_m|}{|\La_n|}\left(1 + o_m(1)\right) \Esp_{\P_n} \sum_{i \in \La_n , j \in \La_n,  |i-j| \leq L} J_{i,j} \W(\bx_{i},\bx_{j}) \\
= \frac{|\La_m|}{|\La_n|}\left(1 + o_m(1)\right) \left( \Esp_{\P_n} \sum_{i \in \La_n , j \in \La_n} J_{i,j} \W(\bx_{i},\bx_{j}) \pm |\La_n| \delta \right).
\end{multline*}

Combining all those estimates, dividing by $|\La_m|$ and sending $m \to \infty$, we obtain:
\begin{equation*}
\Hhi(\Stat_n(\P_n)) = \frac{1}{|\Lan|} \ \Hhn(\P_n) \pm \Cc  L^\d \times o_n(1) + \delta.
\end{equation*}
Since $\delta$ is arbitrary and $L$ depends only on $\delta$, we do get \eqref{stat_energy} with an error term uniform with respect to the choice of the measure $\P_n$.
\end{proof}

\begin{remark}
If $\P$ is stationary, if $f$ is a bounded function which is $\Lal$-local, we have by \eqref{fLalLanStat}:
	\begin{equation}
	\label{fPnStatPn}
	\Esp_{\Stat_n(\P_{|\La_n})}[f] = \Esp_{\P}[f] + \O\left( \frac{\ell}{n}  \right) \|f\|_\infty.
	\end{equation}
Thus, if $\P$ is stationary, then $\Stat_n(\P_{|\La_n})$ converges to $\P$ in the local topology as $n \to \infty$. The error term is however not zero for finite $n$, and if $\ell = n$, it might be large. In other words, the stationarization with respect to $\Lan$ of a stationary measure $\P$ does not exactly coincide with $\P$ in general.
\end{remark}

\subsection{Wasserstein distance between spin measures}
\label{sec:Wasserstein}
\newcommand{\Cpl}{\mathsf{Cpl}}

\subsubsection*{Wasserstein distance and optimal transport in finite volume} 
As a finite product of Riemannian manifold, $\mathsf{Conf}_{n}$ is a Riemannian manifold and its Riemannian distance $\dn$ satisfies
\begin{equation*}
  \dn^{2}(\bx, \by) = \sum_{i \in \Lambda_{n}} \di^{2}(\bx_{i}, \by_{i}), \qquad \bx,\by \in \mathsf{Conf},
\end{equation*}
where we recall that $\di$ stands for the distance on $\M$.
	
Next, if $\Pz_n, \Pu_n$ are two spin measures in $\Ppn$, we define:
	\begin{equation}
	\label{def:Wn}
\Wn^2(\Pz_n, \Pu_n) := \inf_{\Pi \in \Cpl(\Pz_n, \Pu_n)} \iint_{\Confn \times \Confn} \dn^2(\sigz, \sigu)\  \dd \Pi(\sigz, \sigu),
	\end{equation}
	where the infimum is taken over all possible couplings $\Pi$ of $(\Pz_n, \Pu_n)$. $\Wn$ is the Wasserstein distance on $\Ppn = \Pp\left(M^{\Lan}\right)$ associated to the distance square cost function.
  We refer to \cite[Sec.~6]{VillaniOldNew} for elementary properties. 
  The quantity $\mathcal{W}_{n}$ defines a complete geodesic distance on $\Ppn$. By compactness of $\M$, both $\dn^{2}$ and $\Wn^{2}$ are bounded from above by $\operatorname{diam}(M)^{2} \cdot |\Lambda_{n}|$.

\subparagraph{Some reminders.} The study of the Monge--Kantorovich minimization problem appearing in \eqref{def:Wn}, namely understanding the optimal coupling $\Pi$ and its associated cost, belongs to the theory of optimal transportation of measures --- here on a smooth compact manifold. 

Let $\mathrm{P}_n$ and $\mathrm{Q}_n \in \Ppn$, and assume that $\mathrm{P}_n$ is absolutely continuous with respect to the measure $\omega_{n}$. Then by the Brenier--McCann theorem \cite[Thm.\ 9]{McCannPolar}, there exists a Lipschitz map $\theta \colon \mathsf{Conf}_{n} \to \mathbb{R}$, called the Kantorovitch potential, such that:
\begin{itemize}
  \item the map $\bx \mapsto T(\bx) := \exp_{\bx} \paren*{-\nabla \theta(\bx)}$ pushes forward $\mathrm{P}_n$ to $\mathrm{Q}_n$;
  \item $\theta$ is $\frac{1}{2} \dn^{2}$-concave in the sense that $\theta^{\mathsf{c} \mathsf{c}} = \theta$, where
    \begin{equation*}
      h^{\mathsf{c}}(\by) := \Inf*{ \frac{\dn^{2}(\bx,\by)}{2} - h(\bx) : \bx \in \mathsf{Conf}_{n}};
    \end{equation*}
  \item the unique $\mathcal{W}_{n}$-geodesic between $\mathrm{P}_n$ and $\mathrm{Q}_n$ is given by
    \begin{equation*}
      [0,1] \ni t \mapsto \exp_{\bx}\left(- t \nabla \theta(\bx)\right)_{\sharp} \mathrm{P}_n.
    \end{equation*}
  \item for $\mathrm{P}_n$-almost every $\bx$, there exists a unique Riemannian geodesic between $\bx$ and $T(\bx)$ \cite[Thm.~4.2]{cordero2001riemannian}. In other words, the optimal transportation map almost surely avoids the cut locus.
\end{itemize}

\begin{remark}
The solution of Monge's problem, i.e. the construction of an optimal transport \emph{map}, is due to McCann \cite{McCannPolar} in the case of a manifold. The optimal map also solves Kantorovich's formulation, i.e. it provides an optimal transportation \emph{plan} --- this is made clear e.g. in \cite[Thm 1.10]{Gigli_2011}.
\end{remark}

\subsubsection*{Wasserstein distance in infinite-volume}
\newcommand{\Pzn}{\Pz_{|n}}
\newcommand{\Pun}{\Pu_{|n}}
\newcommand{\Q}{\mathrm{Q}}
For $\P$ and $\Q \in \mathscr{P}(\mathsf{Conf})$, we simply write $\Wn(\P, \Q)$ for the Wasserstein distance between the restrictions of $\mathrm{P}, \mathrm{Q}$ to $\Lan$. 
 
\begin{lemma}
\label{def:WassInfinite}
If $\P, \Q$ are stationary spin measures, then the map $n \mapsto \Wn^2\left(\P, \Q \right)$ is non-decreasing, the limit $\lim_{n \geq 1} \frac{1}{|\Lan|} \Wn^2\left(\P, \Q\right)$ exists and we have:
\begin{equation}
\label{def:Wi}
	\Wi^2(\P, \Q) := \lim_{n \geq 1} \frac{1}{|\Lan|} \Wn^2\left(\P, \Q\right) = \sup_{n \geq 1} \frac{1}{|\Lan|} \Wn^2\left(\P, \Q\right).
\end{equation}
\end{lemma}
\begin{proof}
The fact the limit exists and is given by a supremum follows from a standard sub-additive argument found e.g. in \cite[Lem.\ 15.11]{GeorgiiGibbs}, with a different sign convention there.
Let us check that the assumptions of this lemma are satisfied.

For convenience, instead of working on a box $\Lambda_{n}$, we shall work on arbitrary finite subsets of $\Zd$, with a slight abuse of notation.
Take $\Lambda$ and $\Lambda' \Subset \Zd$, disjoint.
First of all, by stationarity, for all $u \in \Zd$, we have
\begin{equation*}
  \mathcal{W}_{\theta_{u} \Lambda}(\mathrm{P}, \mathrm{Q}) = \mathcal{W}_{\Lambda}(\theta_{-u}\mathrm{P}, \theta_{-u}\mathrm{Q}) = \mathcal{W}_{\Lambda}(\mathrm{P}, \mathrm{Q}).
\end{equation*}
Next, by the existence of an optimal coupling there exist random spin configurations $\bx,\by$ of law $\P$, $\Q$ on $\Lambda \cup \Lambda'$ such that:
\begin{equation*}
  \mathcal{W}_{\Lambda \cup \Lambda'}^{2}(\mathrm{P}, \mathrm{Q}) = \Esp\left[\di^{2}_{\Lambda \cup \Lambda'}(\bx, \by) \right] \geq \Esp \left[ \di^{2}_{\Lambda}(\bx_{\Lambda}, \by_{\Lambda}) \right] + \Esp\left[ \di^{2}_{\Lambda'}(\bx_{\Lambda'}, \by_{\Lambda'}) \right] \geq \mathcal{W}_{\Lambda}^{2}(\mathrm{P}, \mathrm{Q}) + \mathcal{W}_{\Lambda'}^{2}(\mathrm{P}, \mathrm{Q}).
\end{equation*}
The first inequality holds because, using that $\Lambda \cap \Lambda' = \emptyset$, we have
  $\di_{\Lambda \cup \Lambda'}^{2} = \di^{2}_{\Lambda} + \di^{2}_{\Lambda'}$. The second inequality is by definition of the Wasserstein distances.
  Thus, the assumptions from \cite[Lem.~15.11]{GeorgiiGibbs} are satisfied, which completes the proof.
\end{proof}

Since for each $n$ the distance $\Wn$ satisfies the triangle inequality, so does $\Wi$.
Moreover, since it is obtained a supremum, $\Wi$ controls all the finite-volume Wasserstein distances.
\begin{corollary}
$\Wi$ defines a distance on $\Ppis$, which we call the specific Wasserstein distance.
\end{corollary}

\begin{remark}
\label{rem:difference_top}
  The specific Wasserstein induces on $\mathscr{P}^{\mathrm{s}}$ a topology strictly finer than the weak topology, or than the local topology.
  For instance, take $\d = 1$, and $M$ to be the unit circle.
  Let $\bx_{n}$ the random variable in $\mathsf{Conf}_{2n}$ with independent spins distributed as follows: the spins at site $\{-n, \dots, -1\}$ are sampled uniformly on the upper semi-circle $\{ \mathrm{e}^{\mathrm{i} \theta} : \theta \in [0,\pi]\}$, while the spins at $\{0,\dots,n\}$ are uniform on the lower semi-circle.

  Write $\mathrm{Q}^{n}$ for the law of $\bx_{n}$, and let $\mathrm{P}^{n} = \mathsf{Stat}_{2n}(\mathrm{Q}^{n})$.
  For the local topology, this sequence of stationary measure converges to the mixture $\mathrm{P}^{\pm} = \frac{1}{2} (\mathrm{P}^{+} + \mathrm{P}^{-})$, where under $\mathrm{P}^{+}$ or $\mathrm{P}^{-}$ all the spins are distributed independently and uniformly on the upper or lower semi-circle.
  However, $\mathrm{P}^{n}$ remains at positive specific Wasserstein distance of both $\mathrm{P}^{+}$ and $\mathrm{P}^{-}$, and thus of $\mathrm{P}^{\pm}$ because $\P^+, \P^-$ have essentially disjoint supports.
\end{remark}

\begin{remark}
\label{Stat_Conv_Pas}
  If $\mathrm{P} \in \mathscr{P}_{n}$ and $\mathrm{Q} \in \mathscr{P}^{\mathrm{s}}$, then in general $\mathcal{W}(\mathsf{Stat}_{n} \mathrm{P}, \mathrm{Q})$ is not close to ${|\Lambda_{n}|}^{-1} \mathcal{W}_{n}(\mathrm{P}, \mathrm{Q}_{n})$. This can be seen by taking $\mathrm{Q} := \mathrm{P}^{\pm}$ from Remark \ref{rem:difference_top} and $\mathrm{P} := \mathrm{Q}_{n}$.
\end{remark}

\subsection{Convexity properties of the functionals}
\paragraph{Usual convexity.}
First, we endow $\Ppn$ with its usual linear structure, and we recall some elementary convexity facts.
\begin{lemma}[Convexity properties - linear interpolation]
\label{lem:convexity_usual}
For all $n \geq 1$:
\begin{enumerate}
	\item $\Een$ is strictly convex on $\Ppn$.
	\item $\Hhn$ is linear (hence convex, but not strictly convex) on $\Ppn$.
	\item For all $\P_n \in \Ppn$, the map $\Wn^2(\P_n, \cdot)$ is convex on $\Ppn$.
\end{enumerate}
\end{lemma}
\begin{proof}[Proof of Lemma \ref{lem:convexity_usual}]
The first statement is classical \cite[Prop. B.66]{FriedliVelenik}.
The second one is straightforward.
The third one is also well-known \cite[Thm.~4.8]{VillaniOldNew}.
\end{proof}

\begin{corollary}
\label{cor:convexity}
For all $n \geq 1$, $\P$ in $\Ppn$ and $h > 0$, the functional defined on $\Ppn$ by:
\begin{equation*}
\Q \mapsto \frac{1}{2} \Wn^2(\P, \Q) + h \Fn(\Q)
\end{equation*}
is strictly convex, and thus has a unique minimizer.
\end{corollary}
\begin{proof}
  The strict convexity follows from Lemma \ref{lem:convexity_usual}, hence there is at most one minimizer. Existence of a minimizer follows from: \begin{enumerate*}[(i)]
    \item compactness of $\mathscr{P}_{n}$, since $M$ itself is compact;
    \item lower semi-continuity of both $\mathcal{W}_{n}$ \cite[Cor.~6.11]{VillaniOldNew} and $\Fn$ (Lemma \ref{finitevolLSC}).
  \end{enumerate*}
\end{proof}

\paragraph{Displacement convexity.}
On the space of probability measures, besides the usual convex structure, the work of McCann is the first to put forward the existence of a different notion of interpolation, named \emph{displacement interpolation}.
We refer to \cite{McCann_1997} for the Euclidean case, while the case of manifolds is treated in \cite{cordero2001riemannian,CMCS06}, see also \cite[Chaps.~16 \& 17]{VillaniOldNew}.

A functional $\mathcal{G} \colon \mathscr{P}_{n} \to [0,\infty]$ is said to be \emph{$\lambda$-displacement convex} for some $\lambda \in \mathbb{R}$ when for every $\mathrm{P}^{0}$ and $\mathrm{P}^{1}$, and for every Wasserstein geodesic $(\mathrm{P}_{t})_{t \in [0,1]}$ joining $\mathrm{P}^{0}$ to $\mathrm{P}^{1}$, we have for $t \in [0,1]$:
\begin{equation*}
  \mathcal{G}(\mathrm{P}_{t}) \leq (1-t) \mathcal{G}(\mathrm{P}^{0}) + t \mathcal{G}(\mathrm{P}^{1}) - \frac{\lambda}{2} t (1-t) \mathcal{W}_{n}^{2}(\mathrm{P}^{0}, \mathrm{P}^{1}).
\end{equation*}


For our functionals, we have the following properties regarding displacement interpolation.
\begin{lemma}[Convexity properties - displacement interpolation] 
\label{lem:dis_conv_finite}
For all $n \geq 1$:
\begin{enumerate}
 	\item Let $\kappa$ be Ricci curvature lower bound from by \eqref{PosCurv}, then $\Een$ is $\kappa$-displacement convex on $\Ppn$.
	\item $\Hhn$ is $- 2 \|J\|_{\ell^1} \times \|\nabla^2 \W\|_\infty$-displacement convex on $\Ppn$. 
\end{enumerate}
\end{lemma}
\begin{proof}
We recall that Ricci lower bounds are stable under taking products of manifolds. For the first item, see \cite[Thm.~1.1]{vonRenesseSturm} (when $\V \coloneq 0$) or \cite[Thm.~4.9]{SturmActa} in the general case. This is an instance of the celebrated link between curvature and displacement convexity.
See also \cite[Chap.~17]{VillaniOldNew}.

For the second item, it is well-know that interaction terms like $\Hhn$ inherit displacement convexity from the  convexity of the interaction potential $\W$, which is $- 2 \|\nabla^2 \W\|_\infty$-geodesically convex on $\M \times \M$, see for instance \cite[Thm.~15.19]{ABSOptimal} for explicit computations. The factor $2$ comes from the fact that $\W$ depends on two variables.
\end{proof}


\paragraph{Application: uniqueness at high temperature, positive curvature case.}
It is clear that $\Hhi(\P)$ (see \eqref{eq:Hhi}) is linear in $\P$.
It is also a fact \cite[Prop.~6.75]{FriedliVelenik} that the specific relative entropy $\Ee$ is \emph{affine} on the space of stationary spin measures, and thus loses the strict convexity that the relative entropy enjoys in finite volume.
Hence, we see the importance of using displacement interpolations to recover some form of strict convexity.

Compared to \cite{erbar2023optimal}, we do not develop here a full-fletched notion of optimal transportation and displacement interpolation convexity at the infinite-volume level. Had one done so, one would presumably find that $\Eei$ is $\kappa$-displacement convex under \eqref{PosCurv} and that $\Hhi$ is $- 2 \|\nabla^2 \W\|_\infty$-displacement convex. Using finite-volume approximations, we are still able to prove the following result.

\begin{theorem}
\label{theo:uniqueness}
With $\kappa$ as in \eqref{PosCurv}, if $\kappa > 0$ and $\beta < \beta_c := \hal \kappa \left(\|J\|_{\ell^1}  \|\nabla^2 \W\|_\infty\right)^{-1}$ then the free energy $\fbeta$ has a unique minimizer.
\end{theorem}

\newcommand{\Phal}{\P^{\hal}}
\begin{proof}[Proof of Theorem \ref{theo:uniqueness}]
Argue by contradiction and assume that $\Pz, \Pu$ are two distinct minimizers of $\fbeta$. Let $\epsilon := \kappa - 2 \beta \|J\|_{\ell^1} \|\nabla^2 \W\|_\infty > 0$ and let $\delta := \Wi^2(\Pz, \Pu) > 0$. We choose $n$ large enough such that:
\begin{enumerate}
	\item $\frac{1}{|\Lan|} \Wn^2(\Pz, \Pu) > \frac{\delta}{2}$ (possible by definition of $\Wi$, see \eqref{def:Wi}).
	\item $\frac{1}{|\Lan|} \Fn(\Pz) \leq \Fbeta(\Pz) + \epsilon \frac{\delta}{100}$ and $\frac{1}{|\Lan|} \Fn(\Pu) \leq \Fbeta(\Pu) + \epsilon \frac{\delta}{100}$ (possible by definition of $\Fbeta$, see \eqref{def:Fbeta}).
	\item The error term in \eqref{stat_energy} is smaller than $\epsilon \frac{\delta}{100}$.
\end{enumerate}

Consider the restrictions of $\Pz, \Pu$ to $\Lan$, and let $\Phal_n$ be their midpoint in the sense of optimal transport. The displacement convexity statements of Lemma \ref{lem:dis_conv_finite} imply that:
\begin{equation*}
\Fn(\Phal_n) \leq \hal \left( \Fn(\Pz) + \Fn(\Pu) \right) - \frac{1}{8} \left( \kappa - 2 \beta \|\nabla^2 \W\|_\infty \right) \Wn^2(\Pz, \Pu) \\
\leq |\Lan| \times \left(\min \Fbeta + \epsilon \frac{\delta}{100} - \epsilon \frac{\delta}{16}\right).
\end{equation*}
Finally, consider the spin measure $\P := \Stat_n[\Phal_n]$. By Lemma \ref{lem:ppy_stationary} and our choice of $n$ we know that:
\begin{equation*}
\Fbeta(\P) \leq \frac{1}{|\Lan|}\Fn(\Phal_n) + \epsilon \frac{\delta}{100} \leq \min \Fbeta + 2 \epsilon \frac{\delta}{100} - \epsilon \frac{\delta}{16}  < \min \Fbeta, 
\end{equation*}
but $\P$ is stationary, which yields a contradiction.
\end{proof}

\begin{remark}[Link with DLR equations and Dobrushin's uniqueness criterion]
The celebrated Dobrushin's uniqueness criterion \cite{dobruschin1968description} can be used to prove uniqueness of Gibbs states for $\beta$ small enough, see e.g. a presentation in \cite[Sec. 6.5.]{FriedliVelenik}. This is very general and does not rely on curvature assumptions. By the Gibbs variational principle, infinite-volume Gibbs states are exactly minimizers of $\Fbeta$ on $\Ppis$, and thus uniqueness at high temperature holds in a more general context. Our result provides a different point of view, which might be useful in situations for which Dobrushin's criterion is more difficult to state, for instance for point processes.
\end{remark}

\section{Fokker--Planck--Kolmogorov equations in infinite-volume}
\label{sec:IFP}
\subsection{Preliminaries}
\subsubsection*{Short reminder on the finite-dimensional case}
Let $\Phi$ be a scalar field on $\M$ with sufficient regularity --- it would also be possible to work on $\mathbb{R}^{\nn}$ by imposing some decay conditions on $\Phi$. We consider the following stochastic differential equation (SDE):
\begin{equation}
\label{MLangevin}
  \mathtt{d} X_{t} = \sqrt{2} \mathtt{d} B_{t} - \nabla \Phi(X_{t}) \mathtt{d} t.
\end{equation}
Solving this SDE from an initial condition $X_{0}$ of law $\mu_{0}$ gives rise to a semi-group of measures $(\mu(t))_{t \geq 0}$ where $\mu_t$ is the law of $X_{t}$. By Itō's formula, we see that the curve $(\mu(t))_{t \geq 0}$ solves the \emph{Kolmogorov equation}, namely for all $f$ in $\CC^{1,2}((0,+\infty), M)$ and for all $t > 0$ we have:
\begin{equation*}
  \frac{\mathtt{d}}{\mathtt{d}t} \mathbb{E}_{\mu(t)}[f] = \mathbb{E}_{\mu(t)}\bracket*{ \partial_{t} f + \Delta f - \nabla \Phi \cdot \nabla f}.
\end{equation*}

Provided that $\mu(t)$ admits a density $\rho(t)$ with respect to the volume measure on $M$, we expect that the curve of probability densities $(\rho(t))_{t \geq 0}$ solves the \emph{Fokker--Planck} equation:
\begin{equation*}
  \partial_{t} \rho = \Delta \rho + \operatorname{div}(\rho \nabla \Phi),
\end{equation*}
in a weak or strong sense depending on the regularity of $\rho$. It is important to keep in mind the slight difference between the Kolmogorov equation --- an equation on measures --- and the Fokker--Planck equation --- describing the evolution of their hypothetical densities.

Under reasonable conditions on $\Phi$, solutions of the Kolmogorov equations have smooth densities and are unique, we refer to \cite[Chap. 6]{Bogachev_2015} for an overview of relevant results. There are different frameworks to study such equations:
\begin{itemize}
 \item Existence of solutions can be obtained by abstract arguments, or by explicit constructions such as proving existence of the underlying SDE \eqref{MLangevin}, or by building a gradient flow \cite{Jordan_1998}.
  \item Regularity of solutions to the Fokker--Planck equation follows from a bootstrap argument, as in \cite{Jordan_1998}. An \emph{ad hoc} argument also yields uniqueness.
  \item An easy instance of Hörmander's condition \cite{Hormander} ensures that, if $\Phi$ is smooth, then solutions to the Fokker--Planck equation are also smooth.
  \item Malliavin calculus \cite{Malliavin} can be used to derive regularity through the study of the associated stochastic differential equation.
 \end{itemize} 

\subsubsection*{Infinite-dimensional case: lack of global density and non-locality}
\newcommand{\Om}{\mathrm{\Omega}}
The Fokker--Planck--Kolmogorov equation in infinite-volume presents two major differences:
\begin{enumerate}
	\item In general, the spin measures $\P$ that we work with are \emph{not absolutely continuous} with respect to some common reference spin measure, e.g. the countable product measure $\Om := \omega^{\otimes \Zd}$, which means that we cannot consider a global density $\p \overset{??}{:=} \frac{\dd \P}{\dd \Om}$ and write down an equation that would govern the time evolution of $\p$.

	Nonetheless, we show below that, in our setting, the finite-dimensional marginal of $\P$ in every $\La \Subset \Zd$ is absolutely continuous with respect to the reference measure $\oL$. We may thus consider its density $\p_\La$ and try to write an equation for the evolution $t \mapsto \p_\La(t)$ \emph{for all finite $\La \Subset \Zd$}.

	\item However, for any given finite box $\La \Subset \Zd$, spins in $\La$ interact with spins outside $\La$, and thus the equation describing the time evolution of the spin measure in $\La$ is not closed: it contains a term related to the time evolution of the system \emph{outside} $\La$. 

	In particular, this means that even when looking at the evolution within a finite box, one cannot apply the general regularity results because the corresponding drift $\Phi$ depends on the whole $\P(t)$, and so does its regularity.

	Another consequence is that, to quote \cite[Chap. 10]{Bogachev_2015} \enquote{the problem of uniqueness of solutions to (...) parabolic Fokker--Planck--Kolmogorov equations in infinite-dimensional spaces is much more complicated than in finite-dimensional ones}.
\end{enumerate}
For those reasons, instead of a single equation, the infinite-volume Fokker--Planck equation consists in a family of equations (a hierarchy) coupling the densities $\p_\La$ for all $\La \Subset \Zd$, and we need to prove both regularity and uniqueness by hand.

\subsection{Formulations of infinite-volume Fokker--Planck--Kolmogorov equations}
\label{sec:formFP}

For $\La \Subset \Zd$, denote by $\CC^\infty_c([0, + \infty), M^\La)$ the set of functions from $\R_{+} \times M^\La$ to $\R$ which are smooth, and compactly supported with respect to the first variable (recall that $M$ is compact). In what follows, unless specified otherwise, objects depending on the time variable $t$ are defined on $[0, + \infty)$.


\paragraph{Some notations.} Recall the notation $\nabla \HH(\Zd \to \Lambda)$ introduced in \eqref{nablaHZdLa}. For $\mathrm{P} \in \mathscr{P}(\mathsf{Conf})$ and $F \in \mathscr{L}^{1}(\mathrm{P})$, let us abbreviate  $\mathbb{E}_{\mathrm{P}}[F \given \Lambda]$ for the conditional expectation, under $\P$, with respect to the projection on $\mathsf{Conf}_{\Lambda}$. We also consider the conditional expectation of vector fields by taking conditional expectations coordinate-wise.
In particular, for $\bx \in \mathsf{Conf}$, we write:
\begin{equation}
\label{espCondnablaHZdLa}
  \mathbb{E}_{\mathrm{P}} \bracket*{ \nabla \HH(\Zd \to \Lambda) \given \Lambda }(\bx) =  \paren*{ \sum_{j \in \Zd} J_{i,j} \mathbb{E}_{\mathrm{P}} \bracket*{ \partial_{1} \W(\bx_{i}, \bx_{j}) \given \Lambda } }_{i \in \Lambda}.
\end{equation}
When $i$ and $j$ are both in $\La$, then $\mathbb{E}_{\mathrm{P}} \bracket*{ \partial_{1} \W(\bx_{i}, \bx_{j}) \given \Lambda } = \partial_{1} \W(\bx_{i}, \bx_{j})$, which is $\CC^2$ with respect to both variables.
However, when $j \not \in \Lambda$, the conditional expectation takes a more complicated form, and then $\mathbb{E}_{\mathrm{P}} \bracket*{ \nabla \HH(\Zd \to \Lambda) \given \Lambda}$ could fail to be merely differentiable even for a smooth $\W$.

In the rest of the section, we often use quantities of the form $\nabla f \cdot \nabla \HH(\Zd \rightarrow \La)$, where $\La \Subset \Zd$ and $f$ is a smooth $\La$-local function. This is to be understood as:
\begin{equation*}
\nabla f \cdot \nabla \HH(\Zd \rightarrow \La)(\bx) = \sum_{i \in \Lambda} \partial_i f(\bx)  \cdot \left(\sum_{j \in \Zd} J_{i,j} \partial_1 \W(\bx_i, \bx_j) \right),
\end{equation*}
where each summand corresponds a scalar product between two tangent vectors in $\TT_{\bx_i} M$.

\subsubsection*{Dual formulation: the Kolmogorov equation for measures}
\newcommand{\DeltaU}{\Delta_\V}
We say that a measurable curve $t \mapsto \P(t)$ of spin measures satisfies the infinite-volume Fokker--Planck--Kolmogorov equation in the \emph{dual sense} when for all $\La \Subset \Zd$ and for all $f$ in $\CC^\infty_c([0, + \infty), M^\La)$:
\begin{equation}
\tag{Dual}
\label{def:VWEAK}
- \Esp_{\P(0)} [f(0, \cdot)] = \int_{0}^{\infty} \Esp_{\P(t)}\left[\partial_t f(t, \cdot) + \DeltaU f(t, \cdot) - \nabla f(t, \cdot) \cdot \nabla \HH(\Zd \rightarrow \La)   \right] \dd t.
\end{equation}
We use the following consequence of \eqref{def:VWEAK}: for all $0 \leq t_0 < t_1$ we have:
\begin{equation}
\label{conseqVWeak}
\Esp_{\P(t_1)} [f(t_1, \cdot)] - \Esp_{\P(t_0)} [f(t_0, \cdot)] \\
= \int_{t_0}^{t_1} \Esp_{\P(t)}\left[\partial_t f(t, \cdot) + \DeltaU f(t, \cdot) - \nabla f(t, \cdot) \cdot \nabla \HH(\Zd \rightarrow \La)  \right] \dd t.
\end{equation}

\subsubsection*{Weak formulation, for local densities}
Let $t \mapsto \P(t)$ be a measurable curve in the space of infinite spin measures, and assume that for all $\La \Subset \Zd$, the restriction of $\P$ to $\La$ admits a density $\p_\La$ with respect to the reference measure $\oL$ on $M^\La$. 

We say that $t \mapsto \P(t)$ \emph{is a weak solution to the infinite-volume Fokker--Planck equation} when for all $\La \Subset \Zd$ and for all test function $f \in \CC^{\infty}_c([0, + \infty), M^\La)$, the following identity holds (cf \eqref{def:VWEAK}):
\begin{multline}
\label{WeakFP}
\tag{Weak}
 -\int_{M^\La} \p_\La(0, \bx) f(0, \bx) \dd \oL(\bx) \\
= \int_{0}^{+\infty} \int_{M^\La} \p_\La(t, \bx) \left(  \partial_t f(t, \bx) + \DeltaU f(t, \bx) -  \nabla f(t, \bx) \cdot \Esp_{\P(t)} \left[ \nabla \HH(\Zd \rightarrow \La) | \La \right](\bx) \right) \dd \oL(\bx) \dd t.
\end{multline}
We use the following consequence of \eqref{WeakFP}: for all $0 \leq t_0 < t_1$, we have (cf \eqref{conseqVWeak}):
\begin{multline}
\label{WeakFP2}
\int_{M^\La} \p_\La(t_1, \bx) f(t_1, \bx) \dd \oL(\bx) - \int_{M^\La} \p_\La(t_0, \bx) f(t_0, \bx)  \dd \oL(\bx)  \\
= \int_{t_0}^{t_1} \int_{M^\La} \p_\La(t, \bx) \left(  \partial_t f(t, \bx) + \DeltaU f(t, \bx) \right) \dd \oL(\bx) \dd t  \\
-  \int_{t_0}^{t_1} \int_{M^\La} \p_\La(t, \bx) \nabla f(t, \bx) \cdot \Esp_{\P(t)} \left[ \nabla \HH(\Zd \rightarrow \La) | \La \right](\bx) \dd \oL(\bx) \dd t.
\end{multline}

\subsubsection*{Strong formulation, for local densities}
Let $t \mapsto \P(t)$  be a measurable curve in the space of infinite spin measures. Assume that for all $\La \Subset \Zd$, $\P$ admits a density $\p_\La$ (with respect to the reference measure $\oL$ on $M^\La$, see \eqref{def:oL}) which is of class $\CC^1$ with respect to the time variable and of class $\CC^2$ with respect to the spatial variable on $(0, + \infty) \times M^\La$. 

We say that $t \mapsto \P(t)$ \emph{is a strong solution to the infinite-volume Fokker--Planck equation} with initial condition $\P(0)$ when for all $\La \Subset \Zd$ and for all $t \in (0, + \infty)$:
\begin{equation}
\label{def:STRONG}
\tag{Strong}
\partial_t \p_\La = \DeltaU \p_\La(t) + \div_{\V} \left(\p_\La(t) \cdot \Esp_{\P(t)} \left[ \nabla \HH(\Zd \rightarrow \La) | \La \right] \right),
\end{equation}
and $\P(t) \to \P(0)$ in the local topology as $t \to 0^+$.

\begin{remark}
  As mentioned above the vector field $\Esp_{\P(t)} \bracket*{ \nabla \HH(\Zd \to \Lambda) \given \Lambda }$ is a priori not $\CC^{1}$. It is implicitly part of the definition of \eqref{def:STRONG} that the divergence of this vector field is well-defined at all times.
\end{remark}

\subsection{The main regularity result}
\newcommand{\Cud}{\CC^{1,2}}
For $\La \Subset \Zd$, we denote by $\Cud((0, +\infty) \times M^\La)$ the space of functions $\rho$ which are of class $\CC^1$ with respect to the time variable and of class $\CC^2$ with respect to the spatial variable, and such that for all $0 < \delta < T$:
\begin{equation}
\label{C0C2borneStrong}
\sup_{t \in [\delta, T]} \left(\|\rho(t, \cdot)\|_\infty +  \|\nabla^2 \rho(t, \cdot)\|_\infty + \|\partial_t \rho(t, \cdot)\|_\infty \right) < +\infty.
\end{equation}

\paragraph{The main regularity result.}
The following theorem, which is one of the main result of this paper, states that all solutions to our Fokker--Planck--Kolmogorov are in fact \emph{strong} solutions.
\begin{theorem}[Regularity of the solutions] \
\label{theo:regul}
\begin{itemize}
  \item Every solution of \eqref{def:VWEAK} has densities with respect to $\oL$ for all $\La \Subset \Zd$ which are solutions of \eqref{WeakFP}.
  \item Every solution of \eqref{WeakFP} has local densities in $\Cud((0, +\infty) \times M^\La)$ which are solutions of \eqref{def:STRONG}.
\end{itemize}
\end{theorem}
We postpone the proof of Theorem \ref{theo:regul} to Section \ref{s:proof-regularity}. The proof uses a bootstrap argument which is similar in spirit to the one of \cite{Jordan_1998} (see also \cite{mei2018mean}), with several important modifications.

\begin{remark}
If we were to choose the spin-spin interaction potential $\W \in \mathscr{C}^{\infty}(M \times M)$, our method would  yield that the local densities $\mathrm{p}_{\Lambda}$ are space-time smooth.
Note that under this assumption, one could also use Malliavin calculus methods to show that the densities are smooth in space, see \cite{HolleyStroockDiffusionTorus} for such computations on the circle when the interactions have finite range.
\end{remark}

\paragraph{Finiteness of the Fisher information.}
Recall the following elementary result: if $g : \R^\nn \to \R$ is a $\CC^2_c$ map such that $g \geq 0$ everywhere on $\R^\nn$, then we have the pointwise bound:
\begin{equation*}
  |\nabla g(x)|^{2} \leq 2 g(x) \|\nabla^2 g\|_\infty.
\end{equation*}
Indeed, we have by a Taylor's expansion, with $u$ a unit vector and $\varepsilon > 0$
\begin{equation*}
0 \leq g(x+\varepsilon u) \leq g(x) + \varepsilon \abs{\nabla g(x)} + \hal \varepsilon^{2} \|\nabla^2 g\|_\infty.
\end{equation*}
Non-positivity of the discriminant of this degree $2$ equation in $\varepsilon$ is exactly the result.
In particular, the quantity $\frac{\|\nabla g\|^2(x)}{g(x)}$ is always well-defined and bounded.

By a similar argument, we see that if $\rho$ is a non-negative function in $\Cud((0, +\infty) \times M^\La)$, \eqref{C0C2borneStrong} holds, and then for all $0 < \delta < T$:
\begin{equation}
\label{FiniteFisher1}
\sup_{t \in [\delta, T]} \sup_{\bx \in M^\La} \frac{\|\nabla \rho(t, \bx)\|^2}{\rho(t, \bx)} < + \infty.
\end{equation}
Hence, the Fisher information of $\rho$ is bounded locally uniformly in time, namely we have:
\begin{equation}
\label{FiniteFisher2}
\sup_{t \in [\delta, T]} \int_{M^{\Lambda}} \|\nabla \log \rho(t, \bx)\|^2 \rho(t,\bx)  \dd \oL (\bx) = \sup_{t \in [\delta, T]} \int_{M^\La} \frac{\|\nabla \rho(t, \bx)\|^2}{\rho(t, \bx)} \dd \omega_\La(\bx) \dd t < + \infty.
\end{equation}
In particular, we obtain:
\begin{corollary}[Finiteness of Fisher information]
\label{th:fisher}
  Let $\P$ be a spin measure whose restriction to $\Lambda \Subset \Zd$ has a density $\mathrm{p}_{\Lambda}$ in $\Cud((0, +\infty) \times M^\La)$. Then for all $0 < \delta < T < \infty$ we have:
  \begin{equation}
  \label{FiniteFisher}
    \sup_{t \in [\delta, T]} \int_{M^{\Lambda}} \frac{\|\nabla \mathrm{p}_{\Lambda}(t, \bx)\|^2}{\mathrm{p}_{\Lambda}(t,\bx)} \dd \oL (\bx) < + \infty.
  \end{equation}
\end{corollary}

\subsection{Evolution Variational Inequality and uniqueness}

\subsubsection*{Reminders on EVI-gradient flows}
\label{sec:evi}
We work here on the metric space $(\mathscr{P}^{\mathrm{s}}, \mathcal{W})$ of stationary spin measures endowed with the specific Wasserstein distance intoduced in Section \ref{sec:Wasserstein}, and with the free energy functional $\fbeta\colon \mathscr{P}^{\mathrm{s}} \to [0,\infty]$, which is lower semi-continuous by Lemma \ref{lem:PropertiesOfTheLimits}. 

Recall that a curve $(\mathrm{P}(t))_{t \in I}$ defined on an interval $I \subset \mathbb{R}$ and with values in $\mathscr{P}^{\mathrm{s}}$ is said to be \emph{absolutely continuous} provided there exists $g \in \mathscr{L}^{1}(I)$ such that
\begin{equation*}
  \mathcal{W}(\mathrm{P}(s), \mathrm{P}(t)) \leq \int_{s}^{t} g(u) \mathtt{d} u \text{ for all } s,t \in I.
\end{equation*}
As a consequence, for all $\mathrm{R} \in \mathscr{P}^{\mathrm{s}}$, the map $t \mapsto \mathcal{W}(\mathrm{P}(t), \mathrm{R})$ is Lipschitz on $I$, hence differentiable for almost every $t \in I$. We recall the following important definition.

\begin{definition}[EVI-gradient flow]
  For $K \in \mathbb{R}$ we say that the curve $(\mathrm{P}(t))$ is a $\mathsf{EVI}(K)$-gradient flow for $\fbeta$ provided $\fbeta(\mathrm{P}(0)) < \infty$ and one of the following equivalent condition holds.
\begin{enumerate}
  \item The curve $(\mathrm{P}(t))_{t \in I}$ is locally absolutely continuous and satisfies the differential inequality
    \begin{equation}\label{def:evi:differential}
      \frac{1}{2} \frac{\mathtt{d}}{\mathtt{d} t} \mathcal{W}^{2}(\mathrm{P}(t), \mathrm{R}) + \frac{K}{2} \mathcal{W}^{2}(\mathrm{P}(t), \mathrm{R}) + \fbeta(\mathrm{P}(t)) \leq \fbeta(\mathrm{R}), \text{ for all } \mathrm{R} \in \mathscr{P}^{\mathrm{s}} \text{ and a.e.}\ t \in I.
    \end{equation}
  \item The curve $(\mathrm{P}(t))_{t \in I}$ satisfies, for all $\mathrm{R} \in \mathscr{P}^{\mathrm{s}}$ and all $s < t \in I$,  the integral inequality:
  \begin{equation}\label{def:evi:integral}
    \int_{s}^{t} \mathrm{e}^{Ku} \fbeta(\mathrm{P}(u)) \mathtt{d} u + \frac{1}{2} \mathrm{e}^{Kt} \mathcal{W}^{2}(\mathrm{P}(t), \mathrm{R}) - \frac{1}{2} \mathrm{e}^{Ks} \mathcal{W}^{2}(\mathrm{P}(s), \mathrm{R}) \leq \left(\int_{s}^{t} \mathrm{e}^{Ku} \mathtt{d} u\right) \fbeta(\mathrm{R}).
  \end{equation}
\end{enumerate}
\end{definition}
Let us briefly explain the equivalence.
\eqref{def:evi:integral} follows from \eqref{def:evi:differential} by multiplying by $\mathrm{e}^{Kt}$ and integrating.
Conversely, choosing $\mathrm{R} = \mathrm{P}(s)$ in \eqref{def:evi:integral} yields
\begin{equation*}
  \frac{1}{2} \mathcal{W}^{2}(\mathrm{P}(t), \mathrm{P}(s)) \leq \left(\int_{s}^{t} \mathrm{e}^{Ku} \mathtt{d} u\right) (\fbeta(\mathrm{P}(s) - \inf \fbeta)),
\end{equation*}
thus $(\mathrm{P}(t))_t$ is absolutely continuous, and we obtain \eqref{def:evi:differential} from \eqref{def:evi:integral} by differentiating.

\subparagraph{Properties of EVI gradient flows.} We recall here without proofs some well-known consequences of EVI-gradient flows, see \cite{DaneriSavare} for details and more general statements.
\begin{proposition}\label{th:properties-evi}
  Let $(\mathrm{P}(t))_{t \geq 0}$ be a $\mathsf{EVI}(K)$-gradient flow for $\fbeta$.
  Then:
  \begin{enumerate}
    \item $\fbeta(\mathrm{P}(t)) < \infty$ for all $t > 0$.
    \item $t \mapsto \fbeta(\mathrm{P}(t))$ is non-increasing.
    \item For all $t > 0$ and all $\mathrm{R} \in \mathscr{P}^{\mathrm{s}}$:
      \begin{equation*}
        \fbeta(\mathrm{P}(t)) \leq \fbeta(\mathrm{R}) + \frac{1}{2 \int_{0}^{t} \mathrm{e}^{Ku} \mathtt{d} u} \mathcal{W}^{2}(\mathrm{P}(0), \mathrm{R}).
      \end{equation*}
    \item If $(\mathrm{Q}(t))_{t \geq 0}$ is another $\mathsf{EVI}(K)$-gradient flow, we have the following contractivity estimate:
      \begin{equation*}
         \mathcal{W}(\mathrm{P}(t), \mathrm{Q}(t)) \leq \mathrm{e}^{Kt} \, \mathcal{W}(\mathrm{P}(0), \mathrm{Q}(0)), \qquad t \geq 0,
       \end{equation*}
    \item To a given initial condition there exists at most one corresponding $\mathsf{EVI}(K)$-gradient flow.
    \item If $K > 0$, then $\fbeta$ admits a unique minimizer $\bar{\mathrm{P}}$, and for all $t > 0$ we have:
      \begin{equation}
      \label{expofast}
      \mathcal{W}(\mathrm{P}(t), \bar{\mathrm{P}}) \leq \mathrm{e}^{-Kt} \mathcal{W}(\mathrm{P}(0), \bar{\mathrm{P}}),  \quad \text{and} \quad \fbeta(\mathrm{P}(t)) - \min \fbeta \leq \frac{1}{2} \frac{K}{\mathrm{e}^{Kt}-1} \mathcal{W}^{2}(\mathrm{P}(0), \bar{\mathrm{P}}).
      \end{equation}
  \end{enumerate}
\end{proposition}

\subsubsection*{Solutions to the infinite-volume Fokker--Planck equations form an EVI-gradient flow}
We now show that strong solutions to the Fokker--Planck equation form an EVI-gradient flow for $\fbeta$.
\begin{theorem}[Specific EVI]\label{th:evi}
  On the metric space $(\mathscr{P}^{\mathrm{s}}, \mathcal{W})$, every solution $(\mathrm{P}(t))_t$ to \eqref{def:STRONG} such that $\fbeta(\mathrm{P}(0)) < \infty$ is an $\mathsf{EVI}(K_{\beta})$-gradient flow $\fbeta$, with $K_\beta$ given by:
  \begin{equation}
  \label{def:Kbeta}
K_\beta := \kappa - 2 \beta \|J\|_{\ell^1} \|\W\|_{\mathscr{L}^\infty},
  \end{equation}
  where $\kappa$ is a not necessarily positive, uniform lower bound on $\Ricc + \nabla^2 \V$ as in \eqref{PosCurv}.
\end{theorem}
 
Before proving Theorem \ref{th:evi}, let us state some immediate consequences obtained from the general properties of EVI-gradient flows recalled in Proposition \ref{th:properties-evi}.
\begin{corollary}[Uniqueness of the gradient flow]
\label{th:uniqueness}
To a given initial solution in $\mathscr{P}^{\mathrm{s}}$, there corresponds at most one \eqref{def:STRONG} solution.
\end{corollary}

\begin{corollary}[Long-time behavior of the flow]\label{th:long-time-behaviour} The quantity $t \mapsto \fbeta(\mathrm{P}(t))$ is always non-increasing. Moreover, if $K_{\beta} > 0$, with $K_\beta$ as in \eqref{def:Kbeta}, then $\fbeta$ admits a unique minimizer $\bar{\mathrm{P}}$ and we have:
  \begin{enumerate}
    \item $\fbeta(\mathrm{P}(t)) \to \fbeta(\bar{\P}) = \min \fbeta$ as $t \to \infty$, exponentially fast as in \eqref{expofast}.
    \item There is exponential convergence of the flow to $\bar{\P}$ \emph{in specific Wasserstein distance}, namely
      \begin{equation*}
        \mathcal{W}^{2}(\mathrm{P}(t), \bar{\mathrm{P}}) \leq \mathrm{e}^{-K_{\beta}t} \mathcal{W}^{2}(\mathrm{P}(0), \bar{\mathrm{P}}).
      \end{equation*}
  \end{enumerate}
\end{corollary}
To prove Theorem \ref{th:evi}, we consider a solution $(\mathrm{P}(t))_{t\geq 0}$ and pretend that we want to prove that its restriction $(\mathrm{P}_{n}(t))_t$ to the finite box $\Lan$ is the EVI-gradient flow of the finite-volume free energy functional $\Fn$.
Indeed, it is well-known that solutions to finite-volume Fokker--Planck equations on $\mathsf{Conf}_{n}$ with drift $-\beta \nabla \HH_n$ form an $\mathsf{EVI}(K_{\beta})$-gradient flow for the finite-volume free energy functional $\Fn$ with respect to the usual Wasserstein distance $\mathcal{W}_{n}$, see for instance \cite[Prop~4.4]{Erbar_2010} (for the heat equation).

Since $(\mathrm{P}_{n}(t))_t$ does \emph{not} exactly solve the finite-volume Fokker--Planck equation, an error term arises, and we show how to control it in the $n \to \infty$ limit.

\begin{proof}[Proof of Theorem \ref{th:evi}]Let $t \mapsto \P(t)$ be a strong solution, and let $\p_n$ be its local density in $\Lan$ for $n \geq 1$. 

\subsubsection*{Step 1. Time derivative of the Wasserstein distance.}
\begin{lemma}
\label{th:evi:wasserstein-derivative}
  Let $n \geq 1$ and let $\mathrm{Q} \in \mathscr{P}^{\mathrm{s}}$.
  The map $t \mapsto \mathcal{W}_{n}^{2}(\mathrm{P}(t), \mathrm{Q})$ is differentiable almost everywhere on $[0, + \infty)$.
  Moreover, at any differentiability point $t$, denoting by $\theta_{n}(t)$ the Kantorovich potential for the optimal transport from $\mathrm{P}(t)$ to $\mathrm{Q}$ in $\Lan$ (with respect to $\mathcal{W}_{n}$), we have:
  \begin{equation}
  \label{WassDeriv}
    \frac{1}{2} \frac{\mathtt{d}}{\mathtt{d} t} \Wn^{2}(\mathrm{P}(t), \mathrm{Q}) = -  \mathbb{E}_{\mathrm{P}(t)} \bracket*{ \nabla \theta_{n}(t) \cdot \paren*{ \frac{\nabla \mathrm{p}_{n}(t)}{\mathrm{p}_{n}(t)} + \beta \nabla \HH_{n} } } + o(|\Lambda_{n}|),
\end{equation}
where $o(|\Lambda_{n}|)$ depends only on the parameter of the model, and in particular is uniform in $t, \mathrm{P}, \mathrm{Q}$.
\end{lemma}
\begin{proof}
  This is closely related to \cite[Thm.~23.9]{VillaniOldNew}, however, the result there is stated for locally Lipschitz vector fields, yet, the regularity of the conditional expectation appearing in \eqref{def:STRONG} is not clear.
  From the proof of \cite[Thm.~23.9]{VillaniOldNew}, one could observe that the Lipschitz property is not needed for our purposes. We give a short different proof for completeness.

By Corollary \ref{th:fisher}, $\nabla \log \mathrm{p}_{n}(t)$ is well-defined for all $t \geq 0$.
  For $\bx \in M^\Lan$, define the vector field
  \begin{equation*}
    v_{n}(t,\bx) := \nabla \log \mathrm{p}_{n}(t,\bx) + \beta \nabla \HH_n(\bx) + \beta \mathbb{E}_{\mathrm{P}(t)} \bracket*{ \nabla \HH(\Zd \setminus \Lambda_{n} \to \Lambda_{n}) \given \Lambda_{n} }(\bx).
  \end{equation*}
  The last two terms in the right-hand side sum up to $\beta \mathbb{E}_{\mathrm{P}(t)} \bracket*{ \nabla \HH(\Zd \to \Lambda) \given \Lambda }$ but we split up that quantity into a local part and a non-local part.
  The following \emph{continuity equation} is an immediate consequence of \eqref{def:STRONG} (which implies \eqref{def:VWEAK}):
  \begin{equation}\label{eq:continuity-equation:local}
    \partial_{t} \mathrm{P}_{n}(t) = \operatorname{div}(v_{n}(t, \cdot) \mathrm{P}_{n}(t)), \quad \text{in the sense of distributions}.
  \end{equation}
Note that, although it does not appear explicitly in the definition, the term $\nabla \V_{n}$ is contained in the vector field $v_{n}$ through $\mathrm{p}_{n}$, which is the density with respect to the weighted measure $\omega_{n}$, defined in \eqref{def:oL}. 

The Benamou--Brenier formula on compact manifolds \cite[Prop.~2.30]{AGUser} guarantees that
  \begin{equation*}
    \mathcal{W}_{n}^{2}(\mathrm{P}_{n}(t), \mathrm{P}_{n}(s)) \leq (t-s) \int_{s}^{t} \int_{M^\Lan} \|v_{n}(u,\bx)\|^{2} \mathrm{p}_{n}(u,\bx) \dd \omega_{n}(\bx) \mathtt{d} u.
  \end{equation*}
  Thanks to the short-range assumption \eqref{eq:short-range}, $\nabla \HH_{n} + \beta \mathbb{E}_{\mathrm{P}(t)} \bracket*{ \nabla \HH(\mathbb{Z}^{\d} \to \Lambda) \given \Lambda }$ is bounded. Combining this with the boundedness of Fisher information (Corollary \ref{th:fisher}), we deduce that the map $t \mapsto \mathcal{W}_{n}^{2}(\mathrm{P}(t), \mathrm{Q})$ is locally Lipschitz and, thus, by Rademacher's theorem, differentiable almost everywhere.

  Now, choose a differentiability point $t$ and take $h > 0$. The Kantorovich dual formulation of optimal transport yields:
  \begin{equation*}
    \mathcal{W}_{n}^{2}(\mathrm{P}(t+h), \mathrm{Q}) - \mathcal{W}_{n}^{2}(\mathrm{P}(t), \mathrm{Q}) \geq \mathbb{E}_{\mathrm{P}_{n}(t+h)}[\theta_{n}(t)] - \mathbb{E}_{\mathrm{Q}_{n}}[\theta_{n}(t)^{c}] - \mathbb{E}_{\mathrm{P}_{n}(t)}[\theta_{n}(t)] + \mathbb{E}_{\mathrm{Q}_{n}}[\theta_{n}(t)^{c}].
  \end{equation*}
  Using \eqref{eq:continuity-equation:local} and integrating by parts, we obtain
  \begin{equation*}
    \mathbb{E}_{\mathrm{P}_{n}(t+h)} \bracket*{ \theta_{n}(t)} - \mathbb{E}_{\mathrm{P}_{n}(t)} \bracket*{ \theta_{n}(t)} = - \int_{t}^{t+h} \mathbb{E}_{\mathrm{P}_{n}(s)} \bracket*{ \nabla \theta_{n}(t) \cdot v_{n}(s) } \mathtt{d} s.
  \end{equation*}
  Combining the two previous equations, we find
  \begin{equation*}
    \mathcal{W}_{n}^{2}(\mathrm{P}(t+h), \mathrm{Q}) - \mathcal{W}_{n}^{2}(\mathrm{P}(t), \mathrm{Q}) \geq -h \mathbb{E}_{\mathrm{P}_{n}(t)} \bracket*{  v_{n}(t) \cdot \nabla \theta_{n}(t) } + o(h).
  \end{equation*}
  Dividing by $h$ and letting $h \to 0$, then repeating with $h < 0$, we get \eqref{WassDeriv} provided we show
  \begin{equation}
  \label{Epnnthetasmall}
    \mathbb{E}_{\mathrm{P}_{n}(t)} \bracket*{ \nabla \theta_{n}(t) \cdot \mathbb{E}_{\mathrm{P}(t)} \bracket*{ \nabla \HH(\mathbb{Z}^{\d} \setminus \Lambda_{n} \to \Lambda_{n}) \given \Lambda_{n} } } = o (|\Lambda_{n}|).
  \end{equation}
  From the short-range range assumption, we know that $\sup_{\bx \in \Conf} \norm*{ \nabla \HH(\mathbb{Z}^{\d} \setminus \Lambda_{n} \to \Lambda_{n}) \given \Lambda_{n} }(\bx)^{2} = o(|\Lambda_{n}|)$, where $o(|\Lambda_{n}|)$ depends only on the parameters of the model. Thus by the Cauchy--Schwarz inequality
  \begin{equation*}
    \mathbb{E}_{\mathrm{P}_{n}(t)} \bracket*{ \nabla \theta_{n}(t) \cdot \mathbb{E}_{\mathrm{P}(t)} \bracket*{ \nabla \HH(\mathbb{Z}^{\d} \setminus \Lambda_{n} \to \Lambda_{n}) \given \Lambda_{n} } } \leq o(|\Lambda_{n}|)^{1/2} \paren*{\mathbb{E}_{\mathrm{P}_{n}(t)} \bracket*{ \|\nabla \theta_{n}(t)\|^{2} }}^{1/2}.
  \end{equation*}
  Since we have chosen $\theta_n$ as the Kantorovitch potential between $\P$ and $\Q$ in $\Lan$, the Brenier--McCann theorem \cite[Thm.~9]{McCannPolar} ensures that the expectation on the right-hand side is exactly $\mathcal{W}^2_{n}(\mathrm{P}(t), \mathrm{Q})$, which is $\O(|\Lambda_{n}|)$ because $M$ is compact. We thus get \eqref{Epnnthetasmall}.
\end{proof}

\paragraph{Step 2. An \emph{almost} EVI in finite volume.}
\begin{lemma}\label{th:evi:finite-volume}
  Let $n \geq 1$ and $\mathrm{Q} \in \mathscr{P}^{\mathrm{s}}$.
  Then, at a differentiability point,
  \begin{equation}\label{eq:evi:finite-volume}
    \frac{1}{2} \frac{\mathtt{d}}{\mathtt{d} t} \mathcal{W}_{n}^{2}(\mathrm{P}(t), \mathrm{Q}) + \frac{K_{\beta}}{2} \mathcal{W}_{n}^{2}(\mathrm{P}(t), \mathrm{Q}) + \Fn(\mathrm{P}(t)) \leq \Fn(\mathrm{Q}) + o(|\Lambda_{n}|),
  \end{equation}
  where the error term $o(|\Lambda_{n}|)$ is as in Lemma \ref{th:evi:wasserstein-derivative} --- in particular, it is uniform in $t$.
\end{lemma}
\begin{proof}
  Using a de Bruijn's type argument, for instance \cite[Prop.~5.2.2]{BakryGentilLedoux}, together with \eqref{eq:continuity-equation:local}, we find that $\fbeta(\mathrm{P}(t))$, and thus also $\Fn(\P(t))$, is finite for all $t \geq 0$. We may therefore assume that $\Fn(\mathrm{Q}) < \infty$, otherwise there is nothing to prove.
  In particular, $\mathrm{Q}_{n}$ admits a density with respect to $\omega_n$.

  Let again $\theta_{n}(t)$ be the Kantorovich potential for the transport from $\mathrm{P}_{n}(t)$ to $\mathrm{Q}_{n}$ with respect to $\mathcal{W}_{n}$.
  By stability of the Ricci curvature under products, on $M^\Lan$ we have, with $K_\beta$ as in \eqref{def:Kbeta}:
  \begin{equation*}
    \mathrm{Ric}_{n} + \nabla^{2} \V_{n} + \beta \nabla^{2} \HH_n \geq K_{\beta}.
  \end{equation*}
 We can then apply \cite[Prop.~4.2]{CMCS06} (note the different sign convention for Kantorovich potentials), with
  \begin{equation*}
    \mu := \frac{1}{\int \mathrm{e}^{-\mathsf{H}_{n}} \mathtt{d} \omega_{n}} \mathrm{e}^{-\beta \HH_n} \omega_{n}, \quad \text{and} \quad g := \frac{\mathtt{d}\mathrm{Q}_{n}}{\mathtt{d} \mu}, \quad \text{and} \quad f := \frac{\mathtt{d} \mathrm{P}_{n}(t)}{\mathtt{d} \mu} = \mathrm{p}_{n}(t) \mathrm{e}^{\V_{n}} \mathrm{e}^{\beta \HH_n},
  \end{equation*}
  which gives:
  \begin{equation*}
    \Fn(\mathrm{Q}) - \Fn(\mathrm{P}(t)) \geq - \int_{M^\Lan} \nabla \theta_{n}(t) \cdot \nabla f \mathtt{d} \mu + \frac{K_{\beta}}{2} \mathcal{W}_{n}^{2}(\mathrm{P}(t), \mathrm{Q}).
  \end{equation*}
  By the chain rule, we have:
  \begin{equation*}
    \nabla f \mathtt{d} \mu = \nabla \mathrm{p}_{n}(t) \mathtt{d} \omega_{n} + \beta \nabla \HH_n \mathtt{d} \mathrm{P}_{n}(t),
  \end{equation*}
  and thus we get:
  \begin{equation*}
 \int_{M^\Lan} - \nabla \theta_{n}(t) \cdot \left( \frac{\nabla \mathrm{p}_{n}(t)}{\mathrm{p}_{n}(t)} + \beta \nabla \HH_n \right) \mathrm{p}_{n}(t) \mathtt{d} \omega_{n}    + \beta \nabla \HH_n \mathtt{d} \mathrm{P}_{n}(t) + \Fn(\mathrm{P}(t)) \leq  \Fn(\mathrm{Q}).
  \end{equation*}
 Inserting this into \eqref{WassDeriv} gives \eqref{eq:evi:finite-volume}.
\end{proof}

\paragraph{Step 3. Conclusion: proof of Theorem \ref{th:evi}}
  Formally, we divide by $|\Lambda_{n}|$ in \eqref{th:evi:finite-volume} and let $n \to \infty$.
  However, the limit and the time derivative might fail to commute.
  Thus, we write instead the integral version of \eqref{eq:evi:finite-volume}, noting that it holds for almost every $t$ by \cref{th:evi:wasserstein-derivative},
  \begin{equation*}
    \int_{s}^{t} \mathrm{e}^{uK_{\beta}} \Fn(\mathrm{P}(u)) \mathtt{d} u + \frac{1}{2} \mathrm{e}^{t K_{\beta}} \mathcal{W}_{n}^{2}(\mathrm{P}(t), \mathrm{Q}) - \frac{1}{2} \mathcal{W}_{n}^{2}(\mathrm{P}(s), \mathrm{Q}) \leq \paren*{\Fn(\mathrm{Q}) - o(|\Lambda_{n}|)} \int_{s}^{t} \mathrm{e}^{u K_{\beta}} \mathtt{d} u.
  \end{equation*}
  Integrating the $o(|\Lambda_{n}|)$ in \eqref{eq:evi:finite-volume} and taking it out of the integral is allowed since we have shown that it is uniform in time.
  Now we divide the above inequality by $|\Lambda_{n}|$ and let $n \to \infty$.
  
  In the convergence $\Fn \to \fbeta$, the convergence of the entropy term is monotonous and the convergence of the energy term is uniform, by Lemma \ref{lem:PropertiesOfTheLimits}. We may thus pass to the limit in the integral:
  \begin{equation*} 
    \frac{1}{|\Lambda_{n}|} \int_{s}^{t} \mathrm{e}^{uK_{\beta}} \Fn(\mathrm{P}(u)) \mathtt{d} u \xrightarrow[n \to \infty]{} \int_{s}^{t} \mathrm{e}^{u K_{\beta}} \fbeta(\mathrm{P}(u)).
\end{equation*}
We obtain the integral formulation \eqref{def:evi:integral} of the EVI.
\end{proof}

\subsection{Regularity of solutions: Proof of \texorpdfstring{Theorem \ref{theo:regul}}{the regularity}}
\label{s:proof-regularity}
\newcommand{\loc}{\mathrm{loc}}
\newcommand{\Liloc}{\mathscr{L}^\infty_{\loc}}
\subsubsection*{Strategy of proof and comparison with the classical case}
We use a bootstrap procedure as sketched in \cite[pp. 17-19]{Jordan_1998} and implemented in \cite[Lemma~10.7]{mei2018mean} for the Fokker--Planck equation $\partial_t \rho = \Delta \rho + \div(\rho \V)$ in Euclidean space and with a \emph{local} drift.
Our setting presents several important differences:
\begin{itemize}
\item The drift term in our Fokker--Planck equation differs from the one in \cite{Jordan_1998,mei2018mean} in two ways: it depends on time (this is the case in \cite{mei2018mean} but not in \cite{Jordan_1998}) and is is \emph{non-local}.
This does not pose a significant issue; however, during the bootstrap procedure, one needs to carefully track the dependency of the regularity estimates on the size of the box.
  \item Our spins are living on a compact manifold and not on the Euclidean space. On the one hand, this means we can avoid introducing spatial cut-off functions. On the other hand, the heat kernel is not as nice as in $\R^\nn$, since it takes into account the curvature.
    Both of these proofs rely on \enquote{potential estimates}, namely they control the $\mathscr{L}^p \to \mathscr{L}^p$ operator norm for space-time convolution with the heat kernel, as well as with its first and second derivative.
    In the Euclidean case, such controls for $p = 2$ proceed from fairly simple Fourier analysis, and for $p \in (1, + \infty)$ from more advanced tools (\cite{Jordan_1998,mei2018mean} refer to \cite[Chap.~4, Sec.~3]{zbMATH03277871}, whose ideas are also surveyed in \cite{SalamonParabolic}).
    In the case of a compact manifold, we are not aware of such such bounds --- and their proof would presumably requires different techniques.
\end{itemize} 
Let us summarize the main steps of the regularity argument in \cite{Jordan_1998,mei2018mean}.
\begin{enumerate}
  \item \cite{Jordan_1998} \& \cite{mei2018mean} prove a $\Liloc(0, + \infty)$ bound in time on the $\mathscr{L}^p$ norm in space of the density $\rho$ for $p > 1$ close enough to $1$, using pointwise upper bounds on the heat kernel and its first derivative plus elementary estimates.
  \item 
 \begin{itemize}
  \item \cite{Jordan_1998}: in particular, one has a $\mathscr{L}^p_{\loc}$ bound in space-time on the density $\rho$. Then “usual bootstrap arguments”, using the fine upper bound on the $\mathscr{L}^p \to \mathscr{L}^p$ operator norm mentioned above, allow them to show that the first, second derivatives etc. of $\rho$ are in $\mathscr{L}^p$.

  \item \cite{mei2018mean}: upgrade, still by fairly elementary estimates, the $\Liloc$ bound on $\|\rho\|_{\mathscr{L}^p}$ for $p$ close to $1$ into an $\Liloc$ bound on $\|\rho\|_{\mathscr{L}^p}$ for all $p > 1$. They claim it eventually gives an $\Liloc$ bound on $\|\rho\|_{\mathscr{L}^\infty}$, which is unclear to us, but working with $p$ arbitrarily large is fine.
  \end{itemize}
  \item \cite{mei2018mean}: use the $\mathscr{L}^p \to \mathscr{L}^p$ controls (they use it with $p = \infty$, which is  puzzling, as this does not seem to be covered by the literature, but $p$ large works well) to “bootstrap” the regularity of $\rho$, one derivative at each step.\end{enumerate}
It is unclear how one would improve regularity by one degree at each step of the bootstrap without the fine potential estimates used in \cite{Jordan_1998,mei2018mean} --- which do \emph{not} follow simply from upper bounds on the heat kernel.
We proceed in a similar spirit, but gaining only $1-\epsilon$ degree of regularity at each step --- to do that, it suffices to know pointwise bounds on the heat kernel and its derivatives, such as the ones listed in Appendix \ref{s:heat-kernel}, valid on a general smooth compact manifold.
Our strategy extends readily to the Euclidean case and provides an alternative argument to prove regularity of weak solutions to the usual Fokker--Planck equation without using those potential estimates.

\medskip 

Let us turn to the proof of our regularity result. For $\Lambda \Subset \Zd$, we use properties of the heat semi-group $(\mathsf{G}^{\Lambda}_{t})_{t \geq 0}$ on $M^{\Lambda}$, recalled in Appendix \ref{s:heat-kernel}. For the sake of readibility, we drop the dependency in $\Lambda$ wherever it would not raise confusion.

\subsubsection*{Step 0. Duhamel's principle}
Let $\La \Subset \Zd$. We know from the dual formulation \eqref{conseqVWeak} that for all $F \in \CC^{\infty}( (0, + \infty), M^\La)$ and all $0 \leq t_0 < t_1$:
\begin{multline}
\label{conseqVWeak_recall}
\Esp_{\P(t_1)} [F(t_1, \cdot)] - \Esp_{\P(t_0)} [F(t_0, \cdot)] \\
= \int_{t_0}^{t_1} \Esp_{\P(t)}\left[\partial_t F(t, \cdot) + \Delta_\V F(t, \cdot) - \nabla F(t, \cdot) \cdot \Esp_{\P(t)} \left[ \nabla \HH(\Zd \rightarrow \La) | \La \right]  \right] \dd t.
\end{multline}
Given a function $f \in \mathscr{C}^{\infty}(M^{\Lambda})$, and $\delta > 0$, define $F_\delta$ as:
\begin{equation}\label{eq:f-regul}
 (t, \bx) \mapsto F_\delta(t, \bx) := \mathsf{G}_{t_{1} - t + \delta} f(\bx) \times \chi_{\delta}(t)
\end{equation}
where $\chi_{\delta}$ is a smooth cut-off function equal to $1$ on $[0, t_1 + \frac{\delta}{2}]$ and to $0$ on $[t_1 + \delta, + \infty)$.
By construction and by definition of the heat kernel, $F_\delta$ is smooth and solves
\begin{equation}\label{eq:heat:f-delta}
  \begin{dcases}
  & \partial_{t} F_\delta(t,x) + \Delta_\V F_\delta(t,x) = 0, \qquad (t,x) \in (t_{0},t_{1}) \times M^{\Lambda};
\\& F_\delta(t_{1},\cdot) = \mathsf{G}_{\delta} f.
  \end{dcases}
\end{equation}
Using $F_\delta$ as the test function in \eqref{conseqVWeak_recall}, we obtain:
\begin{equation}\label{eq:duhamelA}
  \mathbb{E}_{\mathrm{P}_{\Lambda}(t_{1})}\bracket*{ \mathsf{G}_{\delta} f } = \mathbb{E}_{\mathrm{P}_{\Lambda}(t_{0})}\bracket*{\mathsf{G}_{t_{1} - t_{0} + \delta} f} - \int_{t_0}^{t_{1}} \mathbb{E}_{\mathrm{P}_{\Lambda}(t)} \bracket*{ \paren*{\nabla \mathsf{G}_{t_{1}-t + \delta} f} \cdot  \mathbb{E}_{\mathrm{P}(t)} \bracket*{ \nabla \HH(\mathbb{Z}^{\d} \to \Lambda) \given \Lambda } }  \mathtt{d} t.
  \end{equation}
  Finally, sending $\delta \to 0$, we obtain the following expression known as \emph{Duhamel's principle}:
\begin{equation}\label{eq:duhamel}
  \mathbb{E}_{\mathrm{P}_{\Lambda}(t_{1})}\bracket*{ f } = \mathbb{E}_{\mathrm{P}_{\Lambda}(t_{0})}\bracket*{\mathsf{G}_{t_{1} - t_{0}} f} - \int_{t_0}^{t_{1}} \mathbb{E}_{\mathrm{P}_{\Lambda}(t)} \bracket*{ \paren*{\nabla \mathsf{G}_{t_{1}-t} f} \cdot \mathbb{E}_{\mathrm{P}(t)} \bracket*{ \nabla \HH(\mathbb{Z}^{\d} \to \Lambda) \given \Lambda }  } \mathtt{d} t.
  \end{equation}

\begin{remark} We deduce from \eqref{eq:duhamel} that if $t \mapsto \P(t)$ is a dual solution, then it particular $\P(t_1)$ converges to $\P(0)$ in the local topology as $t_1 \to 0$, which is consistent with our notion of “initial condition” for strong solutions, see \eqref{def:STRONG}. Indeed, taking $t_0 = 0$ in \eqref{eq:duhamel} one gets, for any bounded local function $f$:
\begin{equation*}
  \mathbb{E}_{\mathrm{P}_{\Lambda}(t_{1})}\bracket*{ f }  =   \mathbb{E}_{\mathrm{P}_{\Lambda}(0)}\bracket*{\mathsf{G}_{t_{1}}f } - \int_{0}^{t_{1}} \| \nabla \mathsf{G}_{t} f \|_{\infty} \mathtt{d} t  \times \|f\|_{\infty} \times \O_\La(1), 
\end{equation*}
the first term in the right-hand side tends to $\mathbb{E}_{\mathrm{P}_{\Lambda}(0)}\bracket*{\mathsf{G}_{t_{1} f}}$ as $t_1 \to 0$ while the second one goes to $0$ because, as recalled in Section \ref{s:heat-kernel}, $t \mapsto \| \nabla \mathsf{G}_{t} f \|_{\infty}$ blows up as $t^{-\hal}$ near $0$ and is thus integrable.
\end{remark}

\subsubsection*{Step 1. Existence and some integrability for the local densities}
\begin{claim}
For all $\Lambda \Subset \mathbb{Z}^{\d}$, and for all $t > 0$, the measure $\mathrm{P}_{\Lambda}(t)$ is absolutely continuous with respect to $\oL$, and its density $\mathrm{p}_{\Lambda}(t)$ is in $\mathscr{L}^{p}(M^\La)$ for all $p \in \paren*{1, \frac{\nn|\Lambda|}{\nn|\Lambda|-1}}$.
Moreover, for all $0 < \delta < T$, we have:
\begin{equation}
\label{Lpcontrol}
  \sup_{t \in [\delta,T]} \norm{\mathrm{p}_{\Lambda}(t)}_{\mathscr{L}^{p}(\oL)} \leq \Cc(\delta, T, p, |\Lambda|).
\end{equation}
Note that the constant $\Cc(\delta, T, p, |\Lambda|)$ depends on $\La$ only through its size.
\end{claim}
\begin{proof}
We work first with $t_{0} := 0$. We write $p'$ for the Hölder conjugate of $p$. Applying Duhamel's principle \eqref{eq:duhamel} and using a rough bound for the conditional expectation we find, for all $f \in \mathscr{C}^{\infty}(M^{\Lambda})$:
  \begin{equation*}
  \abs*{\mathbb{E}_{\mathrm{P}_{\Lambda}(t_{1})}\bracket*{f}} \leq \|\mathsf{G}_{t_{1}}f\|_\infty +  \int_{0}^{t_{1}} \sup_{\bx \in \Conf} \|\nabla \HH(\mathbb{Z}^{\d} \to \Lambda)(\bx) \| \times \|\mathbb{E}_{\mathrm{P}(t)}\left[ \nabla \mathsf{G}_{t_{1}-t} f \right] \| \ \mathtt{d} t.
  \end{equation*}
  Thanks to our short-range assumption, and the fact that the interaction potential $\Psi$ is $\CC^1$ on a compact manifold, we know that for some constant $\Cc$ depending only on the model and on $|\La|$,
  \begin{equation*}
\sup_{\bx \in \Conf} \|\nabla \HH(\mathbb{Z}^{\d} \to \Lambda)(\bx) \| \leq \Cc(|\La|).
  \end{equation*}

  On the one hand, applying Young’s inequality \eqref{eq:heat-semigroup:sobolev} with $k := 0$ and $q := \infty$, we know that:
  \begin{equation}
  \label{ApplyYoung1}
\|\mathsf{G}_{t_{1}}f\|_\infty \leq \Cc(|\La|) \times \norm{f}_{\mathscr{L}^{p'}} \times  t_{1}^{-\frac{\nn|\Lambda|}{2}\paren*{1-\frac{1}{p}}},
  \end{equation}
with a constant which is locally uniform in time, so may write:
\begin{equation}
\label{eq:PremierBound}
  \abs*{\mathbb{E}_{\mathrm{P}_{\Lambda}(t_{1})}\bracket*{f}} \leq \Cc(|\La|) \times \norm{f}_{\mathscr{L}^{p'}} t_{1}^{-\frac{\nn|\Lambda|}{2}\paren*{1-\frac{1}{p}}}  + \Cc(|\La|) \times \int_{0}^{t_{1}}  \|\mathbb{E}_{\mathrm{P}(t)}\left[ \nabla \mathsf{G}_{t_{1}-t} f \right] \| \ \mathtt{d} t.
\end{equation}
On the other hand, still by \eqref{eq:heat-semigroup:sobolev} with $k := 1$ and $q := \infty$, we find that:
  \begin{multline}
  \label{k1qinfini}
   \int_{0}^{t_1} \|\mathbb{E}_{\mathrm{P}(t)}\left[ \nabla \mathsf{G}_{t_{1}-t} f \right] \| \ \mathtt{d} t \leq \int_{0}^{t_{1}} \| \nabla \mathsf{G}_{t_{1}-t} f \|_\infty  \mathtt{d} t \leq \norm{f}_{\mathscr{L}^{p'}} \times \int_{0}^{t_{1}}  t^{-\frac{\nn |\Lambda|}{2} \paren*{1-\frac{1}{p}}} \paren*{ 1 + (t_{1}-t)^{-\frac{1}{2}}} \dd t \\
   \leq \Cc(|\La|) \times  \norm{f}_{\mathscr{L}^{p'}} t_{1}^{\frac{1}{2} - \frac{\nn|\Lambda|}{2} \paren*{1-\frac{1}{p}} },
  \end{multline}
the integral being finite thanks to our choice of $p < \frac{\nn|\Lambda|}{\nn|\Lambda|-1}$. 

By duality, we deduce that $\mathrm{P}_{\Lambda}(t_{1})$ admits a density $\mathrm{p}_{\Lambda}(t_{1}) \in \mathscr{L}^{p}(M^\La)$  satisfying 
  \begin{equation*}
    \norm{\mathrm{p}_{\Lambda}(t_{1})}_{\mathscr{L}^{p}} \leq \Cc(|\La|) \times \left(t_{1}^{\frac{1}{2} - \frac{\nn|\Lambda|}{2} \paren*{1-\frac{1}{p}} } + t_{1}^{- \frac{\nn|\Lambda|}{2} \paren*{1-\frac{1}{p}} }\right).
  \end{equation*}
  The above quantity is bounded away from $t_{1} = 0$ and $t_{1} = \infty$, which yields \eqref{Lpcontrol}.
\end{proof}

\subsubsection*{Step 2. More integrability for the local densities}
Having established the $p$-integrability of the local densities for some sufficiently small $p > 1$ depending on the size of the box, we turn to a bootstrap argument yielding $p$-integrability for all $p$.
\begin{claim}
\label{claim:Lpforallp}
For all $\Lambda \Subset \mathbb{Z}^{\d}$, for all $t > 0$, and for all $p \in [1, \infty)$, the local density $\mathrm{p}_{\Lambda}(t)$ is in $\mathscr{L}^{p}(M^\La)$.
  Moreover, for all $0 < \delta < T$, we have:
  \begin{equation}
  \label{eq:Linftyplarge}
    \sup_{t \in [\delta,T]} \norm{\mathrm{p}_{\Lambda}(t)}_{\mathscr{L}^{p}(M^\La)} \leq \Cc(\delta, T, p, |\Lambda|).
  \end{equation}
\end{claim}
  
\begin{proof}
By induction, we show that there is a sequence $p_n \to \infty$ such that \eqref{eq:Linftyplarge} holds with $p = p_n$. 

First, choose $p_1 \in \left(1, \frac{\nn |\La|}{\nn |\La| -1}\right)$, such that $t^{-\hal - \frac{\nn |\La|}{2}\left(1-\frac{1}{p_1}\right)}$ is integrable near $0$, as in the previous step. We then know that \eqref{eq:Linftyplarge} holds for $p = p_1$.

Next, at each induction step, we choose $p_{n+1}$ by requiring that:
\begin{equation*}
\frac{1}{p_{n+1}} := \frac{1}{p_n} + \left( \frac{1}{p_1} -1 \right).
\end{equation*} 
Observe that since $p_1 > 1$ we have $\frac{1}{p_1} -1 < 0$ and thus $\frac{1}{p_{n+1}} < \frac{1}{p_n}$ which means that the sequence $\{p_n\}_n$ is non-increasing. It is not hard to see that it cannot be bounded, thus it tends to $+ \infty$.

Applying Hölder's inequality, we find (cf. \eqref{k1qinfini}) for all $t < t_1$:
    \begin{equation} 
    \label{eq:regularity:bootstrap-lp}
\|\mathbb{E}_{\mathrm{P}(t)}\left[ \nabla \mathsf{G}_{t_{1}-t} f \right] \| \leq \int_{M^\La} \|\nabla \mathsf{G}_{t_{1}-t} f(\bx) \| \mathrm{p}_{\Lambda}(t) \dd \oL(\bx)
\leq \norm{\mathrm{p}_{\Lambda}(t)}_{\mathscr{L}^{p_{n}}}  \|\nabla \mathsf{G}_{t_{1}-t} f \|_{\mathscr{L}^{(p_{n})'}}. 
\end{equation}
Next, we apply Young's inequality \eqref{eq:heat-semigroup:young} for $k= 1$ with $r = (p_n)', q = (p_{n+1})'$ and $p = p_1$. By definition, we have:
\begin{equation*}
\frac{1}{p_n} + \frac{1}{p_1} = 1 + \frac{1}{p_{n+1}}, \text{ and thus also } \frac{1}{(p_n)'} + \frac{1}{p_1} = 1 + \frac{1}{(p_{n+1})'},
\end{equation*}
and we get for all $t < t_1 \leq T$: 
\begin{equation*}
\|\nabla \mathsf{G}_{t_{1}-t} f \|_{\mathscr{L}^{(p_{n})'}} \leq \Cc(T, p_{n+1}, |\La|) \times \norm{f}_{\mathscr{L}^{(p_{n+1})'}} (t_{1}-t)^{-\frac{\nn|\Lambda|}{2}\paren*{1-\frac{1}{p_{1}}}} \paren*{ 1 + (t_{1}-t)^{-\frac{1}{2}}}.
\end{equation*}
Inserting this into \eqref{eq:regularity:bootstrap-lp} yields:
\begin{multline}
\label{grandt}
\|\mathbb{E}_{\mathrm{P}(t)}\left[ \nabla \mathsf{G}_{t_{1}-t} f \right] \| \\
\leq \Cc(T, p_{n+1}, |\La|) \times  \norm{\mathrm{p}_{\Lambda}(t)}_{\mathscr{L}^{p_{n}}} \times \norm{f}_{\mathscr{L}^{(p_{n+1})'}} \times (t_{1}-t)^{-\frac{\nn|\Lambda|}{2}\paren*{1 - \frac{1}{p_{1}}}} \paren*{ 1 + (t_{1}-t)^{-\frac{1}{2}} },
\end{multline}
and we use this bound for $\hal t_1 \leq t \leq t_1$. On the other hand, for $\delta \leq t_0 \leq t \leq \hal t_1$ we can write:
\begin{equation}
\label{petitt}
\|\mathbb{E}_{\mathrm{P}(t)}\left[ \nabla \mathsf{G}_{t_{1}-t} f \right] \| \leq \|\nabla \mathsf{G}_{t_{1}-t} f\|_\infty \leq \Cc(\delta, p_{n+1}, |\La|) \times \norm{f}_{\mathscr{L}^{(p_{n+1})'}},
\end{equation}
taking advantage of the fact that since $t$ is far from $t_1$, the heat kernel at time $t_1 -t$ is smooth.

Now we take $t_{0} \geq \delta > 0$ and argue as in the previous step, starting from Duhamel's principle. We reach \eqref{eq:PremierBound} but instead of \eqref{k1qinfini}, we now use \eqref{grandt}, \eqref{petitt} and write, distinguishing between the two cases “$t$ far from $t_1$” and “$t$ close to $t_1$” in the integral:
    \begin{multline*}
\int_{t_0}^{t_1} \|\mathbb{E}_{\mathrm{P}(t)}\left[ \nabla \mathsf{G}_{t_{1}-t} f \right] \| \dd t 
\leq \int_{t_0}^{\hal t_1} \|\mathbb{E}_{\mathrm{P}(t)}\left[ \nabla \mathsf{G}_{t_{1}-t} f \right] \| \dd t +
\int_{\hal t_1}^{t_1} \|\mathbb{E}_{\mathrm{P}(t)}\left[ \nabla \mathsf{G}_{t_{1}-t} f \right] \| \dd t 
\\
\leq \Cc(\delta, p_{n+1}, |\La|) \times \norm{f}_{\mathscr{L}^{(p_{n+1})'}} \\ 
+  \Cc(T, p_{n+1}, |\La|) \times \int_{\hal t_1}^{t_1}  \norm{\mathrm{p}_{\Lambda}(t)}_{\mathscr{L}^{p_{n}}} \times \norm{f}_{\mathscr{L}^{(p_{n+1})'}} \times (t_{1}-t)^{-\frac{\nn|\Lambda|}{2}\paren*{1 - \frac{1}{p_{1}}}} \paren*{ 1 + (t_{1}-t)^{-\frac{1}{2}} } \dd t.
\end{multline*}

Using our induction hypothesis, we have $\sup_{t \in [\hal t_1, t_1]} \norm{\mathrm{p}_{\Lambda}(t)}_{\mathscr{L}^{p_{n}}} \leq \Cc(\delta, T, p_{n}, |\La|)$, we are thus left with:
\begin{multline*}
\int_{t_0}^{t_1} \|\mathbb{E}_{\mathrm{P}(t)}\left[ \nabla \mathsf{G}_{t_{1}-t} f \right] \| \dd t  \leq \Cc(\delta, p_{n+1}, |\La|) \times \norm{f}_{\mathscr{L}^{(p_{n+1})'}} \\
+ \Cc(\delta, T, p_{n}, |\La|) \times \norm{f}_{\mathscr{L}^{(p_{n+1})'}} \times \int_{\hal t_1}^{t_1} (t_{1}-t)^{-\frac{\nn|\Lambda|}{2}\paren*{1 - \frac{1}{p_{1}}}} \paren*{ 1 + (t_{1}-t)^{-\frac{1}{2}} } \dd t.
\end{multline*}
Our choice of $p_{1}$ guarantees that this last integral converges, and we obtain:
\begin{equation*}
\int_{t_0}^{t_1} \|\mathbb{E}_{\mathrm{P}(t)}\left[ \nabla \mathsf{G}_{t_{1}-t} f \right] \| \dd t \leq \Cc(\delta, T, p_{n+1}, |\La|) \times \norm{f}_{\mathscr{L}^{(p_{n+1})'}}.
\end{equation*}
Thus by duality, \eqref{eq:Linftyplarge} holds for $p_{n+1}$, which concludes the proof of the claim by induction.
\end{proof}

\begin{remark}
\cite[Proof of Lem.~10.7, Step 2.]{mei2018mean} performs a similar iteration in the Euclidean case. It yields, for all $p > 1$, an $\mathscr{L}^{\infty, p}_{\mathrm{loc}}((0, + \infty), M^\La)$ bound on the local densities depending only on $p$ and on $|\La|$. However, since this bound grows with $p$, it is unclear to us how one would directly deduce that local densities are in $\mathscr{L}^{\infty, \infty}_{\mathrm{loc}}((0, + \infty), M^\La)$ as claimed in \cite[Proof of Lemma~10.7, end of Step 2]{mei2018mean}, although a posteriori it is indeed the case (because the densities are found to be continuous).

 This is however not a big issue, because it is enough to work with $p$ arbitrarily large, see below. In fact, it is even preferable to do so, because the case $p = + \infty$ of the potential estimate used in \cite[(10.51)]{mei2018mean} is not covered in the reference \cite{zbMATH03277871} and does not seem to be true.
\end{remark}

The rest of the proof differs from the approach of \cite{Jordan_1998,mei2018mean}.

\subsubsection*{Step 3. Sobolev regularity of local densities}
We start again from Duhamel's principle, which we recall here for convenience:
\begin{equation}
\label{DuhamelCC}
  \mathbb{E}_{\mathrm{P}_{\Lambda}(t_{1})}\bracket*{ f } = \mathbb{E}_{\mathrm{P}_{\Lambda}(t_{0})}\bracket*{\mathsf{G}_{t_{1} - t_{0}} f} - \int_{t_0}^{t_{1}} \mathbb{E}_{\mathrm{P}_{\Lambda}(t)} \bracket*{ \paren*{\nabla \mathsf{G}_{t_{1}-t} f} \cdot \mathbb{E}_{\mathrm{P}(t)} \bracket*{ \nabla \HH(\mathbb{Z}^{\d} \to \Lambda) \given \Lambda }  } \mathtt{d} t.
\end{equation}
In the previous step, in order to prove that local densities are in $\mathscr{L}^{p}$, we have used a light argument, working within a fixed $\Lambda$ and applying a rough bound on the conditional expectation of the form $\|\mathbb{E}_{\mathrm{P}(t)} \bracket*{ \nabla \HH(\mathbb{Z}^{\d} \to \Lambda) \given \Lambda }\| \leq \sup_{\bx \in \Conf} \|\nabla \HH(\mathbb{Z}^{\d} \to \Lambda)(\bx)\| \leq \Cc(|\La|)$.

To obtain Sobolev regularity, we need to use finer arguments.

Using the definition \eqref{espCondnablaHZdLa} and the fact that $\P$ admits local densities, we have, for $i \in \La$:
\begin{equation}
\label{prevw}
  \mathbb{E}_{\mathrm{P}(t)} \bracket*{ \nabla \HH(\mathbb{Z}^{\d} \to \Lambda) \given \Lambda }(\bx)_{i} = \nabla \HH_\La(\bx)_i + \frac{1}{\mathrm{p}_{\Lambda}(t,\bx)} \sum_{j \in \mathbb{Z}^{\d} \setminus \Lambda} J_{i,j} \int_M \partial_{1} \W(\bx_{i}, y) \mathrm{p}_{\Lambda \cup \{j\}}(t, \bx , y) \dd \omega(y),
\end{equation}
where we recall that $\partial_{1} \W$ is the gradient with respect to the first coordinate of $\W$. From this expression it follows that, defining for all $\bx \in \Conf(\Lambda),\, y \in M$ and $j \in \Zd \setminus \La$, the vectors
\begin{equation}
\label{def:vw}
 v_{\Lambda}(\bx) := \left(\nabla \V(\bx_{i}) + \nabla \HH_\La(\bx)_i\right)_{i \in \La} \quad w_{\Lambda \cup \{j\}}(\bx,y)  \coloneq \left(J_{ij} \partial_{1} \W(\bx_{i}, y)\right)_{i \in \La},
\end{equation}
we can re-write the integrand in \eqref{DuhamelCC} as:
\begin{multline}\label{eq:regularity:rewriting-integrand}
    \mathbb{E}_{\mathrm{P}_{\Lambda}(t)} \bracket*{ \nabla \mathsf{G}_{t_{1} - t} f \cdot \mathbb{E}_{\mathrm{P}(t)} \bracket*{ \nabla \HH(\mathbb{Z}^{\d} \to \Lambda) \given \Lambda }  } = \int_{M^{\Lambda}}  \nabla \mathsf{G}_{t_{1}-t}f(\bx) \cdot v_{\Lambda}(\bx)  \mathrm{p}_{\Lambda}(t,\bx) \dd \omega_{\Lambda}(\bx) \\ + \sum_{j \in \Zd \setminus \Lambda} \int_{M^{\Lambda \cup \{j\}}} \nabla \mathsf{G}_{t_{1}-t} f(\bx) \cdot  w_{\Lambda \cup \{j\}}(\bx,y) \mathrm{p}_{\Lambda \cup \{j\}}(t,\bx,y) \dd \omega_{\Lambda}(\bx) \dd \omega(y).
\end{multline}
The non-locality of the equations appears now clearly. Nonetheless, we can prove the following.

\begin{claim}[Bootstrap in regularity]
\label{claim:MainBootstrap}
Assume that there exists $s \in [0,2]$ such that for all $\Lambda \Subset \Zd$ and all $t > 0$, the local density $p_{\Lambda}(t) \in \mathscr{W}^{s,p}$, and that for all $0 < \delta < T$:
  \begin{equation*}
    \sup_{t \in [\delta,T]} \norm{\mathrm{p}_{\Lambda}(t)}_{\mathscr{W}^{s,p}} \leq \Cc(\delta, T, s, p, |\Lambda|).
  \end{equation*}
Then for all $\varepsilon \in (0,1)$, for all $\Lambda \Subset \Zd$ and for all $t > 0$, the local density $p_{\Lambda}(t) \in \mathscr{W}^{s+ 1 - \varepsilon, p}$. Moreover:
  \begin{equation}
  \label{SoboLocBound}
    \sup_{t \in [\delta,T]} \norm{\mathrm{p}_{\Lambda}(t)}_{\mathscr{W}^{s + 1 -\varepsilon,p}} \leq \Cc'(\delta, T, s, \varepsilon, p, |\Lambda|).
  \end{equation}
\end{claim}

\newcommand{\GG}{\mathsf{G}}
\begin{proof}
We start again from Duhamel's formula, applied to $f \coloneq (-\Delta_\V)^{\frac{s+\varepsilon}{2}}\varphi$ for some smooth function $\varphi$. Inserting \eqref{eq:regularity:rewriting-integrand} into \eqref{DuhamelCC}, we get 
\begin{multline}
\label{DuhaBoot}
  \mathbb{E}_{\mathrm{P}_{\Lambda}(t_{1})}\bracket*{ f }   = \mathbb{E}_{\mathrm{P}_{\Lambda}(t_{0})}\bracket*{\mathsf{G}_{t_{1} - t_{0}} f} - \int_{t_0}^{t_{1}} \int_{M^{\Lambda}}  \nabla \mathsf{G}_{t_{1}-t}f(\bx) \cdot v_{\Lambda}(\bx)  \mathrm{p}_{\Lambda}(t,\bx) \dd \omega_{\Lambda}(\bx) \dd t \\ +  \int_{t_0}^{t_{1}}  \sum_{j \in \Zd \setminus \Lambda} \int_{M^{\Lambda \cup \{j\}}} \nabla \mathsf{G}_{t_{1}-t} f(\bx) \cdot  w_{\Lambda \cup \{j\}}(\bx,y) \mathrm{p}_{\Lambda \cup \{j\}}(t,\bx,y) \dd \omega_{\Lambda}(\bx) \dd \omega(y) \dd t.
\end{multline}
The three terms in the right-hand side can be controlled in a similar fashion, and we focus here on the last one, which reads:
\begin{multline*}
\int_{t_0}^{t_{1}}  \sum_{j \in \Zd \setminus \Lambda} \int_{M^{\Lambda \cup \{j\}}} \nabla \mathsf{G}_{t_{1}-t} f(\bx) \cdot  w_{\Lambda \cup \{j\}}(\bx,y) \mathrm{p}_{\Lambda \cup \{j\}}(t,\bx,y) \dd \omega_{\Lambda}(\bx) \dd \omega(y) \dd t \\
=  \sum_{j \in \Zd \setminus \Lambda} \int_{t_0}^{t_{1}}  \int_{M^{\Lambda \cup \{j\}}} \nabla (-\Delta_\V)^{\frac{s+\varepsilon}{2}} \mathsf{G}_{t_{1}-t} \varphi(\bx) \cdot  w_{\Lambda \cup \{j\}}(\bx,y) \mathrm{p}_{\Lambda \cup \{j\}}(t,\bx,y) \dd \omega_{\Lambda}(\bx) \dd \omega(y) \dd t. 
\end{multline*}
Since by assumption $\W \in \mathscr{C}^{3}(M \times M)$, we have $w_{\Lambda \cup \{j\}} \in \mathscr{W}^{2,\infty} \subset \mathscr{W}^{s,\infty}$ since $s \leq 2$ (see \eqref{def:vw} for $w$). Thus for all $t > 0$ we have $ w_{\Lambda \cup \{j\}} \mathrm{p}_{\Lambda \cup \{j\}}(t) \in \mathscr{W}^{s,p}$ with
  \begin{equation*}
     \norm{ w_{\Lambda \cup \{j\}}  \mathrm{p}_{\Lambda \cup \{j\}}(t)}_{\mathscr{W}^{s,p}} \leq \norm{w_{\Lambda \cup \{j\}}}_{\mathscr{W}^{s,\infty}} \norm{\mathrm{p}_{\Lambda \cup \{j\}}(t)}_{\mathscr{W}^{s,p}},
  \end{equation*}
  the norm $\norm{\mathrm{p}_{\Lambda \cup \{j\}}(t)}_{\mathscr{W}^{s,p}}$ being finite by assumption. The generalized Hölder inequality gives, for all $j \in \Zd \setminus \La$ and all $t \in [t_0, t_1]$:
  \begin{multline*}
    \abs*{\int_{M^{\Lambda \cup \{j\}}} \nabla (-\Delta_\V)^{\frac{s+\varepsilon}{2}} \mathsf{G}_{t_{1}-t} \varphi \cdot w_{\Lambda \cup \{j\}}(\bx, y) \mathrm{p}_{\Lambda \cup \{j\}}(t)(\bx, y)  \mathtt{d} \omega_{\Lambda \cup \{j\}}(\bx,y) } \\
    \leq \norm{\nabla (-\Delta_\V)^{\frac{s+\varepsilon}{2}} \mathsf{G}_{t_{1} - t} \varphi}_{\mathscr{W}^{-s,p'}} \times  \norm{w_{\Lambda \cup \{j\}}}_{\mathscr{W}^{s,\infty}} \times  \norm{\mathrm{p}_{\Lambda \cup \{j\}}(t)}_{\mathscr{W}^{s,p}} .
  \end{multline*}
The norm $\norm{w_{\Lambda \cup \{j\}}}_{\mathscr{W}^{s,\infty}}$ is bounded by $|J_{i,j}|$ times a constant depending only on the model (and on $|\La|$), and $\norm{\mathrm{p}_{\Lambda \cup \{j\}}(t)}_{\mathscr{W}^{s,p}}$ is bounded by $\Cc(\delta, T, s, p, |\La| + 1)$ by assumption. On the other hand, we have:
  \begin{equation*}
    \norm{\nabla (-\Delta_\V)^{\frac{s+\varepsilon}{2}} \mathsf{G}_{t_{1} - t} \varphi}_{\mathscr{W}^{-s,p'}} \leq \norm{\mathsf{G}_{t_{1} - t} \varphi}_{\mathscr{W}^{1+\varepsilon,p'} } \leq c \norm{\varphi}_{\mathscr{L}^{p'}} \paren*{ 1 + (t_{1} - t)^{-1/2}}^{1+\varepsilon}.
  \end{equation*}
For the first inequality, observe that we are controlling $1 + 2 \times \frac{s + \varepsilon}{2} - s$ derivatives of $\varphi$ in $\mathscr{L}^{p'}$, the second inequality is an instance of\footnote{Take $s = 1 +\epsilon$, $p = 1$, $r = q =  p'$ with the notation of Theorem \ref{YoungHeat}.} \eqref{eq:heat-semigroup:young}. Since $\varepsilon < 1$, the last expression in $t$ is integrable near $t_{1}$ (this wouldn't be the case for $\epsilon =1$, which is the reason why we cannot quite get one full level of regularity at each step of the bootstrap). We can now sum these estimates over $j \in \Zd \setminus \Lambda$ and obtain for $\delta \leq t_0 < t_1 < T$:
\begin{equation*}
 \sum_{j \in \Zd \setminus \Lambda}  \int_{t_0}^{t_{1}} \int_{M^{\Lambda \cup \{j\}}} \nabla \mathsf{G}_{t_{1}-t} f(\bx) \cdot  w_{\Lambda \cup \{j\}}(\bx,y) \mathrm{p}_{\Lambda \cup \{j\}}(t,\bx,y) \dd \omega_{\Lambda}(\bx) \dd \omega(y) \dd t \leq \Cc'(\delta, T, s, \varepsilon, p, |\La|) \|\varphi\|_{\mathscr{L}^{p'}}.
\end{equation*}
Returning to \eqref{DuhaBoot}, we thus get $\mathbb{E}_{\mathrm{P}_{\Lambda}(t_{1})}\bracket*{ f }  \leq \Cc'(\delta, T, s, \varepsilon, p, |\La|) \|\varphi\|_{\mathscr{L}^{p'}}$ with $f = (-\Delta_\V)^{\frac{s+\varepsilon}{2}}\varphi$, which yields that the local density $\p_\La(t_1) \in \mathscr{W}^{s + \varepsilon,p}$ and \eqref{SoboLocBound}.
\end{proof}

\newcommand{\Ww}{\mathscr{W}}
Combining Claim \ref{claim:Lpforallp} and Claim \ref{claim:MainBootstrap}, we deduce that the local densities are in $\Ww^{s, p}$ for all $s < 3$ and all $p \in [1, +\infty)$. By embedding Sobolev spaces into Hölder spaces, as per \eqref{eq:sobolev:embedding:holder}, we get that the local densities have Hölder regularity $\CC^{2, 1-\epsilon}$ in space for all $\epsilon > 0$.
\begin{corollary}
\label{coro:C2epsilon}
For all $\Lambda \Subset \Zd$ and all $t > 0$, the local density $p_{\Lambda}(t)$ is in $\CC^{2, 1-\epsilon}(M^\La)$ for all $\epsilon > 0$. Moreover, for all $0 < \delta < T$, we have:
\begin{equation}
\label{C2locunif}
    \sup_{t \in [\delta,T]} \norm{\mathrm{p}_{\Lambda}(t)}_{\CC^{2, 1-\epsilon}} \leq \Cc'(\delta, T, \varepsilon, |\Lambda|).
\end{equation}
\end{corollary}
In particular, the local densities $\p_\La$ \emph{are} indeed in $\mathscr{L}^{\infty}(M^\La)$, locally uniformly in time. It remains to establish the part regarding the regularity in time of Theorem \ref{theo:regul}.

\subsubsection*{Step 4. Time regularity of local densities}
Regularity in space of the local densities allows us to state a “strong in space, weak in time” version of Duhamel's principle \eqref{eq:duhamel}, indeed for all $\La \Subset \Zd$, for all $\bx_0 \in M^\La$ and for all $0 \leq t_0 < t_1$, we have:
\begin{equation}\label{eq:duhamelStrongWeak}
  \begin{split}
    \p_\La(t_1, \bx_0) &= \int_{M^\La} \gL(t_1 -t_0, \bx_0, \bx) \p_\La(t_0, \bx) \dd \oL(\bx) 
                       \\& - \int_{t_0}^{t_1} \int_{M^\La} \nabla \gL(t_1-t,\bx_0, \bx) \cdot \left(v_\La(\bx) \p_\La(t, \bx)\right) \dd \oL(\bx)  \dd t
                       \\& - \sum_{j \in \Zd \setminus \La} \int_{t_0}^{t_1} \int_{M^{\La\cup\{j\}}} \nabla \gL(t_1-t,\bx_0, \bx) \cdot \left(w_{\La\cup\{j\}}(\bx, y) \p_{\La\cup \{j\}}(t, \bx, y) \right)  \dd \omega_{\La \cup\{j\}}(\bx,y) \dd t,
  \end{split}
\end{equation}
where $\gL$ is the heat kernel on $M^\La$ (see Appendix \ref{s:heat-kernel}) and $v, w$ are as in \eqref{def:vw}.

\begin{claim}
  For all $\Lambda \Subset \Zd$, and all $\bx \in M^\La$, the local density $\mathrm{p}_{\Lambda}(\cdot, \bx)$ is in $\CC^1(0, + \infty)$. Moreover, for all $0 < \delta < T$ we have
  \begin{equation*}
\sup_{t \in [\delta, T]} \|\partial_t \mathrm{p}_{\Lambda}(t, \cdot)\|_{\mathscr{L}^\infty} \leq \Cc(\delta, T, |\La|).
  \end{equation*}
\end{claim}
\begin{proof}
The first term in the right-hand side on \eqref{eq:duhamelStrongWeak} is smooth in time and space, since $t_0 < t_1$.
Hence, we focus on the third one; the second one being treated similarly.
Let us fix $j \in \Zd \setminus \La$. First, an integration by parts with respect to the $\bx$ variable gives:
\begin{multline*}
\int_{t_0}^{t_1} \int_{M^{\La\cup\{j\}}} \nabla \gL(t_1-t,\bx_0, \bx) \cdot \left(\p_{\La\cup \{j\}}(t, \bx, y) w_{\La\cup\{j\}}(\bx, y)\right)   \dd \omega_{\La \cup\{j\}}(\bx,y) \dd t \\
= - \int_{t_0}^{t_1} \int_{M^{\La}} \int_{M}  \gL(t_1-t,\bx_0, \bx)\  \div_{\bx} \left(\p_{\La\cup \{j\}}(t, \bx, y) w_{\La\cup\{j\}}(\bx, y)\right)   \dd \oL(\bx) \dd \omega(y) \dd t.
\end{multline*}
From the previous analysis, we know that the local densities have enough derivatives for these integrals to make sense.
Then, we differentiate this expression with respect to $t_1$, which yields, using the fact that $\gL(0, \bx_0, \cdot) = \delta_{\bx_0}$:
\begin{equation}\label{partialit_1}
  \begin{split}
    & \frac{\dd}{\dd t_1} \int_{t_0}^{t_1} \int_{M^{\La}} \gL(t_1-t,\bx_0, \bx) \int_M  \div_{\bx} \left(\p_{\La\cup \{j\}}(t, \bx, y) w_{\La\cup\{j\}}(\bx, y)\right) \dd \omega(y) \dd \oL(\bx)  \dd t \\
    = & \int_{M}  \div_{\bx} \left(\p_{\La\cup \{j\}}(t_1, \bx_0, y) w_{\La\cup\{j\}}(\bx_0, y)\right) \dd \omega(y) \\
    + &\int_{t_0}^{t_1} \int_{M^{\La}} \int_M  \partial_{t_1} \gL(t_1-t,\bx_0, \bx)\   \div_{\bx} \left(\p_{\La\cup \{j\}}(t, \bx, y) w_{\La\cup\{j\}}(\bx, y)\right)  \dd \oL(\bx) \dd \omega(y) \dd t.
  \end{split}
\end{equation}
We bound the first term in the right-hand side by $\|\p_{\La\cup \{j\}}(t_1)\|_{\mathscr{C}^1} \times \|w_{\La\cup\{j\}}\|_{\mathscr{C}^1}$, and by Corollary~\ref{coro:C2epsilon} regarding the regularity of $\p_\La$, the definition \eqref{def:vw} of $w$ and our short-range assumption, we get for $\delta \leq t_1 \leq T$:
\begin{equation*}
\sum_{j \in \Zd \setminus \La} \left| \int_{M}  \div_{\bx} \left(\p_{\La\cup \{j\}}(t_1, \bx_0, y) w_{\La\cup\{j\}}(\bx_0, y)\right) \dd \omega(y) \right| \leq \Cc(\delta, T, |\La|).
\end{equation*}
We now control the second term in the right-hand side of \eqref{partialit_1}.
We have, using first the definition of the heat kernel, then an integration by parts:
\begin{equation*}
  \begin{split}
& \int_{t_0}^{t_1} \int_{M^{\La\cup\{j\}}}  \partial_{t_1} \gL(t_1-t,\bx_0, \bx)\   \div_{\bx} \left(\p_{\La\cup \{j\}}(t, \bx, y) w_{\La\cup\{j\}}(\bx, y)\right)   \dd \omega_{\La \cup\{j\}}(\bx,y) \dd t \\
    =& \int_{t_0}^{t_1} \int_{M^{\La\cup\{j\}}}  \Delta_\V \gL(t_1-t,\bx_0, \bx)\   \div_{\bx} \left(\p_{\La\cup \{j\}}(t, \bx, y) w_{\La\cup\{j\}}(\bx, y)\right)   \dd \omega_{\La \cup\{j\}}(\bx,y) \dd t \\
    =& - \int_{t_0}^{t_1} \int_{M^{\La\cup\{j\}}}  \nabla \gL(t_1-t,\bx_0, \bx) \cdot \nabla  \div_{\bx} \left(\p_{\La\cup \{j\}}(t, \bx, y) w_{\La\cup\{j\}}(\bx, y)\right)   \dd \omega_{\La \cup\{j\}}(\bx,y) \dd t.
  \end{split}
\end{equation*}
Since the local densities are known to be in $\CC^2$ (and \eqref{C2locunif} holds), and the potential $\Psi$ was assumed to be $\CC^3$, we have:
\begin{equation*}
\| \nabla  \div_{\bx} \left(\p_{\La\cup \{j\}}(t, \bx, y) w_{\La\cup\{j\}}(\bx, y)\right) \| \leq \Cc(\delta, T, |\La|) \times |J_{0,j}|.
\end{equation*}
On the other hand, the heat kernel satisfies $\int_{t_0}^{t_1} \int_{M^{\La}}  \|\nabla \gL(t_1-t,\bx_0, \cdot)\| \leq \Cc(\delta, T, |\La|)$, which concludes.
\end{proof}

\section{Gradient flow construction in infinite-volume}
\label{sec:JKO}
\newcommand{\Pkmh}{\P^{k-1, h}}
\newcommand{\Pkh}{\P^{k, h}}
\newcommand{\Pkn}{\overline{\P}^{k, h}}
\newcommand{\Pknt}{\overline{\P}^{k, h}_{\tau}}

\newcommand{\pkmh}{\p^{k-1, h}}
\newcommand{\pkh}{\p^{k, h}}
\newcommand{\pkn}{\overline{\p}^{k, h}}
\newcommand{\hPkn}{\hat{\P}_{n}^{k, h}}
\newcommand{\Pkmn}{\P^{k-1, h}_{|n}}

\newcommand{\Phit}{\Phi_\tau}
\newcommand{\pkt}{\overline{p}^{k}_{\tau}}
\newcommand{\pk}{\overline{p}^{k}}
\newcommand{\hz}{\hat{\zeta}}
\newcommand{\hpkn}{\hat{\p}^{(k)}}
\newcommand{\Confl}{{\Conf_\ell}}

For a given functional $\F$ on a metric space $(\mathcal{X}, d)$, a general algorithm to approximate the gradient flow of $\F$ (starting at $x_0 \in \mathcal{X}$) by discrete steps is to iterately solve (with $h > 0$ the “step-size”):
\begin{equation}
\label{MMS}
x_{k+1} \in \argmin \left(\hal d^2(x_k, \cdot) + h \F(\cdot) \right),
\end{equation}
a method known as a \emph{Minimizing Movement Scheme}, see e.g. \cite[Sec. 2]{santambrogio2017euclidean}. Sending $h \to 0$, one formally recovers a gradient descent of the form $\frac{\dd x}{\dd t}(t) = - \nabla \F(x(t))$. See for instance \cite{ambrosio2005gradient}, \cite[Chap. 23]{VillaniOldNew}, \cite{santambrogio2017euclidean} for more details.

This scheme is implemented in \cite{Jordan_1998} in the case where $\mathcal{X}$ is the space of probability measures on $\R^\nn$ endowed with the Wasserstein distance, and $\F$ is a finite-volume free energy functional of the form $\F(\mu) := \int \mu \log \mu + \int \V \mu$.
Defining the scheme is not difficult in itself, the two main tasks are:
\begin{enumerate}
  \item Proving, usually by some kind of compactness argument, that the discrete trajectories have a well-defined continuous limit.
  \item Showing that this limiting trajectory satisfies the Fokker--Planck equation. This requires to understand the minimality condition \eqref{MMS} satisfied at each step of the discrete scheme, and passing this information to the limit.
\end{enumerate}

Here, we work on the space $\Ppis$ endowed with the specific Wasserstein distance $\Wi$ defined in Section \ref{sec:Wasserstein}. A significant difference compared to the general algorithm mentioned above is that we go through finite-dimensional restrictions, solve a variational problem similar to \eqref{MMS} in finite dimension, and then return to the infinite-dimensional, stationary setting using the stationarization procedure $\Stat$ of Section \ref{sec:stationarisation}, as in the following diagram:
\[
\begin{tikzcd}[column sep=14em, row sep=2em]
(\Ppis, \Wi) \arrow[r, dashed, "\text{One step of our discrete scheme}"] \arrow[d, "\text{Restriction to $\Lan$}"] & (\Ppis, \Wi) \\
(\Pp(\Conf_n), \Wn) \arrow[r, "\text{Finite volume variational problem}"'] & (\Pp(\Conf_n), \Wn) \arrow[u, "\text{Stationarization in $\Lan$}"']
\end{tikzcd}
\]

\subsection{The variational scheme}
\label{sec:discrete_scheme}
We construct a sequence of \emph{stationary} spin measures corresponding to a time-discretization of the gradient flow of the free energy. 

\paragraph{The step-size.}
For each value of the step-size $h > 0$, we choose $n$ large enough such that:
\begin{enumerate}
	\item For all $\P \in \Ppis$, we have as in Lemma \ref{lem:PropertiesOfTheLimits}:
  \begin{equation}
  \label{eq:TheLimitsh}
\Fbeta(\P) \geq \frac{1}{|\Lan|} \Fn(\P) - h.
  \end{equation}
	\item For all $\Pn$ in $\Ppn$, we have as in Lemma \ref{lem:ppy_stationary}:
  \begin{equation}
  \label{eq:Stationaryh}
\Fbeta(\Stat_n(\P)) \leq \frac{1}{|\Lan|}  \Fn(\Pn) + h.
  \end{equation}
  \item $\frac{1}{nh} \to 0$ as $h \to 0$ (so in particular $n \to \infty$ when $h \to 0$).
\end{enumerate}

\paragraph{Definition of the scheme.} 
Fix $\Pz$ a stationary spin measure such that $\Fbeta(\Pz) < + \infty$, which will serve as our initial condition.
We fix $h > 0$, and we iteratively define a sequence $(\Pkh)_{k \geq 0}$ of \emph{stationary} spin measures as follows:

\begin{enumerate}
	\item We set $\P^{0, h} := \Pz$.
 \item Assume that $\Pkmh$ has been constructed for some $k \geq 1$, then:
\begin{enumerate}
	\item We let $\Pkn$ be the unique minimizer (see Corollary \ref{cor:convexity}) over $\Ppn$ of:
	\begin{equation}
	\label{finite_n_problem}
		\P \mapsto \frac{1}{2} \Wn^2(\Pkmh, \P) + h \Fn(\P),
	\end{equation}
  where $\Wn$ is the Wasserstein distance on $\Ppn$.
	We emphasize that, at this point, we obtain a \emph{finite} spin measure supported on $\Confn$.
	\item We build the stationary version $\Stat_n(\Pkn)$ of $\Pkn$ with averages over $\Lan$ as in Section~\ref{sec:stationarisation}, and we set $\Pkh := \Stat_n(\Pkn)$, which is by construction a \emph{stationary} spin measure.
\end{enumerate}
\end{enumerate}

\paragraph{A priori estimates.}
\begin{lemma}
\label{lem:apriori}
  For all $h > 0$ and $k \geq 1$, the free energy $\fbeta(\Pkh)$ is finite.
  More precisely, for all $T > 0$ we have:
  \begin{equation}\label{eq:jko:sup-free-energy}
    \sup_{h > 0} \sup_{0 \leq k \leq \frac{T}{h}} \fbeta(\Pkh) \leq \Cc(\mathrm{P}^{0}, T).
  \end{equation}
  Moreover, we have the following control on the Wasserstein distances appearing when solving \eqref{finite_n_problem}:
  \begin{equation}\label{eq:jko:sup-wasserstein}
    \sup_{h > 0} \frac{1}{h} \sum_{k=1}^{\frac{T}{h}} \frac{1}{|\Lambda_{n}|} \mathcal{W}_{n}^{2}(\mathrm{P}^{k-1,h}, \Pkn) \leq \Cc(\mathrm{P}^{0}, T).
  \end{equation}
\end{lemma}

\begin{proof}
At each step of the scheme, we can take $\mathrm{P} := \Pkmh_{|\Lan}$ as a competitor in \eqref{finite_n_problem}, which yields 
\begin{equation*}
\Fn(\Pkn) \leq \Fn(\Pkmh) - \frac{1}{2h} \Wn^2(\Pkmh, \Pkn) \leq \Fn(\Pkmh).
\end{equation*}
Using \eqref{eq:TheLimitsh} and \eqref{eq:Stationaryh} we obtain:
\begin{equation*}
  \fbeta(\Pkh) \leq \frac{1}{|\Lambda_{n}|} \Fn(\Pkn) + h \leq \frac{1}{|\Lambda_{n}|} \Fn(\mathrm{P}^{k-1,h}) + h \leq \fbeta(\mathrm{P}^{k-1,h}) + 2 h.
\end{equation*}
By induction, we get $\fbeta(\Pkh) \leq \fbeta(\mathrm{P}^{0}) + 2kh$ for all $k \geq 0$, which gives \eqref{eq:jko:sup-free-energy}.

\medskip

Using again $\mathrm{P} := \mathrm{P}^{k-1,h}_{|\Lan}$ as a competitor in \eqref{finite_n_problem}, we see that:
\begin{equation*}
  \frac{1}{2h|\Lambda_{n}|} \mathcal{W}_{n}^{2}(\mathrm{P}^{k-1,h}, \Pkn) \leq \frac{1}{|\Lambda_{n}|} \paren*{ \Fn(\mathrm{P}^{k-1,h}) - \Fn(\Pkn) } \leq \left(\fbeta(\mathrm{P}^{k-1,h}) - \fbeta(\Pkh)\right) + 2h,
\end{equation*}
using again \eqref{eq:TheLimitsh} and \eqref{eq:Stationaryh}. The telescopic sum yields
\begin{equation*}
\frac{1}{2h} \sum_{k=1}^{\frac{T}{h}} \frac{1}{|\Lan|} \Wn^2(\Pkmh, \Pkn) \leq \Fbeta(\Pz) - \Fbeta(\P^{\frac{T}{h}, h}) + 2 T, 
\end{equation*}
Since $\Fbeta$ is bounded below, and $\Fbeta(\Pz)$ is assumed to be finite, we get \eqref{eq:jko:sup-wasserstein}.
\end{proof}

\subsection{Variational properties of the finite-volume minimizer}
The next proposition is an important step: it states an approximate, discrete Kolmogorov equation satisfied by the trajectory of our discrete scheme.

For all $L \geq 1$, and $k \geq 0$, let us write $\mathrm{p}^{k,h}_{L}$ for the local density of $\Pkh$ in $\La_L$, which exists because the specific entropy $\fbeta(\Pkh)$ is always finite by Lemma \ref{lem:apriori}.
\begin{proposition}
\label{prop:JKO}
  Let $\ell $, $L$ with $1 \leq \ell \leq L \leq n$, and let $\varphi$ be a smooth, $\Lal$-local test function. We have, for all $k \geq 1$:
\begin{multline}
\label{eq:Discrete_FP}
\left| \frac{1}{h} \left( \Esp_{\Pkh}\left[  \varphi \right] -  \Esp_{\Pkmh}\left[  \varphi \right] \right) +  \Esp_{\Pkh} \left[ - \Delta_U \varphi  + \nabla \HH(\La_L \to \Lal) \cdot \nabla \varphi \right] \right|
\\ \leq \Cc(\varphi, \ell) \left( \frac{1}{h} \frac{1}{|\Lan|} \Wn^2(\Pkmh, \Pkn)+ \frac{1}{h} \frac{L}{n} + o_{L}(1) \right),
\end{multline}
where the term $o_{L}(1)$ tends to $0$ as $L \to \infty$ with $\ell$ and $\varphi$ fixed.
\end{proposition}
The proof of Proposition \ref{prop:JKO} follows the lines of \cite[Sec. 5]{Jordan_1998}, with some adaptations. At the conceptual level, our approach is very similar to theirs: we derive Euler--Lagrange equations expressing the minimality of $\Pkn$ in \eqref{finite_n_problem}.
In our setting we have to deal with the following additional difficulties:
\begin{itemize}
  \item The interactions are non-local. We rely on the short-range assumption \eqref{eq:short-range} to bypass this difficulty by working in a sufficiently large box.
    Hence, the introduction of the length scale $L$.
  \item We have an additional stationarization step, which could disturb the good variational properties of $\Pkn$. We rely on Lemma \ref{lem:ppy_stationary} to control the changes induced by the stationarization.
\end{itemize}

\begin{proof}[Proof of Proposition \ref{prop:JKO}] \
\subsubsection*{Step 1. Some perturbative computations.}
Let $\xi$ be a smooth vector field on $M^{\Lan}$. Since $M^{\Lan}$ is compact, $\xi$ defines a global flow $\{\Phit\}_{\tau \in \R}$ satisfying the ODE $\partial_\tau \Phit = \xi \circ \Phit$. As $\tau \to 0$ we have for all $\bx \in M^{\Lan}$:
\begin{equation}
\label{Phit_petittau}
\Phit(\bx) = \exp_x(\tau \xi(\bx)) + o(\tau).
\end{equation}
For all $\tau \in \R$, the map $\Phit$ is a $\CC^\infty$ diffeomorphism of $M^{\Lan}$, $\Phi_0$ being the identity map. We may thus define a \emph{perturbed} spin measure $\Pknt$ on $\Confn$ as the push-forward of $\Pkn$ by $\Phit$. We use $\Pknt$ as a competitor to $\Pkn$ in the minimization problem \eqref{finite_n_problem} and derive Euler--Lagrange equations from there.

By definition of the push-forward measure, for all test functions $f \in \CC^0(M^{\Lan})$ we have: 
\begin{equation*}
\Esp_{\Pknt} \left[ f \right] = \Esp_{\Pkn} \left[  f \circ \Phit \right].
\end{equation*}

The following claim is analogous to \cite[Eq.~(37)]{Jordan_1998} in a slightly different context.
\begin{claim}[Perturbation of the spin interactions]
\label{claim:JKO_spin_spin}
The map $\tau \mapsto \mathcal{H}_{n}(\Pknt)$ is differentiable at $0$, with
  \begin{equation}\label{eq:jko:pertubartive:energy} 
    \frac{\mathtt{d}}{\mathtt{d} \tau}_{| \tau = 0}  \mathcal{H}_{n}(\Pknt) = \Esp_{\Pkn} \left[ \nabla \HH_\Lan \cdot \xi \right] 
\end{equation}
\end{claim}
\begin{proof}[Proof of Claim \ref{claim:JKO_spin_spin}]
We use \eqref{Phit_petittau} and a Taylor's expansion of $\HH_\Lan$.
\end{proof}

The following is analogous to \cite[eq. (38)]{Jordan_1998} and we omit the proof, which relies on the change of variable formula for densities and an expansion for the determinant near the identity map.
\begin{claim}[Perturbation of the relative entropy]
\label{claim:JKO_entropy}
The map $\tau \mapsto \mathcal{E}_{n}(\Pknt)$ is differentiable at $0$ with:
\begin{equation}
\label{eq:jko:pertubartive:entropy} 
  \frac{\mathtt{d}}{\mathtt{d} \tau}_{| \tau = 0} \mathcal{E}_{n}(\Pknt) = - \Esp_{\Pkn}\left[  \operatorname{div}_{\V} \xi  \right].
\end{equation}
\end{claim}

Finally, we need the following result, which is significantly more subtle in the Riemannian case than in the Euclidean case.
\begin{claim}[Perturbation of the Wasserstein distance] 
\newcommand{\Qq}{\mathrm{Q}}
\label{claim:PertubWass}
Let $\P, \Qq \in \Ppn$ be fixed and let $T : M^\Lan \to M^\Lan$ be the optimal transport map from $\Q$ to $\P$, write $T = \Exp*{ - \nabla \theta}$ as in Section \ref{sec:Wasserstein}. Then, we have:
  \begin{equation}
  \label{eq:jko:perturbative:wasserstein}
    \frac{1}{2} \mathcal{W}_{n}^{2}(\mathrm{Q}, \P_{\tau}) \leq \frac{1}{2} \mathcal{W}_{n}^{2}(\Q, \P) + \tau \int_{M^\Lan}\left[ \nabla \hal \dn^2(\sig, \cdot)\right]_{T(\sig)} \cdot \xi(T(\sig))\ \dd \mathrm{Q}(\bx) + o(\tau).
  \end{equation}
\end{claim}
\begin{proof}
By construction, the map $\Phi_\tau \circ T$ pushes $\Q$ forward onto $\P_\tau$, which allows us to bound the Wasserstein distance between $Q$ and $\P_\tau$ by writing:
\begin{equation*}
\mathcal{W}_{n}^{2}(\mathrm{Q}, \P_\tau) \leq \int_{M^\Lan} \dn^2(\bx, \Phi_\tau \circ T(\bx)) \dd \mathrm{Q}(\bx).
\end{equation*}
We thus get:
\begin{equation}
\label{UBW2QP}
\mathcal{W}_{n}^{2}(\mathrm{Q}, \P_\tau) - \mathcal{W}_{n}^{2}(\mathrm{Q}, \P)  \leq \int_{M^\Lan} \left(\dn^2(\bx, \Phi_\tau \circ T(\bx)) - \dn^2(\bx, T(\bx))\right)\dd \mathrm{Q}(\bx),
\end{equation}
and we would like to apply a Taylor's expansion to the integrand in the right-hand side, the problem being that the distance squared is not differentiable on the entire manifold (of course, such an issue does not appear in the Euclidean setting).

Let $\sig, \sig'$ be two spin configurations in $\Confn$ seen as two points on the manifold $M^\Lan$. If $\sig'$ is not in the cut-locus of $\sig$, then the function $\hal \dn^2(\bx, \cdot)$ is differentiable at $\sig'$, and thus using the definition of the flow $\Phit$ we have as $\tau \to 0$:
\begin{equation*}
\frac{1}{\tau} \left( \hal \dn^2(\sig, \Phit(\sig')) - \hal \dn^2(\sig, \sig') \right) \to \left[ \nabla \hal \dn^2(\sig, \cdot)\right]_{\sig'} \cdot \xi(\sig').
\end{equation*}
Fine properties of the optimal transportation map (see \cite[Thm 4.2]{cordero2001riemannian}) guarantee that \emph{for $\Q_{|\Conf_n}$-a.e. $\sig$}, the image $T(\sig)$ is \emph{not} in the cut-locus of $\sig$, and thus we have, $\Q$-almost surely, as $\tau \to 0$:
\begin{equation*}
\left( \hal \dn^2(\sig, \Phit \circ T(\sig)) - \hal \dn^2(\sig, T(\sig)) \right) =  \tau \left[ \nabla \hal \dn^2(\sig, \cdot)\right]_{T(\sig)} \cdot \xi(T(\sig)) + o(\tau). 
\end{equation*}
On the other hand, the function $\hal \dn^2(\bx, \cdot)$ is always Lipschitz (globally on $M^\Lan$ and uniformly with respect to $\sig \in \Lan$), and $\|\Phit - \id \|_\infty$ is bounded by $\Cc \tau$ for some constant $\Cc$ depending on our choice of vector field, thus the quantity $\left( \hal \dn^2(\sig, \Phit(\sig')) - \hal \dn^2(\sig, \sig') \right)$ is a $\O(\tau)$ uniformly for $\sig, \sig'$ in $M^\Lan$ and $\tau$ near $0$. We may apply the dominated convergence theorem  and conclude that:
\begin{multline}
\label{LIMSUPCUT}
\limsup_{\tau \to 0} \frac{1}{\tau}  \int_{M^\Lan} \left(\dn^2(\bx, \Phi_\tau \circ T(\bx)) - \dn^2(\bx, T(\bx))\right)\dd \mathrm{Q}(\bx) \\
\leq \int_{M^\Lan}  \left[ \nabla \hal \dn^2(\sig, \cdot)\right]_{T(\sig)} \cdot \xi(T(\sig))\  \dd \Q(\sig).
\end{multline}
Combining \eqref{UBW2QP} with \eqref{LIMSUPCUT} proves the claim.
\end{proof}

We are about to take $\xi$ of the form $\xi = \nabla \varphi$ and in that case it is important to make the following observation.
\begin{claim}
\label{claim:calculNabla}
With the notation of the previous claim, if $\xi = \nabla \varphi$, we have:
\begin{equation*}
\left|\int_{M^\Lan}\left[ \nabla \hal \dn^2(\sig, \cdot)\right]_{T(\sig)} \cdot \nabla \zeta(T(\sig))\ \dd \mathrm{Q}(\bx) - \left(\Esp_{\P}[\zeta] - \Esp_{\Q}[\zeta] \right) \right| \leq \hal \|\nabla^2 \zeta\|_{\infty} \times \mathcal{W}_{n}^{2}(\mathrm{Q}, \P).
\end{equation*}
\end{claim}
\begin{proof}
Recall the following general fact (see e.g. \cite{cordero2001riemannian}, equation (18)): on a smooth, compact Riemannian manifold with distance $d$, if $y$ is not in the cut-locus of $x$, then:
\begin{equation}
\label{exponabla}
\exp_y \left( - \left[\nabla \hal d^2(x, \cdot) \right]_y \right) = x.
\end{equation}
On the other hand, by simple calculus, if $f$ is a smooth function on the manifold and $u$ a tangent vector at $y$, we have a Taylor's expansion:
\begin{equation}
\label{Taylorexp}
\left|f\left(\exp_y(u)\right) - f(y) + \nabla f(y) \cdot u \right| \leq  \hal \|\nabla^2 f\|_{L^\infty} \times d^2(y, \exp_y(u)).
\end{equation}
As we mentioned earlier, $\Q$-a.s. the image $T(\sig)$ is not in the cut-locus of $\sig$ and thus applying \eqref{exponabla} we see that:
\begin{equation*}
\zeta(T(\sig)) - \zeta(\sig) = \zeta(T(\sig)) - \zeta\left(\exp_{T(\sig)} \left( - \left[\nabla \hal \dn^2(\sig, \cdot) \right]_{T(\sig)} \right) \right),
\end{equation*}
and thus using \eqref{Taylorexp} we get the following expansion:
\begin{equation*}
\left| \zeta(T(\sig)) - \zeta\left( \bx \right) - \nabla \zeta(T(\sig)) \cdot \left[ \nabla \hal \dn^2(\sig, \cdot) \right]_{T(\sig)} \right| 
\leq \hal \|\nabla^2 \zeta\|_{\infty} \times \dn^2(\sig, T(\sig)).
\end{equation*}
Integrating this against $\Q$, we obtain:
\begin{multline*}
\left| \int_{M^\Lan} \zeta(T(\sig)) \dd \Q(\bx) - \int_{M^\Lan} \zeta\left( \bx \right) \dd \Q(\bx) -\int_{M^\Lan}   \nabla \zeta(T(\sig)) \cdot \left[ \nabla \hal \dn^2(\sig, \cdot) \right]_{T(\sig)} \dd \Q(\bx) \right| \\
\leq \hal \int_{M^\Lan} \|\nabla^2 \zeta\|_{\infty} \times \dn^2(\sig, T(\sig)) \dd \Q(\bx),
\end{multline*}
but by definition $T$ pushes $\Q$ onto $\P$ so the difference of the first two terms is exactly $\Esp_{\P}[\zeta] - \Esp_{\Q}[\zeta]$, and it is the optimal map thus the right-hand side is $\hal \|\nabla^2 \zeta\|_{\infty} \times \mathcal{W}_{n}^{2}(\mathrm{Q}, \P)$, which proves the claim.
 \end{proof}
Combining Claim \ref{claim:PertubWass} and Claim \ref{claim:calculNabla}, we obtain:
\begin{equation}
\label{eq:PerturbWassFull}
    \frac{1}{2} \mathcal{W}_{n}^{2}(\mathrm{Q}, \P_{\tau}) \leq \frac{1}{2} \mathcal{W}_{n}^{2}(\Q, \P) + \tau \left(\Esp_{\P}[\zeta] - \Esp_{\Q}[\zeta] \right) + \frac{|\tau|}{2} \|\nabla^2 \zeta\|_{\infty} \times \mathcal{W}_{n}^{2}(\mathrm{Q}, \P) + o(\tau).
\end{equation}

\subsubsection*{Step 2. An Euler--Lagrange equation for $\Pkn$.}
Choosing for $\xi$ a gradient vector field $\xi := \nabla \varphi$ for some smooth $\varphi$, and combining the perturbative estimates from the previous step, we obtain the following result, analogous to the fundamental estimate \cite[Eq.~(41)]{Jordan_1998} in the original JKO's scheme and understood as a discretized Fokker--Planck equation. 
\begin{lemma}
 Take $\varphi \in \mathscr{C}^{\infty}(\mathsf{Conf}_{n})$. For all $k \geq 1$,
  \begin{equation}\label{eq:jko:discrete-fokker-planck-tilde}
    \left| \frac{1}{h} \left( \Esp_{\Pkn}\left[  \varphi \right] -  \Esp_{\Pkmh}\left[  \varphi \right] \right) + \mathbb{E}_{\Pkn} \bracket*{ \nabla \HH_n \cdot \nabla \varphi - \Delta_{\V} \varphi }  \right| \leq \frac{1}{2h} \mathcal{W}^{2}_{n}(\Pkmh, \Pkn) \times \|\nabla^{2} \varphi\|_{\infty}.
  \end{equation}
\end{lemma}
\begin{proof} \newcommand{\bP}{\overline{\P}}
  For short, set $\Q := \Pkmh$ and $\bP := \Pkn$. Applying the perturbative estimates \eqref{eq:jko:pertubartive:energy}, \eqref{eq:jko:pertubartive:entropy}, and \eqref{eq:PerturbWassFull} with $\xi := \nabla \varphi$, we get:
\begin{multline*}
\left(\frac{1}{2} \mathcal{W}_{n}^{2}(\mathrm{P}, \bP_{\tau}) + h \Fn(\bP_{\tau})\right) - \left(\frac{1}{2} \mathcal{W}_{n}^{2}(\mathrm{P}, \bP) + h \Fn(\bP) \right) \\
\leq \tau \left(\Esp_{\bP}[\zeta] - \Esp_{\Q}[\zeta] + h \mathbb{E}_{\bP} \bracket*{ \nabla \HH_n \cdot \nabla \varphi - \Delta_{\V} \varphi} \right) + \frac{|\tau|}{2} \|\nabla^2 \zeta\|_{\infty} \times \mathcal{W}_{n}^{2}(\mathrm{Q}, \P) + o(\tau).
\end{multline*}
Taking $\tau \to 0$ and using the minimality of $\bP$ in \eqref{finite_n_problem}, we must have
\begin{equation*}
\left|\Esp_{\bP}[\zeta] - \Esp_{\Q}[\zeta] + h \mathbb{E}_{\bP} \bracket*{ \nabla \HH_n \cdot \nabla \varphi - \Delta_{\V} \varphi}\right| \leq \frac{1}{2} \|\nabla^2 \zeta\|_{\infty} \times \mathcal{W}_{n}^{2}(\mathrm{Q}, \P),
\end{equation*}
which concludes the proof.
\end{proof}

Note that \eqref{eq:jko:discrete-fokker-planck-tilde} involves $\Pkn$, which is defined the solution to the variational problem \eqref{finite_n_problem} on $\Ppn$, but that in our version of the JKO scheme, the next iterate $\Pkh$ is not chosen as $\Pkn$ itself (it is a finite spin measure on $\Confn$), but rather as its stationarized version. Thus, it remains to show that if $\varphi$ is a test function that \emph{depends on a fixed finite set of coordinates}, then the approximate Euler--Lagrange equation \eqref{eq:jko:discrete-fokker-planck-tilde} remains true for $\Pkh$ instead of $\Pkn$, up to an error that becomes small as $n \to \infty$. 

\subsubsection*{Step 3. Effect of the stationarization on the discrete Fokker--Planck equation}
We now conclude the proof of \eqref{eq:Discrete_FP} by showing that in \eqref{eq:jko:discrete-fokker-planck-tilde} we can replace $\Pkn$ by its stationarized version $\Pkh$ (with averages over $\Lan$) up to a controlled error.

Assume that the test function $\varphi$ is $\Lal$-local, and for $L$ such that $\ell \leq L \leq n$, define the \emph{stationary version} of $\varphi$ as: \newcommand{\hphi}{\hat{\varphi}}
\begin{equation*}
  \hphi := \fint_{\Lambda_{n-L}} \varphi \circ \theta_{u} \mathtt{d} u.
\end{equation*}
By construction, $\hphi$ is a $\Lan$-local function. Applying \eqref{eq:jko:discrete-fokker-planck-tilde} to $\hphi$ yields:
\begin{multline}
\label{FP_for_hz}
    \left| \frac{1}{h} \left( \Esp_{\Pkn}\left[  \hphi \right] -  \Esp_{\Pkmh}\left[  \hphi \right] \right) + \mathbb{E}_{\Pkn} \bracket*{ \nabla \HH_n \cdot \nabla \hphi - \Delta_{\V} \hphi }  \right| \leq \frac{1}{2h} \mathcal{W}^{2}_{n}(\Pkmh, \Pkn) \times \|\nabla^{2} \hphi\|_{\infty}.
\end{multline}
Recall that $\Pkh$ is chosen as $\Stat_n[\Pkn]$, as defined in Section \ref{sec:stationarisation}. We compare each term involving $\Pkn$ to the corresponding term for $\Pkh$.

\medskip

\begin{claim}[Discrete time derivative]
  We have
  \begin{equation*}
   \frac{1}{h} \left( \Esp_{\Pkn}\left[  \hphi \right] -  \Esp_{\Pkmh}\left[  \hphi \right]\right)  = \frac{1}{h} \left( \Esp_{\Pkh}\left[  \varphi \right] -  \Esp_{\Pkmh}\left[  \varphi \right] \right)   + \frac{1}{h} \norm{\varphi}_{\infty} \O\paren*{\frac{L}{n}}.
  \end{equation*}
\end{claim}
\begin{proof}
First, since by construction $\Pkmh$ is a stationary measure, we have the identity:
\begin{equation*}
  \Esp_{\Pkmh}\left[  \hphi \right] = \Esp_{\Pkmh}\left[  \varphi \right].
\end{equation*}
We now turn to comparing $\Esp_{\Pkn}\left[  \hphi \right]$ and $\Esp_{\Pkh}\left[  \varphi \right]$. We have:
\begin{equation*}
\Esp_{\Pkn}\left[  \hphi \right] =  \Esp_{\Pkn}\left[  \fint_{\Lambda_{n-L}} \varphi \circ \theta_{u} \mathtt{d} u \right] = \Esp_{\Pkh}\left[ \varphi \right]  + \norm{\varphi}_{\infty} \O\paren*{\frac{L}{n}},
\end{equation*}
using \eqref{fLalLanStat} for the second equality, because $\varphi$ is $\La_\ell$ and thus $\La_L$-local. This yields the claim.
\end{proof}

With a similar argument, we obtain the following estimate.
\begin{claim}[The Laplacian term]
  \begin{equation*}
    \Esp_{\Pkn}\left[ \Delta_{\V} \hat{\varphi} \right] = \Esp_{\Pkh}\left[ \Delta_{\V} \varphi \right] + \norm{\nabla^{2}\varphi}_{\infty} \O\paren*{\frac{L}{n}}.
\end{equation*}
\end{claim}

Finally, we have following estimate for the energy.
\begin{claim}[The energy term]
  \begin{equation*}
    \mathbb{E}_{\Pkn} \left[ \nabla \HH_n \cdot \nabla \hphi \right] = \mathbb{E}_{\Pkh} \left[ \nabla \HH_{L} \cdot \nabla \varphi \right] + \Cc(\ell, \varphi) \times \paren*{ \O\paren*{\frac{L}{n}} + o_{L}(1) },
  \end{equation*}
  where $o_L(1)$ tends to $0$ as $L \to \infty$ for $\ell$ fixed.
\end{claim}
\begin{remark}
There are two changes between the left-hand side and the right-hand side: not only do we go from $\Pkn$ to $\Pkh$ (and from $\hphi$ to $\varphi$) when applying the stationarization, but we also change the size of the box and restrict the interaction to $\La_L$ instead of $\La_n$
\end{remark}

\begin{proof}
For $\bx \in \Conf(\Lan)$ fixed, we have by definition of $\hphi$:
\begin{equation*}
\nabla \HH_n \cdot \nabla \hat{\varphi}(\bx) = \fint_{\Lambda_{n - L}} \sum_{i,j \in \Lan} J_{i,j} \partial_{1} \W(\bx_{i}, \bx_{j}) \cdot \nabla_i (\varphi\circ \theta_{u}) (\bx) \dd u.
\end{equation*}
Combining the chain rule, the fact that $\varphi$ is $\Lambda_{\ell}$-local and some translation of indices in the sums, one gets, for any fixed $u \in \Lambda_{n-L}$:
    \begin{multline*}
 \sum_{i,j \in \Lan} J_{i,j} \partial_{1} \W(\bx_{i}, \bx_{j}) \cdot \nabla_i (\varphi\circ \theta_{u}) (\bx)
= \sum_{i,j  \in \Lambda_{n}} J_{i,j} \partial_{1} \W(\bx_{i}, \bx_{j}) \cdot (\nabla_{i-u} \varphi)(\theta_{u} \bx)
\\ =  \sum_{\substack{i \in \Lambda_{n} -u\\j  \in \Lambda_{n}}} J_{i,j} \partial_{1} \W(\bx_{i+u}, \bx_{j}) \cdot (\nabla_{i} \varphi)(\theta_{u} \bx)
=  \sum_{\substack{i \in \Lambda_{\ell},  j  \in \Lambda_{n - u}}} J_{i,j} \partial_{1} \W(\bx_{i+u}, \bx_{j+u}) \cdot (\nabla_{i} \varphi)(\theta_{u} \bx).
\end{multline*}
Moreover, for $u \in \Lambda_{n-L}$, we have $\Lambda_{L} \subset \Lambda_{n-u}$. Thanks to the short-range assumption \eqref{eq:short-range}, we can restrict the sum over $j  \in \Lambda_{n - u}$ to a sum over $j \in \Lambda_{L}$ as follows:
\begin{multline*}
\sum_{\substack{i \in \Lambda_{\ell},  j  \in \Lambda_{n - u}}} J_{i,j} \partial_{1} \W(\bx_{i+u}, \bx_{j+u}) \cdot (\nabla_{i} \varphi)(\theta_{u} \bx) \\
= \sum_{\substack{i \in \Lambda_{\ell},  j  \in \La_L}} J_{i,j} \partial_{1} \W(\bx_{i+u}, \bx_{j+u}) \cdot (\nabla_{i} \varphi)(\theta_{u} \bx) + |\La_\ell| \times \|\partial_1 \W\|_\infty \times \|\nabla \varphi\|_\infty \times o_{L}(1),
\end{multline*}
with $o_L(1)$ as in the statement.

Computing the expectation under $\Pkn$, we are thus left with:
\begin{equation*}
\mathbb{E}_{\Pkn} \left[ \nabla \HH_n \cdot \nabla \hphi \right] = \mathbb{E}_{\Pkn} \left[ \fint_{\Lambda_{n - L}} \sum_{\substack{i \in \Lambda_{\ell},  j  \in \La_L}} J_{i,j} \partial_{1} \W((\theta_{u} \bx)_{i}, (\theta_{u} \bx)_{j}) \cdot (\nabla_{i} \varphi)(\theta_{u} \bx) \dd u \right] + \Cc(\ell, \varphi) o_L(1).
\end{equation*}
Using \eqref{fLalLanStat} we can compare the right-hand side to $\mathbb{E}_{\Pkh} \left[  \sum_{\substack{i \in \Lambda_{\ell},  j  \in \La_L}} J_{i,j} \partial_{1} \W(\bx_{i}, \bx_{j}) \cdot (\nabla_{i} \varphi)(\bx) \right]$ up to an error of size $|\La_\ell| \times \|\partial_1 \W\|_\infty \times \|\nabla \varphi\|_\infty \times \O\left(\frac{L}{n}\right)$, which yields the claim.
\end{proof}

The proof of \eqref{eq:Discrete_FP} follows then from the combination of \eqref{FP_for_hz} and the three previous claims. One needs moreover to observe that since $\varphi$ is $\Lambda_{\ell}$-local, we have (see \eqref{secondpartial}):
  \begin{equation*}
    \norm{\nabla \hat{\varphi}}_{\infty} \leq \frac{|\Lambda_{\ell}|^{\frac{1}{2}}}{|\Lambda_{n}|^{\frac{1}{2}}} \norm{\nabla \varphi}_{\infty}, \text{ and } \norm{\nabla^{2} \hat{\varphi}}_{\infty} \leq \frac{|\Lambda_{\ell}|}{|\Lambda_{n}|} \norm{\nabla^{2} \varphi }_{\infty},
  \end{equation*}
  and thus the error term in the right-hand side of \eqref{FP_for_hz} can indeed be written as:
  \begin{equation*}
\frac{1}{2h} \mathcal{W}^{2}_{n}(\Pkmh, \Pkn) \times \|\nabla^{2} \hphi\|_{\infty} = \Cc(\varphi, \ell) \frac{1}{h} \frac{1}{|\Lan|} \mathcal{W}^{2}_{n}(\Pkmh, \Pkn).
  \end{equation*}
\end{proof}

\subsection{Convergence of the scheme}
\newcommand{\Pph}{\P^{(h)}}
\newcommand{\ph}{\mathrm{p}^{(h)}}
Given $h > 0$, we have constructed a sequence $(\Pkh)_{k \geq 1}$ of stationary processes. We turn it into a piecewise continous curve by setting:
\begin{equation}
\label{def:Pht}
\Pph(t) := \Pkh \text{ for } t \in [kh, (k+1)h).
\end{equation}
Constructing the curve up to time $T$ corresponds to considering the $\frac{T}{h}$ first iterates. By \eqref{eq:jko:sup-free-energy}, we have for all $T > 0$: 
\begin{equation*}
\sup_{h > 0} \sup_{t \in [0,T]} \Fbeta(\Pph(t)) \leq \Cc(\Pz, T).
\end{equation*}
It implies that the specific relative entropy of $\Pph(t)$ is bounded on $[0,T]$ as $h \to 0$, and by monotonicity (see \eqref{eq:Eei}) we get:
\begin{equation*}
\sup_{h > 0} \sup_{t \in [0,T]}  \sup_{\ell \geq 1} \frac{1}{|\Lal|} \Ee_\ell(\Pph(t)) \leq \Cc(\Pz, T).
\end{equation*}
We then deduce that:
\begin{enumerate}
	\item For all $t \geq 0$, for all $\ell \geq 1$, the restriction of $\Pph(t)$ to $\Lal$ has a density on $M^{\Lal}$, which we denote by $\ph_{\Lal}(t, \cdot)$.
	\item For all $\ell \geq 1$, for all $T > 0$, the map $(t, x) \mapsto \ph_{\Lal}(t, x)$ is uniformly integrable on $[0, T] \times M^\Lal$.
\end{enumerate}
Using the second item, and up to extraction, as $h \to 0$ we get:
\begin{theorem}
\label{prop:Convergence}
There exists a measurable family $(\P(t))_{t \in [0,+ \infty)}$ of measures on $\Conf$ such that:
\begin{itemize}
   \item For a.e. $t \in [0, + \infty)$, $\P(t)$ is a stationary spin measure which admits local densities $\p_{\Lal}$ for all $\ell \geq 1$.
   \item The following convergence holds for all $T > 0$ and all $\ell \geq 1$:
   \begin{equation}
\label{eq:convergencefaible}
\ph_{\Lal} \to \p_{\Lal} \text{ weakly in } L^1\left([0, T] \times M^\Lal\right).
\end{equation}
 \end{itemize} 
\end{theorem}

We can then use the discrete estimate of Proposition \ref{prop:JKO} to show that we recover solutions of the Fokker--Planck equation.
In \cite{Jordan_1998,mei2018mean} this step is left to the reader, we provide here a short proof for completeness.

\begin{proposition}
 \label{prop:plalFP}
The limit point $(\P(t))_{t \in [0,+ \infty)}$ satisfies the dual formulation \eqref{def:VWEAK} of our Fokker--Planck--Kolmogorov equations.
\end{proposition}

\begin{remark}
By the regularity and uniqueness results of Theorems \ref{theo:regul} and \ref{th:uniqueness}, dual solutions are unique and thus our discrete scheme has a \emph{unique} limit point when $h \to 0$.
\end{remark}
\begin{remark}
It is also possible, by following the arguments in \cite{Jordan_1998} to show that there is time-wise convergence of $\ph_\Lal(t)$ to $\p_{\Lal}(t)$ weakly in $L^1(M^\Lal)$.
\end{remark}

\begin{proof} For fixed $h$, let $1 \leq \ell \leq L \leq n$ as in Proposition \ref{prop:JKO}. Fix $T > 0$ and let $f$ be a test function in $\mathscr{C}^{\infty}([0,T] \times \mathsf{Conf}_{\ell})$. Write $N = \frac{T}{h}$. 

Using the piecewise definition of $\Pph$ and the fundamental theorem of calculus, we have:
  \begin{multline} \label{eq:jko:sum-time-derivative}
  \sum_{k=1}^{N} \Esp_{\Pkh} \left[ f(kh, \cdot) \right] - \Esp_{\Pkmh} \left[ f(kh, \cdot) \right] = \sum_{k=1}^{N-1} \Esp_{\Pkh} \left[ f(kh,\cdot) - f((k+1)h,\cdot)  \right] - \mathbb{E}_{\mathrm{P}^{0}} \bracket*{ f(h, \cdot) }
 \\ = \int_{h}^{T} \Esp_{\Pph(s)} \bracket{ - \partial_{s} f(s,\cdot) } \mathtt{d} s - \mathbb{E}_{\mathrm{P}^{0}} \bracket*{ f(h, \cdot) } = \int_{0}^{T} \Esp_{\Pph(s)} \bracket{ - \partial_{s} f(s,\cdot) } \mathtt{d} s - \mathbb{E}_{\mathrm{P}^{0}} \bracket*{ f(h, \cdot) } + \O(h).
\end{multline}

On the other hand, by Proposition \ref{prop:JKO} applied to $\varphi := f(kh, \cdot)$ for $k = 0, \dots, N$, we have:
\begin{multline*}
\left| \Esp_{\Pkh} \left[ f(kh, \cdot) \right] - \Esp_{\Pkmh} \left[ f(kh, \cdot) \right]  +  h \Esp_{\Pkh} \left[ - \Delta_U f(kh, \cdot)  + \nabla \HH(\La_L \to \Lal) \cdot \nabla f(kh, \cdot) \right] \right|
\\ \leq \Cc(f, \ell) \left( \frac{1}{|\Lan|} \Wn^2(\Pkmh, \Pkn)+ \frac{L}{n} + h o_{L}(1) \right).
\end{multline*}
and we may thus write:
\begin{multline*}
\int_{0}^{T} \Esp_{\Pph(s)} \bracket{ - \partial_{s} f(s,\cdot) } \mathtt{d} s - \mathbb{E}_{\mathrm{P}^{0}} \bracket*{ f(h, \cdot) } = h \sum_{k=1}^{N} \Esp_{\Pkh} \left[ \Delta_U f(kh, \cdot)  - \nabla \HH(\La_L \to \Lal) \cdot \nabla f(kh, \cdot) \right] + \Error,
\end{multline*}
with an $\Error$ term bounded by:
\begin{equation}
\label{ErrorConv}
\Error := \Cc(f, \ell) \left( \sum_{k=1}^N \frac{1}{|\Lan|} \Wn^2(\Pkmh, \Pkn)+ \frac{NL}{n} + Nh o_{L}(1) \right) + \O(h).
\end{equation}

Using again the piecewise definition of $f$, we can re-write:
\begin{multline}
\label{eq:jko:sum-generator}
 \sum_{k=1}^{N} h \times \Esp_{\Pkh} \left[ \Delta_U f(kh, \cdot)  - \nabla \HH(\La_L \to \Lal) \cdot \nabla f(kh, \cdot) \right] \\ = \int_{0}^{T}\Esp_{\Pph(s)} \left[ \Delta_U f(kh, \cdot)  - \nabla \HH(\La_L \to \Lal) \cdot \nabla f(kh, \cdot) \right]  \mathtt{d} s + \O(h).
\end{multline}
Here the $\O(h)$ depends on $|\Lambda_{\ell}|$ and on the derivatives of $f$. We thus obtain:
\begin{equation*}
- \mathbb{E}_{\mathrm{P}^{0}} \bracket*{ f(h, \cdot) } \\ 
= \int_{0}^{T}\Esp_{\Pph(s)} \left[ \partial_{s} f(s,\cdot) + \Delta_U f(s, \cdot)  - \nabla \HH(\La_L \to \Lal) \cdot \nabla f(s, \cdot) \right]  \mathtt{d} s + \Error,
\end{equation*}
with $\Error$ as in \eqref{ErrorConv}.
The integral in the right-hand side can be written as:
\begin{equation*}
\int_{0}^{T} \int_{M^{\La_L}} \Big(\partial_{s} f(s,\bx) + \Delta_U f(s, \bx)  - \nabla \HH(\La_L \to \Lal) \cdot \nabla f(s, \bx)\Big) \ph_{\La_L}(\bx) \dd \omega_L(x), 
\end{equation*}
the term in parenthesis being bounded. Using the weak convergence of the local densities in $L^1([0,T] \times M^\La_L)$, this converges (up to extraction) as $h \to 0$ (and for $T, L$ fixed) to: 
\begin{equation*}
\int_{0}^{T}\Esp_{\P(s)} \left[ \partial_{s} f(s,\cdot) + \Delta_U f(s, \cdot)  - \nabla \HH(\La_L \to \Lal) \cdot \nabla f(s, \cdot) \right]  \mathtt{d} s.
\end{equation*}
It remains to check that the $\Error$ term is small. Recall that $Nh = T$. First, from \eqref{eq:jko:sup-wasserstein}, we know that:
\begin{equation*}
\sum_{k=1}^N \frac{1}{|\Lan|} \Wn^2(\Pkmh, \Pkn) = \Cc(\Pz, T) \times h,
\end{equation*}
which tends to $0$ as $h \to 0$. Moreover, we have $\frac{NL}{n} = \frac{TL}{nh}$, and by our choice we have $\frac{1}{nh} \to 0$ as $h \to 0$. Thus, sending first $h \to 0$ for $T, L$ fixed, then letting $L \to \infty$, we obtain that $\mathrm{P}$ satisfies the dual formulation \eqref{def:VWEAK} of the Fokker--Planck--Kolmogorov equation.
\end{proof}

\section{Infinite-volume interacting diffusion}
\label{sec:sde}

\subsection{Stochastic differential equations on manifolds and Brownian motion}
\label{sec:sde-infinite-volume}

Before getting into the details of our construction, let us recall that, to construct a stochastic differential equation associated to the Brownian motion on a manifold, one often needs to work \emph{not on the manifold itself}, but to embed it in a larger Euclidean space.
This is also the occasion to remind important constructions at the level on $M$ that we are going to adapt to the infinite product $\mathsf{Conf}$.

\paragraph{Brownian motion on $M$.}
We say that a stochastic process $(X_{t})_{t \in \R}$ on $M$ is a \emph{weighted Brownian motion} on $(M,\omega)$ provided, for all $f$ smooth, the following process is a martingale:
\begin{equation*}
t \mapsto f(X_{t}) - f(X_{0}) - \int_{0}^{t} \Delta_{\V} f(X_{s}) \mathtt{d} s.
\end{equation*}
In other words, a Brownian motion on the weighted manifold $(M,\omega)$ is a Markov diffusion whose generator is given by $\Delta_{\V}$.

\paragraph{Stochastic differential equations on $M$.}
Consider the \emph{Brownian motion} $B = (B^{1}, \dots, B^{m})$ on $\mathbb{R}^{m}$, and $m+1$ vectors fields $V_{0}, V_{1}, \dots, V_{m}$ on $M$, we say that a stochastic process $X$ solves the stochastic differential equation
\begin{equation}\label{eq:sde-df}
  \mathtt{d} X_{t} = \sum_{k=1}^{m} V_{k}(X_{t}) \circ \mathtt{d} B^{k}_{t} + V_{0}(X_{t}) \mathtt{d} t,
\end{equation}
provided for all $f$ smooth, and all $t$
\begin{equation*}
  f(X_{t}) = f(X_{0}) + \sum_{i=1}^{m} \int_{0}^{t} V_{i}f(X_{s}) \circ \mathtt{d} B^{i}_{s} + \int_{0}^{t} V_{0}f(X_{s}) \mathtt{d} s,
\end{equation*}
where $\circ$ means Stratonovich integration.
Let us recall that, by \cite[Thms.~1.1.9 \& 1.1.11]{HsuStoAna}, provided $V_{0}, V_{1}, \dots, V_{m}$ are Lipschitz, \eqref{eq:sde-df} admits a unique solution.
Moreover, it has infinite lifetime.
The celebrated Itō's formula, shows that any solution to \eqref{eq:sde-df}, is a diffusion whose generator is given by
\begin{equation}
\label{sumsquares}
  \sum_{i=1}^{m} V_{i}^{2} + V_{0}.
\end{equation}

\paragraph{Brownian motion on $M$ as a stochastic differential equation.}
In particular, we see that the Brownian motion on $M$ can be realised as a solution of a stochastic differential equation \eqref{eq:sde-df} if and only if the Laplace--Beltrami operator has the above form \eqref{sumsquares}. This is known to be false in general, even in the compact case.
The Laplace--Beltrami operator can be written as a sum of squares in local charts, but in arbitrarily small time the underlying Euclidean Brownian used to construct our stochastic differential equation can escape those charts.
To circumvent this issue we use the Nash embedding theorem, namely, there exists $m \geq \mathsf{n}$ and an isometric embedding $I \colon M \to \mathbb{R}^{m}$.
Thus, in this section, we always regard $M$ as a sub-manifold of $\mathbb{R}^{m}$.

By \cite[Thm.~3.14]{Hsu}, this allows us to realise the Laplace--Beltrami operator as a sum of squares.
Write $(e_{k})_{1\leq k \leq m}$ for the canonical basis of $\mathbb{R}^{m}$.
Define the vector field $P_{k} \colon M \to \mathrm{T}M$ by setting $P_{k}(x)$ to be the orthogonal projection of $e_{k}$ on $\mathrm{T}_{x}M$, then, we have $\Delta = \sum_{k=1}^{m} P_{k}^{2}$.
In particular, the weighted Brownian motion is a solution to \eqref{eq:sde-df} with $V_{0} \coloneq - \nabla \V$, and $V_{k} \coloneq P_{k}$.
More, generally when $V_{0} = - \nabla \V - v$, one speaks about a \emph{weighted Brownian motion with drift $v$}.

\begin{remark}
  The embedding $I \colon M \to \R^m$ is isometric in the sense of Riemannian manifolds, namely it preserves the metric tensor: $g(v, v) = \|I^{*} v\|^{2}$ for all tangent vectors $v$ (where $\| \cdot \|$ is the Euclidean norm), but it does not necessarily preserve the distance: we have $\di(x,y) \geq \|I(x) - I(y)\|$, with a strict inequality in general.
\end{remark}
\begin{remark}
  This extrinsic point of view, that is embedding $M$ as a sub-manifold of a Euclidean space might seem rather non-geometric.
  There exists a more intrinsic, and more concrete, approach to the construction of the Brownian motion that works in the \emph{orthonormal frame bundle} $\mathscr{O}M$.
At the technical level, we can lift the Laplace--Beltrami operator to the \emph{horizontal Bochner Laplacian}, which is always a sum of squares, then construct the stochastic differential equation at this level, and finally return to the original manifold.
We refer to \cite[Chap.~3]{Hsu} for details.
This approach is usually preferred, since the Nash embedding is non constructive and one does not know much about the vector fields $(P_{k})_{k=1, \dots, m}$.
In our case, our goal is only to prove the existence of a solution to the infinite dimensional stochastic differential equation that can subsequently be studied by approximating it with solutions of stochastic differential equations on the product manifolds $\mathsf{Conf}_{n}$ as $n \to \infty$.
Since adapting the ideas coming from the orthonormal frame bundle at the level of the infinite product $\mathsf{Conf}$ is a demanding and confusing task, we prefer working with this extrinsic approach at the level of $M$, which is sufficient for our needs.
\end{remark}

\subsection{Definition and properties of the infinite-volume diffusion}
Our goal is now to define and construct the \emph{Brownian motion} on $\mathsf{Conf}$ with drift $- \nabla \V - \beta \nabla \HH$.
Consider the differential operator $\mathbf{L} := \Delta_{\V} - \beta \nabla \HH \cdot \nabla$ acting on a smooth local function $\varphi : \Conf \to \R$ by:
\begin{equation*}
  \mathbf{L}\varphi = \sum_{i \in \Zd} \Delta_{i} \varphi - \nabla \V \cdot \nabla_{i} \varphi - \beta \nabla_{i} \HH \cdot \nabla_{i} \varphi,
\end{equation*}
the sum being finite since $\varphi$ is assumed to be local.

\begin{definition}
  We call \emph{Brownian motion on $\mathsf{Conf}$ with drift $-\nabla \mathrm{U} - \beta \nabla \mathsf{H}$}, or \emph{infinite-volume interacting diffusion},  started from $\bm{x}^{0} \in \mathsf{Conf}$, any stochastic process $t \mapsto X(t) = (X_{i}(t))_{i \in \Zd}$ on $\mathsf{Conf}$ such that $\Prob*{ X(0) = \bm{x}^{0}} = 1$, and for all smooth and local functions $\varphi \colon \mathsf{Conf} \to \mathbb{R}$, the process
  \begin{equation}
  \label{LMartingale}
    t \mapsto \varphi(X(t)) - \int_{0}^{t} \mathbf{L}\varphi(X(s)) \mathtt{d} s,
  \end{equation}
  is a martingale with respect to the canonical filtration of $X$.
\end{definition}

\begin{remark}
  For $\beta = 0$, the Brownian motion on $\mathsf{Conf}$ with drift $-\nabla \mathrm{U}$ can be realised as a system of independent Brownian motions on $M$ with drift $-\nabla \V$, or equivalently a system of independent weighted Brownian motions on $(M, \omega)$.
However, when $\beta \ne 0$, the coordinates are coupled through $\nabla \HH$, and the existence of the Brownian motions does not follow immediately. Indeed, 
since $\mathsf{Conf}$ is not a manifold, we cannot directly apply the well-established theory of stochastic differential equations on Hilbert manifolds \cite{Elworthy} or Banach manifolds \cite{BelopolskayaDaletskij}. See however Section \ref{sec:sde-infinite-volume}.

Following \cite{HolleyStroockDiffusionTorus}, which works on the circle and with finite-interactions, we obtain the existence by approximating our Brownian motion on $\mathsf{Conf}$ by a sequence of solutions of some finite-volume stochastic differential equations.
\end{remark}

\paragraph{Main results.}
We can state our main result regarding the infinite-volume diffusion associated to the spin system.
\begin{theorem}
\label{th:diffusion:solution}
Given $\bx^{0} \in \mathsf{Conf}$, there exists a Brownian motion with drift $-\nabla \mathrm{U} - \beta \nabla \mathsf{H}$ started from $\bx^{0}$.
\end{theorem}

Having constructed our Brownian motion, let us define the family of operators, acting on local functions $\varphi$ by:
\begin{equation*}
  \mathbf{P}_{t}\varphi(\bm{x}) \coloneq \Esp \bracket*{ \varphi(X^{\bm{x}}(t))}, \qquad t \geq 0,
\end{equation*}
where $X^{\bm{x}}$ is the Brownian motion started from $\bm{x}$ as defined above. By duality, we also define, for any spin measure $\P$
\begin{equation*}
  \mathbf{P}_{t}^{\star}\mathrm{P}(\varphi) \coloneq \mathrm{P}(\mathbf{P}_{t} \varphi).
\end{equation*}
The spin measure $\mathbf{P}_{t}^{\star}\mathrm{P}$ corresponds to the law at time $t \geq 0$ of the Brownian motion whose initial law is $\mathrm{P}$.
Solutions of the martingale problem are relevant to use since they connect to our Kolmogorov--Fokker--Planck equations.
\begin{theorem}\label{th:diffusion:fokker-planck}
  For every spin measure $\mathrm{P}$, the family of measures $t \mapsto \mathrm{P}(t) \coloneq \mathbf{P}_{t}^{\star}\mathrm{P}$ satisfies the Kolmogorov equation \eqref{def:VWEAK}.
\end{theorem}

\begin{proof}
  Since \eqref{def:VWEAK} is a linear equation and since $\mathbf{P}_{t}^{\star}\mathrm{P} = \int \mathbf{P}_{t}^{\star}\delta_{\bm{x}} \mathrm{P}(\mathtt{d} x)$, it is sufficient to show the claim for $\mathrm{P} = \delta_{\bm{x}}$ for all $\bm{x} \in \mathsf{Conf}$.
  Fix $\La \Subset \Zd$ and $f \in \mathscr{C}^{\infty}_{c}([0,\infty), M^{\Lambda})$.
  From the martingale property of \eqref{LMartingale}, we obtain for all $t > 0$:
  \begin{equation*}
    \Esp \bracket*{ f(t, X^{\bm{x}}(t)) } - f(0, \bm{x}) = \Esp \bracket*{ \int_{0}^{t} \partial_{t} f(s, X^{\bm{x}}(s)) + \mathbf{L} f(s, X^{\bm{x}}(s)) \mathtt{d} s },
  \end{equation*}
  which is exactly \eqref{def:VWEAK}.
\end{proof}


\paragraph{Outline of the proof of Theorem \ref{th:diffusion:solution}.}
\newcommand{\dg}{\mathsf{d}_\gamma}

Let us first sketch an outline of our construction, which traces back to several works from the 1980's \cite{ShigaShimizu,HolleyStroockDiffusionTorus,LehaRitter}.
We give the details of this construction in Section \ref{sec:diffusion:proofs}.
  \begin{enumerate} 
    \item Thanks to Nash's embedding theorem, we have fixed an embedding $I \colon M \to \R^m$ for a certain $m$. This in turn gives an embedding of $\Conf$ into $(\R^m)^{\Zd}$, which, due to the lack of structure, is not directly helpful.

    \item In Proposition \ref{th:diffusion:nabla-h-lipschitz}, we construct a positive probability measure $\gamma$ on $\Zd$ and a distance $\dg$ on $\mathsf{Conf}$ such that $(\mathsf{Conf}, \dg)$ embeds isometrically into the Hilbert space $\ell^{2}(\gamma)$ of functions from $\Zd$ to $\R^m$ that are square integrable with respect to $\gamma$.

    At any point $\bx \in \Conf$, the \emph{tangent space} $\mathrm{T}_{\bx} \mathsf{Conf}$ is isomorphic to the product $(\mathbb{R}^{m})^{\Zd}$ which is strictly larger than $\ell^{2}(\gamma)$. However, we also show, in Proposition \ref{th:diffusion:nabla-h-lipschitz}, that the vector field $\nabla \HH$ takes its values in $\ell^{2}(\gamma)$, and that moreover the map $\nabla \HH \colon (\mathsf{Conf}, \dg) \to \ell^2(\gamma)$ is Lipschitz.

  \item From there we construct a sequence $(X^{n})_n$ of stochastic processes on $\mathsf{Conf}$, by considering for the coordinates inside $\Lambda_{n}$ a Brownian motion on $\mathsf{Conf}_{n}$ with drift $-\beta \nabla \mathsf{H}_{n}$, and for the coordinates outside of $\Lambda_{n}$ independent Brownian motions.
    Thanks to the Lipschitz nature of $\nabla \mathsf{H}$, we show that $(X^{n})_n$ is Cauchy for a particular topology on stochastic processes built from the distance $\mathsf{d}_{\gamma}$.
    We then verify that the limit $X$ is indeed a Brownian motion on $\mathsf{Conf}$ with drift $-\beta \nabla \mathsf{H}$.

  \end{enumerate}

\begin{remark}
  The aforementioned papers \cite{ShigaShimizu,HolleyStroockDiffusionTorus,LehaRitter} work in a Euclidean setting, or on the unit circle, where the geometry does not really play a role.
  Our approach extends these Euclidean constructions to the manifold setting.
To that extent,  a similar endeavour is undertaken in \cite{ADK03}, where solutions to stochastic differential equations on an infinite product of compact manifolds is considered.
However, some points in their construction appear unclear to the authors.
For instance in \cite[Eq.~(44)]{ADK03} and below they mention the \enquote{Levi--Civita connection} on the orthonormal bundle $\mathscr{O}M$ but they never specify what Riemannian metric they work with, although many different choices are possible.
Similarly, in the proof of \cite[Prop.~3.1]{ADK03}, they say they can \enquote{generate} an isometric embedding of $\mathscr{O} M$ from an isometric embedding of $M$ but no explanations are given.
Moreover, with our notation, \cite[Eq.~(41)]{ADK03} is a stochastic differential equation on $\mathsf{Conf}$.
Since it not always possible to write the Brownian on $M$ as a solution to a stochastic differential equation on $M$, one needs to refer to the Brownian motion on $\mathscr{O} M$.
Consequently, we have decided to give a complete proof, which completely avoids working on $\mathscr{O} M$, although it is close in spirit to that of \cite{ADK03}.
\end{remark}
\begin{remark}
 It would be possible to construct our infinite-volume diffusion by considering a stochastic differential equation directly at the level of $\ell^{2}(\gamma)$, which is a Hilbert space.
  Since $\ell^{2}(\gamma)$ is a Hilbert space and $\nabla \HH$ is globally Lipschitz, to any given initial condition, there exists a unique solution $t \mapsto X(t)$ with infinite lifetime.
  This follows from standard results in the theory of stochastic differential equations \cite[Chap.~VI]{Elworthy}.
  A priori, the stochastic process $X$ thus constructed lives on $\ell^{2}(\gamma)$, which is larger than the image of $\Conf$ in $\ell^{2}(\gamma)$.
  To show that $X$ is in fact a process on $\Conf$, one would also need a finite-volume approximation argument.
\end{remark}

\subsection{Detailed construction and proofs}
\label{sec:diffusion:proofs}
\subsubsection*{Weighted distance on \texorpdfstring{$\mathsf{Conf}$}{Conf} and Lipschitz vector fields}
Let $\gamma := (\gamma_{i})_{i \in \Zd}$ be some positive probability measure on $\Zd$ to be specified later.
We equip the manifold $M$ at the site $i$ of the product $\mathsf{Conf}$ with the weighted metric $g_{i} := \gamma_{i}^{1/2} g$, and we endow $\mathsf{Conf}$ with the distance
\begin{equation*}
  \dg(\bx, \by) := \bracket*{ \sum_{i \in \Zd} \gamma_{i} \, \di(\bx_{i}, \by_{i})^{2} }^{\frac{1}{2}}, \qquad \bx,\, \by \in \mathsf{Conf}.
\end{equation*}
Although $\mathsf{Conf}$ is not a manifold, $\dg$ informally corresponds to the Riemannian distance obtained from the weighted Riemannian structure introduced above. It is immediate that $(\Conf, \dg)$ is complete, and that $\dg$ induces the product topology on $\Conf$. 
We set, for $\bx \in \Conf$ and $i \in \Zd$:
\begin{equation*}
  \left(\mathsf{I} (\bx) \right)_{i} := I \left(\bx_{i}\right).
\end{equation*}

\begin{remark}
  The map $\mathsf{I}$ defines a bi-Lipschitz embedding of the metric space $(\mathsf{Conf}, \mathsf{d}_{\gamma})$ into the Hilbert space $\ell^{2}(\gamma)$, that is the space of functions $\Zd \to \R^m$ that are square integrable with respect to $\gamma$. Indeed, since $I$ is smooth and $M$ compact, we find that $\sup_{x \in M} |I(x)| \leq C$.
This yields that $\sum \gamma_{i} |\mathsf{I}(\bx)_{i}|^{2} \leq C^{2}$.
The map $\mathsf{I}$ is one-to-one since $I$ is.
Using that $I$ is isometric and the variational definition of $\mathsf{d}$, we always have
\begin{equation*}
  \mathsf{d}(x,y) \geq \norm{I(x) - I(y)}, \qquad x,y \in M.
\end{equation*}
This readily implies that $\mathsf{I}$ is Lipschitz continuous.
For the converse inequality, we use that $I$ is a $\mathscr{C}^{\infty}$-diffeomorphsim onto its image.
In particular, $I^{-1} \colon I(M) \to M$ is smooth, and since $I(M)$ is compact, it gives that $I^{-1}$ is Lipschitz.
This shows that $\mathsf{I}$ is a bi-Lipschitz homeomorphism.
\end{remark}

\paragraph{Square-integrable vector bundle.}
Although we do not endow the metric space $(\mathsf{Conf}, \mathsf{d}_{\gamma})$ with a differentiable structure, for $\bx \in \Conf$, we define the \emph{tangent map} $\mathsf{I}^{*}_{\bx}$ on $\mathrm{T}_{\bx} \Conf$ as:
\begin{equation*}
  (\mathsf{I}^{*}_{\bx}v)_{i} := I^{*}_{\bx_{i}}v_{i}, \qquad v \in \mathrm{T}_{\bx} \mathsf{Conf}.
\end{equation*}
Since $I^{*}_{x} \colon \mathrm{T}_{x}M \to \mathrm{T}_{I(x)} \mathbb{R}^{m}$ and $\mathrm{T}_{x}M$ is not compact, the previous argument does not work, and for a general tangent vector $v \in \mathrm{T}_{\bx} \mathsf{Conf}$, we do not always have $\mathsf{I}^{*}_{x}v \in \ell^{2}(\gamma)$.
This leads us to introduce the set $\ell^{2}_{\bx}(\gamma)$ defined as:
\begin{equation*}
\ell^{2}_{\bx}(\gamma) := (\mathsf{I}^{*}_{\bx})^{-1}(\ell^{2}(\gamma)).
\end{equation*}
Since $I_{x}^{*}$ preserves the scalar product, we can rephrase our definition as follows: a tangent vector $v \in \mathrm{T}_{\bx} \mathsf{Conf}$ is in $\ell^{2}_{\bx}(\gamma)$ if and only if it satisfies
\begin{equation*}
\norm{v}_{\gamma} := \sum_{i \in \Zd} \gamma_{i} g_{\bm{x}_{i}}(v_{i}, v_{i}) < \infty.
\end{equation*}
In the rest of this section, we implicitly identify an element of $\ell^{2}_{\bx}(\gamma)$ and its image by $\mathsf{I}_{x}^{*}$ in $\ell^{2}(\gamma)$.

\paragraph{Compatible and Lipschitz vector fields.}
We say that a vector field $v$ on $\mathsf{Conf}$ is \emph{compatible} with $\gamma$ when $v(\bx) \in \ell^{2}_{\bx}(\gamma)$ for all $\bx \in \mathsf{Conf}$, and we say that a compatible vector field $v$ is \emph{Lipschitz} when
\begin{equation*}
  \norm{v(\bx) - v(\by)}_{\ell^{2}(\gamma)} \leq \dg(\bx,\by), \qquad \bx,\by \in \mathsf{Conf}.
\end{equation*}

\subsubsection*{Construction of a convenient distance from the interactions}
We now choose an appropriate probability measure $\gamma$.
\begin{proposition}
\label{th:diffusion:nabla-h-lipschitz}
  There exists a positive probability measure $\gamma$ on $\Zd$  such that the vector field $\nabla \HH$ is compatible with $\gamma$ and Lipschitz in the previous sense.
\end{proposition}

\begin{proof}
  For all $\bx \in \mathsf{Conf}$, we have:
    \begin{equation*}
      \|\nabla \HH(\bx)\|_{\ell^{2}(\gamma)}^{2} : = \sum_{i \in \Zd} \gamma_{i} \|\nabla_{i} \HH(\bx)\|^{2} \leq  \sum_{i \in \Zd} \gamma_{i}  \left( \sum_{j \in \Zd} \left| J_{i,j} \partial_{1} \W(\bx_{i}, \bx_{j}) \right|\right)^{2},
  \end{equation*}
  this quantity being finite thanks to our short-range assumption and the fact that $\gamma$ is a probability measure. This guarantees that $\nabla \HH$ takes values in $\ell^{2}(\gamma)$ (and is thus always “compatible” in the previous sense) regardless of the choice of $\gamma$.

\paragraph{Step 1. Choice of $\gamma$.}
  We follow \cite[Sec. 4]{LehaRitter} in order to construct a measure $\gamma$ such that the second part of the claim holds. For $i, j$ in $\Zd$, denote by $\nabla^{2}_{ij} \HH$ the following quantity:
  \begin{equation*}
    \nabla^{2}_{ij} \HH(\bx) =
    \begin{cases}
      \sum_{k \in \Zd} J_{i,k} \partial^{2}_{11} \W(\bx_{i}, \bx_{k}), & i = j;
      \\ J_{i,j} \partial^{2}_{12} \W(\bx_{i}, \bx_{j}), & i \ne j.
    \end{cases}
  \end{equation*}
  Fix any positive probability measure $\eta$ on $\Zd$, and define
  \begin{equation*}
    \mathbf{S}_{ij} := \norm{\nabla^{2}_{ij} \HH}_{\infty} + \eta_{j} + 1_{i=j}, \qquad i,\, j \in \Zd.
  \end{equation*}
  Using the short-range assumption, we have: 
  \begin{equation}
  \label{eq:defcc}
c := \sup_{i \in \Zd} \sum_{j \in \Zd} \mathbf{S}_{ij} < \infty.
  \end{equation} 
  We now define $\gamma$ by setting, for all $i \in \Zd$,
\begin{equation}
\label{eq:defTgamma}
\gamma_{i} := \frac{1}{\sum_{i \in \Zd}  \mathbf{T}_{1,i}} \mathbf{T}_{1,i} \text{ where } \mathbf{T} := \sum_{n \geq 0} \frac{\mathbf{S}^{n}}{(1+c)^{n}}.
\end{equation}
Since all the coefficients of $\mathbf{S}$ are positive, it is also the case for $\mathbf{T}$ and thus for $\gamma$. Note that
  \begin{equation*}
    \sum_{i \in \Zd} \mathbf{T}_{1,i} \leq \norm{\mathbf{T}}_{op} \leq 1+c,
  \end{equation*}
which ensures that $\gamma$ is well-defined as a positive probability measure on $\Zd$.

\paragraph{Step 2. Boundedness of $\mathbf{S}$.}
  \begin{claim}
  \label{claimli}
    The operator $\mathbf{S} \colon \ell^{\infty}(\gamma) \to \ell^{\infty}(\gamma)$ is bounded.
  \end{claim}
  \begin{proof}
The operator  $\mathbf{S}$ is bounded from $\ell^{\infty}$ to $\ell^{\infty}$ with operator norm $\leq c$ by \eqref{eq:defcc}. Since $\gamma$ is a positive probability measure, we have $\ell^{\infty} = \ell^{\infty}(\gamma)$ with coincidence of the norms.
  \end{proof}
We now prove the corresponding fact for $\ell^{1}(\gamma)$. 
  \begin{claim}
  \label{claiml1}
    The operator $\mathbf{S} \colon \ell^{1}(\gamma) \to \ell^{1}(\gamma)$ is bounded.
  \end{claim}
  \begin{proof}
  Since $\mathbf{T}\mathbf{S} = (1+c)(\mathbf{T}-\mathbf{Id})$, by definition of $\mathbf{T}$ in \eqref{eq:defTgamma}, we find on the one hand for $i \geq 0$:
  \begin{equation*}
  \left(\mathbf{T}\mathbf{S}\right)_{1,i} = \sum_{j \in \Zd} \mathbf{T}_{1,j} \times \mathbf{S}_{j,i} = \frac{1}{\sum_{k \in \Zd} \mathbf{T}_{1,k}} \sum_{j \in \Zd} \gamma_j \mathbf{S}_{j,i},
  \end{equation*}
  and on the other hand, by definition \eqref{eq:defTgamma}:
  \begin{equation*}
\left(\mathbf{T}\mathbf{S}\right)_{1,i} = (1+c)(\mathbf{T}-\mathbf{Id})_{1,i} \leq (1+c) \mathbf{T}_{1,i} = (1+c)  \frac{1}{\sum_{k \in \Zd} \mathbf{T}_{1,k}} \gamma_i,
  \end{equation*}
thus in conclusion:
\begin{equation}\label{eq:schur}
\sum_{j \in \Zd} \gamma_j \mathbf{S}_{j,i} \leq (1+c) \gamma_i.
\end{equation}

By Fubini's theorem and \eqref{eq:schur}, we deduce that for all $b \in \ell^{1}(\gamma)$:
  \begin{equation*}
    \norm{\mathbf{S}b}_{\ell^{1}(\gamma)} \leq \sum_{i\in \Zd} \gamma_{i} \sum_{j \in \Zd} \mathbf{S}_{ij} \abs{b}_{j} \leq (1+c) \sum_{j \in \Zd} \gamma_{j} \abs{b_{j}} \leq (1+c) \norm{b}_{\ell^{1}(\gamma)},
  \end{equation*}
  thus $\mathbf{S}$ is bounded on $\ell^{1}(\gamma)$.
\end{proof}

By the Riesz--Thorin interpolation theorem, Claim \ref{claimli} and Claim \ref{claiml1} imply that:
\begin{corollary}
 $\mathbf{S}$ is bounded on $\ell^{2}(\gamma)$.
\end{corollary}

\paragraph{Step 3. Checking that $\nabla \HH$ is Lipschitz on $(\Conf, \dg)$.}
  Take $\bx$ and $\by$ in $\mathsf{Conf}$, and compute:
  \begin{equation*}
    \norm{\nabla \HH(\bx) - \nabla \HH(\by)}^{2}_{\ell^{2}(\gamma)} = \sum_{i \in \Zd} \gamma_{i} \| \nabla_{i} \HH(\bx) - \nabla_{i} \HH(\by)\|^{2} \leq \sum_{i \in \Zd} \gamma_{i} \paren*{\sum_{j \in \Zd} \norm{\nabla^{2}_{ij} \HH}_{\infty} \di(\bx_{j}, \by_{j}) }^{2},
  \end{equation*}
  where we used Taylor's inequality to bound $\|\nabla_{i} \HH(\bx) - \nabla_{i} \HH(\by)\|$ in terms of the second derivatives of $\HH$.

  Writing $a_{j} = \di(\bx_{j}, \by_{j})$ for $j \in \Zd$, we thus have (using the boundedness of $\mathbf{S}$ in $\ell^2(\gamma)$ and the definition of $\mathbf{S}$):
  \begin{equation*}
    \norm{\nabla \HH(\bx) - \nabla \HH(\by)}_{\ell^{2}(\gamma)} \leq \norm{\mathbf{S} a}_{\ell^{2}(\gamma)} \leq \Cc \times \norm{a}_{\ell^{2}(\gamma)} = \Cc \times \dg(\bx, \by),
  \end{equation*}
thus the map $\nabla \HH$ from $(\mathsf{Conf}, \dg) \to \ell^{2}(\gamma)$ is Lipschitz.
\end{proof}

\subsubsection*{The finite-volume approximation.}
\label{sec:staysonConf}
To show that actually $X$ is a process on $\mathsf{Conf}$, define an auxiliary sequence of processes $(X^{n})_{n \geq 1}$, where $X^n$ is a Brownian motion on $\mathsf{Conf}$ with drift given by
\begin{equation*}
v^{n}_{i} :=
\begin{cases}
  \nabla_{i} \HH_n & \text{ for $i \in \Lambda_{n}$,} 
  \\  0 & \text{ for $i \in \Zd \setminus \Lambda_{n}$.}
\end{cases}
\end{equation*}
All the processes $(X^{n})$ are coupled as solutions of stochastic differential equations with respect to the same Brownian motions.
More precisely, consider a family $(W_{i,k})_{i \in \mathbb{Z}^{\mathsf{d}}, 1 \leq k \leq m}$ of independent Brownian motion on $\mathbb{R}$, and recall that the vector fields $(P_{k})_{1 \leq k \leq m}$ satisfies $\Delta = \sum_{k=1}^{m} P_{k}^{2}$.
Then, we define $X^{n}$ to be the unique solution with infinite lifetime of
\begin{equation}
    \label{eq:diffusion:brownian}
  \begin{dcases}
    \mathtt{d} X^{n}_{i}(t) = \sum_{k=1}^{m} P_{k}(X_{i}(t)) \circ \mathtt{d} W_{i,k} - \nabla \V(X_{i}(t)) \mathtt{d} t - \beta \nabla_{i} \mathsf{H}_{n}(X(t)) \mathtt{d} t, & \qquad i \in \Lambda_{n},
    \\ \mathtt{d} X^{n}_{i}(t) = \sum_{k=1}^{m} P_{k}(X_{i}(t)) \circ \mathtt{d} W_{i,k}, & \qquad i \in \mathbb{Z}^{\mathsf{d}} \setminus \Lambda_{n}. 
  \end{dcases}
\end{equation}
By definition, the coordinates of $X^{n}$ outside of $\Lambda_{n}$ are independent Brownian motions on $M$, while the restriction  of $X^n$ to $\Lambda_{n}$ satisfies a finite-dimensional stochastic differential equation on $\mathsf{Conf}_{n}$, and the two components are independent.
Since $\mathsf{Conf}_{n}$ is a manifold, we can define $X^{n}$ with the usual theory, and it is a stochastic process on $\mathsf{Conf}$.
Let us show that $(X^{n})_n$ has a limit in some suitable topology.

\paragraph{Choice of the topology.}
Consider the \emph{path space} $\Omega := \mathscr{C}^{0}(\mathbb{R}_{+}, \mathsf{Conf})$.
We introduce the following family, indexed by $T \in \mathbb{R}_{+}$, of pseudo-distances on $\Omega$:
\begin{equation*}
  \mathsf{D}_{T}(w, \tilde{w}) := \sup_{0 \leq t \leq T} \dg(w_{s}, \tilde{w}_{s}), \qquad w,\tilde{w} \in \Omega.
\end{equation*}
The family $(\mathsf{D}_{T})_{T \in \mathbb{R}_{+}}$ induces on $\Omega$ the topology of uniform convergence over compact sets of $\mathbb{R}_{+}$, and the corresponding family of pseudo-distances
\begin{equation*}
  \paren*{ \mathbb{E} \bracket*{ \mathsf{D}_{T}(\cdot, \cdot)^{2}}^{\hal} }_{T \in \mathbb{R}_{+}},
\end{equation*}
induces a complete uniform structure on the space of continuous processes on $\mathsf{Conf}$.

\paragraph{Convergence result.}
\begin{proposition}
\label{th:diffusion:cauchy}
The sequence $(X^{n})_{n \in \mathbb{N}^{*}}$ of stochastic processes on $\mathsf{Conf}$  is Cauchy for the uniform structure defined above.
In particular, there exists a stochastic process $X$ on $\mathsf{Conf}$, such that $(X^{n})_n$ converges to $X$ for this uniform structure.
\end{proposition}
\begin{proof}
  Take $p \geq n \geq 1$.
  The two processes $X^{n}$ and $X^{p}$ satisfy Euclidean equations similar to \eqref{eq:diffusion:brownian}.
  Writing $b$ for the vector field accounting for $\nabla \V$ and the Itō--Stratonovich correction, we derive the following stochastic differential equations in Itō integral form, for all $i \in \Zd$
  \begin{equation*}
    \begin{split}
      X^{n}_{i}(t) - X^{p}_{i}(t) &= \sum_{l=1}^{m} \int_{0}^{t} \paren*{P_{l}(X^{n}_{i}(s)) - P_{l}(X^{p}_{i}(s))} \mathtt{d} W_{l,i}(s) 
                                \\&- \int_{0}^{t} \paren*{b(X^{n}_{i}(s)) 1_{i \in \Lambda_{n}} - b(X_{i}^{p}(s)) 1_{i \in \Lambda_{p}} } \mathtt{d} s.
                                \\&- \beta \int_{0}^{t} \paren*{\nabla_{i} \HH_{n} (X^{n}(s)) 1_{i \in \Lambda_{n}} - \nabla_{i}\HH_{p} (X^{p}(s)) 1_{i \in \Lambda_{p}} } \mathtt{d} s.
    \end{split}
  \end{equation*}
  We now use these integral equations to derive an integral inequality on $\Esp \bracket*{ \mathsf{D}_{T}(X^{n}, X^{p})^2}$ that allows us to conclude that it goes to $0$.
  By construction $X_{i}^{n} = X_{i}^{p}$ for $i \in \mathbb{Z}^{\d} \setminus \Lambda_{p}$, since those are independent Brownian motions.
  In particular, we only need to control the above quantities, for $i \in \Lambda_{p}$.
  In this proof, $c$ is a universal constant independent of $n$ and $p$.
First, consider the martingale part
\begin{equation*}
  M_{i}(t) \coloneq \sum_{l=1}^{m} \int_{0}^{t} \paren*{P_{l}(X^{n}_{i}(s)) - P_{l}(X^{p}_{i}(s))} \mathtt{d} W_{l,i}(s).
\end{equation*}
Using that $(a+b)^{2} \leq 2a^{2} + 2b^{2}$ and that $\sup (a+b) \leq \sup a + \sup b$, we have
\begin{equation*}
  \sup_{t \in [0,T]} \norm{M_{i}(t)}^{2} \leq c \sum_{l=1}^{m} \sup_{t \in [0,T]} \norm*{ \int_{0}^{t} \paren*{P_{l}(X^{n}_{i}(s)) - P_{l}(X^{p}_{i}(s))} \mathtt{d} W_{l,i}(s) }^{2}
\end{equation*}
By the Burkholder--Davis--Gundy inequality \cite[Cor.~IV.4.2]{RevuzYor}, and since $P_{l}$ is smooth on $M$ compact, we find:
\begin{equation*}
  \mathbb{E} \bracket*{ \sup_{t \in [0,T]} \norm{M_{i}(t)}^{2}  } \leq c \sum_{l=1}^{m} \mathbb{E} \int_{0}^{T} \norm*{ P_{l}(X^{n}_{i}(s)) - P_{l}(X^{p}_{i}(s))}^{2} \mathtt{d} s  \leq c \int_{0}^{T} \norm*{ X^{n}_{i}(s) - X^{p}_{i}(s)}^{2} \mathtt{d} s.
\end{equation*}
Finally, summing with respect to $\gamma$, and using again that $\sup(a+b) \leq \sup a + \sup b$, we find
\begin{equation}\label{eq:cauchy:bound-martingale}
  \Esp \bracket*{ \sup_{t \in [0,T]} \sum_{i \in \Zd} \gamma_{i} \norm{M_{i}(t)}^{2}} \leq c \int_0^{T} \Esp \bracket*{ \mathsf{D}_{t}(X^{n}, X^{p})^{2} } \mathtt{d} t.
\end{equation}
Now, we handle the drift part.
We only provide details for the interacting part involving $\nabla \mathsf{H}_{n}$ and $\nabla \mathsf{H}_{p}$, the part involving $b$ is handled similarly, and is actually easier due to the lack of interactions.
By Cauchy-Schwarz's inequality we have
\begin{equation*}
  \begin{split}
  &\sup_{t \in [0,T]} \norm*{ \int_{0}^{t} \paren*{\nabla_{i} \HH_{n} (X^{n}(s)) 1_{i \in \Lambda_{n}} - \nabla_{i}\HH_{p} (X^{p}(s)) 1_{i \in \Lambda_{p}} } \mathtt{d} s }^{2} 
    \\\leq & \ T \int_{0}^{T} \norm*{\nabla_{i} \HH_{n} (X^{n}(s)) 1_{i \in \Lambda_{n}} - \nabla_{i}\HH_{p} (X^{p}(s)) 1_{i \in \Lambda_{p}} }^{2} \mathtt{d} s.
  \end{split}
\end{equation*}
By construction,
\begin{equation*}
\nabla_{i} \mathsf{H}_{n}(\bm{x}) = \nabla_{i} \mathsf{H}(\bm{x} | \bm{z}) - \sum_{j \in \Lambda_{n}} J_{i,j} \partial_{1} \psi(\bm{x}_{i}, \bm{x}_{j}), \qquad \bm{x} \in\mathsf{Conf}_{n}, \bm{z} \in \mathsf{Conf}_{\Zd \setminus \Lambda_{n}},
\end{equation*}
where $(\bm{x} | \bm{z})$ is the element of $\mathsf{Conf}$ obtained by merging $\bm{x}$ and $\bm{z}$.
Hence, since, by \cref{th:diffusion:nabla-h-lipschitz}, $\nabla \mathsf{H}$ is Lipschitz, we get
\begin{equation}\label{eq:nabla-hn-lipschitz}
  \norm{\nabla \mathsf{H}_{n}(\bm{x}) - \nabla \mathsf{H}_{p}(\bm{y})}_{\ell^{2}(\gamma)} \leq c\, \mathsf{d}_{\gamma}(\bm{x}, \bm{y}) + 2 \norm{\partial_{1} \psi}_{\infty} \sum_{i \in \Zd} \gamma_{i} \sum_{j \in \Lambda_{p} \setminus \Lambda_{n}} J_{i,j}.
\end{equation}
The constant $c$ comes from the Lipschitz continuity of $\nabla \mathsf{H}$; while the second term on the right-hand side is negligible, since $J$ is summable by \eqref{eq:short-range} and $\gamma$ is a probability measure.
Thus, using again the sub-addivity of the supremum, we obtain
\begin{equation}\label{eq:cauchy:bound-drift}
  \begin{split}
    & \Esp \bracket*{ \sup_{t \in [0,T]} \sum_{i \in \Zd} \gamma_{i} \norm*{  \int_{0}^{t} \paren*{\nabla_{i} \HH_{n} (X^{n}(s)) 1_{i \in \Lambda_{n}} - \nabla_{i}\HH_{p} (X^{p}(s)) 1_{i \in \Lambda_{p}} } \mathtt{d} s }^{2} } 
  \\&\leq  c\, T \int_{0}^{T} \Esp \bracket*{ \mathsf{D}_{t}(X^{n}, X^{p})^{2} } \mathtt{d} t + T\, \varepsilon_{n,p},
  \end{split}
\end{equation}
where $\varepsilon_{n,p}$ is the negligible term in \eqref{eq:nabla-hn-lipschitz}.
Combining \eqref{eq:cauchy:bound-martingale} and \eqref{eq:cauchy:bound-drift} yields
\begin{equation*}
  \Esp \bracket*{ \mathsf{D}_{T}(X^{n}, X^{p})^{2}} \leq c\,T \int_{0}^{T} \mathbb{E} \bracket*{ \mathsf{D}_{t}(X^{n}, X^{p})^{2}} \mathtt{d} t  + T\, \varepsilon_{n,p}.
\end{equation*}
By Gronwall's lemma, this shows that the sequence $(X^{n})_n$ is Cauchy with respect to the complete uniform structure that we have defined on the space of continuous processes on $\mathsf{Conf}$.
Thus, it converges as a continuous stochastic process on $\mathsf{Conf}$ to some limit $X$.
\end{proof}

\subsubsection*{Martingale problem and link with the Fokker--Planck--Kolmogorov equation}
We are left to prove that the limiting process $X$ constructed above is indeed the desired infinite-volume interacting diffusion.
We rely on an approximation argument that leverages properties of local martingales on $\mathbb{R}$, we refer to \cite[\S IV.1]{RevuzYor} for more details.
For $t \geq 0$, let us call $\mathbb{H}_{T}$ the space of all $\mathbb{R}$-valued local martingales $M = (M(t))_{t \in \mathbb{R}_{+}}$ such that
\begin{equation*}
  \norm{M}_{\mathbb{H}_{T}} \coloneq \Esp \bracket*{M(T)^{2}}^{1/2} = \Esp \bracket*{ \hsp{M}{M}_{T}} < \infty,
\end{equation*}
where $\hsp{\cdot}{\cdot}$ is the quadratic variation and the equality follows from \cite[Cor.~IV.1.24]{RevuzYor}.
The quantity $\norm{\cdot}_{\mathbb{H}_{T}}$ is only a seminorm.
We define $\mathbb{H} \coloneq \cap_{T \in \mathbb{R}_{+}} \mathbb{H}_{T}$, and we equip it with the inductive topology.
\begin{lemma}\label{th:diffusion:local-martingale}
  The space $\mathbb{H}$ is Fréchet.
  Moreover, if $M \in \mathbb{H}$, then $\mathbb{H}$ is a martingale.
\end{lemma}
\begin{proof}
  First, by \cite[Cor.~IV.1.25]{RevuzYor}, if $M$ is in $\mathbb{H}$, then, $(M(t))_{0 \leq t \leq T}$ is a martingale for all $T \in \mathbb{R}_{+}$.
  Then, if $\norm{M}_{\mathbb{H}_{T}} = 0$ for all $T \in \mathbb{R}_{+}$ by Doob's inequality, $\sup_{t \leq T} |M(t)| = 0$ for all $T \in \mathbb{R}_{+}$, thus $M = 0$.
  Hence, the family of seminorms induces a topology on $\mathbb{H}$ that is separated.
  By Doob's inequality, we also have $\norm{\cdot}_{\mathbb{H}_{T}} \leq 2 \norm{\cdot}_{\mathbb{H}_{T'}}$, thus we can consider only a countable family of seminorms for the inductive topology.
  To show that $\mathbb{H}$ is complete, consider a Cauchy sequence $(M^{n})_n$.
  Then by Doob's inequality, it means that $(M^{n})_n$ is Cauchy for the locally uniform in time convergence in probability and thus converges to a local martingale $M$, with $\Esp \bracket*{ M(T)^{2}} = \lim_{n} \Esp \bracket*{M^{n}(T)^{2}} < \infty$.
  Finally, if $M \in \mathbb{H}$ then for all $T$, the family $\set*{ M(\tau) : \tau \ \text{stopping time}, \tau \leq T}$ is uniformly integrable.
  Thus $M$ is a martingale by \cite[Prop.~IV.1.7]{RevuzYor}.
\end{proof}

\begin{proof}[Proof of {\cref{th:diffusion:solution}}]
  Take $\varphi \colon \mathsf{Conf} \to \mathbb{R}$ smooth and local.
By definition, we have to show that
\begin{equation*}
  M(t) \coloneq \varphi(X(t)) - \int_{0}^{t} \mathbf{L}\varphi(X(s)) \mathtt{d} s,
\end{equation*}
is a martingale.
To do so, we use our finite-dimensional approximation from above.
  By Itō's formula, the generator of the process $X^{n}$ is given by:
  \begin{equation*}
    \mathbf{L}_{n} \varphi := \sum_{i \in \Zd} (\Delta_{i} - \nabla \V \cdot \nabla_{i}) \varphi - \beta \sum_{i \in \Lambda_{n}} \nabla_{i} \HH_{n} \cdot \nabla_{i} \varphi.
  \end{equation*}
  The first sum above is actually finite since $\varphi$ is local.
  In particular, the process
  \begin{equation*}
    M_{n}(t) \coloneq \varphi(X^{n}(t)) - \int_{0}^{t} \mathbf{L}_{n} \varphi(X^{n}(s)) \mathtt{d} s,
  \end{equation*}
  is a continuous $\mathbb{R}$-valued martingale.
  Let us show that $(M^{n})$ is a Cauchy sequence in $\mathbb{H}$.
  Indeed, fix $T \in \mathbb{R}_{+}$, from Itō's formula, we know that
  \begin{equation*}
    M_{n}(T) = \sum_{i \in \Zd} \sum_{l=1}^{m} \int_{0}^{T} \nabla_{i} \varphi(X^{n}(s)) \cdot P_{l}(X^{n}_{i}(s)) \mathtt{d} W_{l,i}(s),
  \end{equation*}
  where the sum is finite by locality of $\varphi$.
  In particular, choosing $\ell \in \mathbb{N}$ such that $\varphi$ is $\Lambda_{\ell}$-local, using that the $(W_{l,i})$ are independent Brownian motions, the quadratic variation satisfies
  \begin{equation*}
    \hsp{M_{n}}{M_{n}}_{T} = \sum_{i \in \Lambda_{\ell}} \sum_{l=1}^{m} \int_{0}^{T} \norm{\nabla_{i} \varphi(X^{n}(s)) \cdot P_{l}(X_{i}^{n}(s))}^{2} \mathtt{d} s.
\end{equation*}
Since $\varphi$ and $P_{l}$ are smooth and $\mathsf{Conf}$ is compact, the above quantity has finite expectation.
Similarly,
\begin{equation*}
\hsp{M_{n} - M_{p}}{M_{n} - M_{p}}_{T} = \sum_{i \in \Lambda_{\ell}} \sum_{l=1}^{m} \int_{0}^{T} \norm{\nabla_{i} \varphi(X^{n}(s)) \cdot P_{l}(X_{i}^{n}(s)) - \nabla_{i} \varphi(X^{p}(s)) \cdot P_{l}(X^{p}_{i}(s) }^{2} \mathtt{d} s.
\end{equation*}
This quantity goes to zero in expectation as $n,p \to \infty$ since $\nabla_{i}\varphi(x) \cdot P_{l}(x_{i})$ is Lipschitz and the fact that
\begin{equation*}
  \Esp \bracket*{ \mathsf{D}_{T}(X^{n}, X^{p})^{2}} \to 0,
\end{equation*}
as shown in Proposition \ref{th:diffusion:cauchy}.
Since this holds for all $T \in \mathbb{R}_{+}$, $(M_{n})$ converges to some martingale $M \in \mathbb{H}$.

  By Proposition \ref{th:diffusion:cauchy},
  \begin{equation*}
    \Esp \bracket*{ \sup_{t \in [0,T]} (\varphi(X^{n}(t)) - \varphi(X(t)))^{2} } \to 0.
  \end{equation*}
  Finally, using an estimate similar to that of \eqref{eq:nabla-hn-lipschitz}, we find
  \begin{equation*}
    \sup_{t \in [0,T]} \int_{0}^{t} (\mathbf{L}_{n} \varphi(X^{n}(s)) - \mathbf{L}\varphi(X(s))) \mathtt{d} s \xrightarrow[n \to \infty]{L^{2}} 0.
  \end{equation*}
  This shows that $M(t)$ coincides with $\varphi(X(t)) - \int_{0}^{t} \mathbf{L} \varphi(X(s)) \mathtt{d} s$ and is indeed a martingale, which concludes the proof.
\end{proof}

\section{Extension to the non-compact case}\label{s:non-compact}
It could be interesting to extend our results to the case of non-compact manifolds $M$.
Let us point out how our strategy could be adapted, at least for manifolds whose geometry is sufficiently well controlled.
\medskip

\begin{assumption}
$M$ is smooth and complete, with bounded geometry in the following sense.
  \begin{enumerate}
    \item There exists $\kappa \in \mathbb{R}$ such that $\Ricc + \nabla^2 \V \geq \kappa$ uniformly on $M$.
    \item There exists $\eta_{0} > 0$ such that $\inf_{x \in M} \omega(B(x,1)) > \eta_{0}$, where $B(x,1)$ is the ball of center $x$ and radius $1$ for the Riemannian metric.
  \end{enumerate}
\end{assumption}

Under this assumption, the heat kernel bounds mentioned in Appendix \ref{s:heat-kernel} stay valid.
In particular \cite{Engoulatov} extends the results of \cite{Hsu} to the non-compact case.
Whence, at least informally, our arguments could be adapted. However, there are two major obstructions to a full generalization of our result.
\begin{enumerate}
  \item Optimal transport on generic non-compact manifolds is not well-understood.
    For instance, \cite{McCannPolar} works on compact manifolds, while \cite{CMCS06} works with compactly supported measure.
    We are not aware of a complete generalization of the theory of optimal transport to the non-compact case \cite[Problem~10.23]{VillaniOldNew}.

    Nonetheless, when working on $\mathbb{R}^{\nn}$, which is arguably the most relevant case of unbounded spin systems, all the theory of optimal transport is perfectly well-developed.

  \item When $M$ is not bounded, it is not clear whether the specific Wasserstein $\mathcal{W}$ is finite.
  This could cause problems in the construction of the gradient flow in Section \ref{sec:JKO}.
  Even if it would be formally be possible to define an $\mathsf{EVI}$-gradient flow with respect to an extended distance, the authors do not know how such an equation could be interpreted nor what convergence results such an $\mathsf{EVI}$ would imply.
  We show below that, when the curvature $\kappa$ is positive, two spin measures with finite free energy are always at finite specific Wasserstein distance.
\end{enumerate}
These two points indicate that our results should extend at least to $\mathbb{R}^{\nn}$ equipped with a strictly log-concave measure, i.e. a probability measure $\omega(\mathtt{d} x) := \mathrm{e}^{-\V(x)} \mathtt{d} x$ with $\nabla^{2} \V \geq \kappa > 0$, and an interaction $\W \in \mathscr{C}^{3}(\mathbb{R}^{\nn} \times \mathbb{R}^{\nn})$ with bounded derivatives.

\begin{lemma}\label{th:talagrand}
  Assume that $\kappa > 0$, then the specific Wasserstein distance $\mathcal{W}$ is a true distance.
\end{lemma}
\begin{proof}
  By the positive curvature assumption, we have the Talagrand inequality on $\mathsf{Conf}_{n}$ (see \cite{OttoVillani}):
  \begin{equation*}
    \mathcal{W}^{2}_{n}(\mathrm{P}^{0}, \mathrm{P}^{1}) \leq \frac{2}{\kappa} \paren*{\Fn(\mathrm{P}^{0}) + \Fn(\mathrm{P}^{1})}, \qquad \mathrm{P}^{0},\, \mathrm{P}^{1} \in \mathscr{P}_{n}.
  \end{equation*}
  From which we immediately conclude that
  \begin{equation*}
    \mathcal{W}^{2}(\mathrm{P}^{0}, \mathrm{P}^{1}) \leq \frac{2}{\kappa} \paren*{\fbeta(\mathrm{P}^{0}) + \fbeta(\mathrm{P}^{1})}, \qquad \mathrm{P}^{0},\, \mathrm{P}^{1} \in \mathscr{P}^{\mathrm{s}}.
  \end{equation*}
\end{proof}

\appendix

\section{Hölder and Sobolev spaces}
\paragraph{Hölder norm.}
For $k \in \mathbb{N}$, we write $\mathscr{C}^{k}(M)$ for the space of $k$ times continuously differentiable functions from $M$ to $\R$ equipped with the norm
\begin{equation*}
  \norm{f}_{\mathscr{C}^{k}(M)} := \max_{l=1,\dots,k} \norm{f^{(l)}}_{\infty}.
\end{equation*}
For $\alpha \in [0,1]$, we write $\mathscr{C}^{k,\alpha}(M)$ for the space of $f \in\mathscr{C}^{k}(M)$ such that $f^{(k)}$ is $\alpha$-Hölder continuous, we equip it with the norm
\begin{equation*}
  \norm{f}_{\mathscr{C}^{k,\alpha}(M)} := \norm{f}_{\mathscr{C}^{k}(M)} + \sup_{(x,y) \in M, x\ne y} \frac{\abs{f^{(k)}(x) - f^{(k)}(y)}}{\di(x,y)^{\alpha}}.
\end{equation*}

\paragraph{Sobolev spaces.}
Conveniently, we rather work with fractional Sobolev spaces that allow for a functional analytic approach.
We refer for instance to \cite{Aubin,Hebey} for thorough introductions to the subject for the case of integers, and to \cite{Ziemer} for the case of fractional spaces in a Euclidean setting.
Although the result stated here are well-known, we could not identify a reference in our setting so we provide a proof here.

We give the definitions below for the manifold $M$ endowed with the measure $\omega$ as in \eqref{eq:defomega}, they extend readily to $M^\La$ and $\oL$.

\subparagraph{Integer Sobolev spaces.}
For $k \in \mathbb{N}$ and $p \in [1, + \infty]$, the \emph{integer Sobolev space} of functions on with $k$ derivatives in $L^{p}$ is defined as the space
\begin{equation*}
  \mathscr{W}^{k,p}(\omega) := \set*{ f \in \mathscr{L}^{p}(\omega) : \nabla^{j} f \in \mathscr{L}^{p}(\omega),\, j= 1, \dots, k },
\end{equation*}
equipped with the following norm, for which it is a Banach space
\begin{equation*}
  \norm{f}_{\mathscr{W}^{k,p}(\omega)} := \sum_{j=0}^{k} \norm{\nabla^{j} f}_{\mathscr{L}^{p}(\omega)}.
\end{equation*}

\subparagraph{Fractional Sobolev spaces.}
Since $1-\Delta_{\V}$ is a positive operator on $\mathscr{L}^{2}(\omega)$, by spectral calculus, for $s \in \mathbb{R}$, one can define $(1-\Delta_{\V})^{s}$ acting on $\mathscr{L}^{2}(\omega)$, and in particular on the space $\mathscr{D}(M) := \mathscr{C}^{\infty}(M)$ of smooth test functions. Subsequently, we can define $(1-\Delta_{\V})^{s}$ on the space $\mathscr{D}'(M)$ of distributions.

Since $-\Delta_{\V}$ is non negative on $\mathscr{L}^{2}(\omega)$, when $s \geq 0$, we can directly define $(-\Delta_{\V})^{s}$, and all the definitions and statements involving $(1-\Delta_{\V})^{s}$ can be replaced with $(-\Delta_{\V})^{s}$ in that case.

The \emph{fractional Sobolev space} with $s \in \mathbb{R}$ derivatives in $L^{p}$ ($p \in [1,\infty]$) is then defined as:
\begin{equation*}
  \mathscr{W}^{s,p}(\omega) := \set*{ f \in \mathscr{D}'(\omega) : (1-\Delta_{\V})^{\frac{s}{2}}f \in \mathscr{L}^{p}(\omega) },
\end{equation*}
equipped with the following norm, that turns it into a Banach space:
\begin{equation*}
  \norm{f}_{\mathscr{W}^{s,p}(\omega)} := \norm{(1-\Delta_{\V})^{\frac{s}{2}}f}_{\mathscr{L}^{p}(\omega)}.
\end{equation*}
For $s \geq 0$, we have the embedding $\mathscr{W}^{s,p}(\omega) \hookrightarrow \mathscr{L}^{p}(\omega)$, but for negative $s$, these Sobolev spaces are not functions space anymore.

A celebrated result of Bakry \cite{BakryRieszTransform} asserts that for $s \in \mathbb{N}$ and $p \in (1,\infty)$, the definition with iterated derivatives and the one with powers of the weighted Laplacian are consistent: they produce the same space, with equivalent norms.

\subparagraph{Sobolev embeddings.}
We have the following \emph{Sobolev embeddings}.
\begin{align}
  & \label{eq:sobolev:embedding} \mathscr{W}^{k+\varepsilon,p}(\omega) \hookrightarrow \mathscr{W}^{k,p^{*}}(\omega), \qquad k \in \mathbb{N},\, \varepsilon \in [0,1],\, \nn - \varepsilon p >0,\, p^{*} := \frac{\nn p}{\nn-\varepsilon p}.
\\& \label{eq:sobolev:embedding:holder} \mathscr{W}^{s,p}(\omega) \hookrightarrow \mathscr{C}^{k,\alpha}(M), \qquad k + \alpha = s - \frac{\nn}{p},\, k \in \mathbb{N},\, \alpha \in (0,1),\, p \in [1,\infty],\, s > 0.
\end{align}
\begin{proof}
  \eqref{eq:sobolev:embedding} with $\varepsilon = 1$ and \eqref{eq:sobolev:embedding:holder} are stated in \cite[Thm.~2.21]{Aubin}. Let us treat here the case $\epsilon \neq 0$, and $k =1$, which is sufficient.
  By spectral calculus \cite[\S~13]{Rudin}, $(-\Delta_{\V})^{\frac{s}{2}}$ is an integral operator whose kernel has the form
  \begin{equation*}
    M \times M \ni (x,y) \mapsto c(s) \int_{0}^{\infty} t^{\frac{s-\nn}{2}-1} g_{t}(x,y) \mathtt{d} t,
  \end{equation*}
  where $g_{t}$ is the weighed heat kernel associated with $\Delta_{\V}$, see Appendix \ref{s:heat-kernel} for definitions.
  By \eqref{eq:heat-kernel:bound:upper} and \eqref{eq:heat-kernel:bound:lower}, this kernel has the same behaviour as $\di(x,y)^{-\nn-s}$.
  The proof then proceeds as in the Euclidean setting, see \cite[Thm.~2.82.]{Ziemer}.
\end{proof}

\subparagraph{Duality for Sobolev spaces.}
For $s \in \mathbb{R}$ and $p \in (1,\infty)$, the duality of Lebesgue spaces extends to fractional Sobolev spaces and we have $(\mathscr{W}^{s,p})' = \mathscr{W}^{-s,p'}$, where $p'$ is the Hölder conjugate.
This leads to the following characterization of Sobolev spaces.
\begin{lemma}
  Let $f \in \mathscr{L}^{p}(\omega)$, $s \geq 0$, $p \in (1,\infty)$, and $C > 0$ such that for all $\varphi \in \mathscr{D}(M)$:
  \begin{equation*}
    \int f \varphi \mathtt{d} \omega \leq C \norm{\varphi}_{\mathscr{L}^{p'}(\omega)}.
  \end{equation*}
  Then, $f \in \mathscr{W}^{s,p}(\omega)$ with $\norm{f}_{\mathscr{W}^{s,p}(\omega)} \leq C$.
\end{lemma}

\section{Heat kernel and heat semi-group}\label{s:heat-kernel}
\subsection{Reminders and definitions}
For the reader's convenience, we recall the main properties of the weighted heat kernel on $M$, and on $M^{\Lambda}$ for $\Lambda \Subset \Zd$. Those facts can be found in standard Riemannian geometry textbooks, for instance \cite{Chavel} for the non-weighted case $\V \equiv 0$, and \cite{Grigoryan} for the general case.

\begin{theorem}\label{th:heat-kernel:properties}
  There exists a unique $g \in \mathscr{C}^{1,2}((0,+\infty) \times (M \times M))$ such that, for all $y \in M$, the function $f(t,x) \coloneq g(t,x,y)$ satisfies
  \begin{equation*}
    \begin{dcases}
      & \partial_{t} f = \Delta_{\V} f, \qquad \text{on}\ (0,+\infty) \times M;
    \\& \lim_{t \to 0} f(t,\cdot) = \delta_{y}, \qquad \text{in the sense of distributions}.
    \end{dcases}
  \end{equation*}
  Moreover, the function $g$ satisfies:
  \begin{itemize}
    \item $g(t,\cdot,\cdot) \geq 0$ for all $t > 0$.
    \item $g \in \mathscr{C}^{\infty}((0,+\infty) \times M \times M)$.
    \item $g(t,x,y) = g(t,y,x)$ for all $x$ and $y \in M$, and $t > 0$.
    \item $\int_M g(t,x,y) \omega(\mathtt{d}x) = 1$, for all $t > 0$ and $y \in M$.
  \end{itemize}
\end{theorem}
We often write $g_{t} \coloneq g(t,\cdot,\cdot)$.

To the heat kernel $g$, we associate the \emph{heat semi-group} $(\mathsf{G}_{t})_{t \geq 0}$ defined by:
\begin{equation*}
  (\mathsf{G}_{t}f)(x) \coloneq \int_M f(y) g(t, x, y) \dd \omega(y), \qquad f \in \mathscr{L}^{1}(\omega).
\end{equation*}
It is known that $(\mathsf{G}_{t})_{t \geq 0}$ is the Feller semi-group associated with the weigthed Brownian motion on $M$:
\begin{equation*}
  \mathsf{G}_{t}f(x) = \mathbb{E} \bracket*{ f(B_{t}^{x}) },
\end{equation*}
where $(B_{t}^{x})_{t \geq 0}$ is the Brownian motion with drift $-\nabla \V$ started at $x$. The heat semi-group has a strong regularisation property: for all $f \in \mathscr{L}^{1}(\omega)$ and at all times $t >0$ we have $\mathsf{G}_{t}f \in \mathscr{C}^{\infty}(M)$.

\subparagraph{Heat kernel on $M^\La$.}
For any fixed $\La \Subset \Zd$, $M^{\Lambda}$ is still a smooth Riemannian manifold, so it possesses a weighted heat kernel $g^{\Lambda}$ and a weighted heat semi-group $\mathsf{G}^{\Lambda}_{t}$.
In view of the product structure, the heat kernel on $M^\La$ is given (for all $t \geq 0$, $\bx, \by$ in $M^\La$) by:
\begin{equation*}
  g^{\Lambda}_{t}(\bx, \by) = \prod_{i \in \Lambda} g_{t}(\bx_{i}, \by_{i}).
\end{equation*}
We may thus easily deduce properties of $g^\Lambda$ from the properties of $g$.
In that regard, only the cardinal of $\Lambda$ matters.

\subparagraph{Brownian motion and radial estimates}
In a Euclidean space, the position of a Brownian motion at fixed time $t$ has a Gaussian law whose centered moments are explicitly known. On $M^{\Lambda}$ we have the following estimate (recall that $\kappa \in \mathbb{R}$ is a lower bound on the Ricci curvature of $M$)
\begin{lemma}
  Let $(B_{t})_{t \geq 0}$ be the weigthed Brownian motion on $(M^{\Lambda}, \omega_{\Lambda})$ started from $x \in M^{\Lambda}$.
  Then for all $t \geq 0$ and all $p \geq 1$, we have:
  \begin{equation}\label{eq:brownian:radial}
    \mathbb{E}\bracket*{ \d_\La^{2p}(B_{t}, x) } \leq \mathrm{e}^{pt\frac{\kappa}{3}} 2^{p} t^{p} p!
  \end{equation}
\end{lemma}
\begin{proof}
  This result is proved in \cite[Thm.~6]{Thompson} for $\V = 0$ --- this paper mentions the case of drift but it is not clear to us whether the proof still applies. Our setting is, in any cases, covered by the much more general setting of \cite[Thm.~1.2]{ThalmaierThompson}.
  The term $\kappa$ appearing in the exponent in right-hand side only depends on the curvature bound of $M^{\Lambda}$, which is independent on $\Lambda$.
\end{proof}

\subsection{Useful heat kernel bounds}
We recall the following pointwise bounds on the heat kernel on a compact manifold, and on its derivatives (they often follow from strong comparison principles). 
\begin{lemma}
\label{lem:pointwiseBound}
  There exist positive constants $C$, $C'$ and $C_{k}$ ($k \in \mathbb{N}$) depending only on $M$ and $|\Lambda|$ such that for all $t \in (0,1)$, and $x$ and $y \in M^{\Lambda}$:
  \begin{align}
    &\label{eq:heat-kernel:bound:upper} g^{\Lambda}_{t}(\bx,\by) \leq C t^{-\frac{\nn|\Lambda|}{2}} \Exp*{ - \frac{\di_\La(\bx,\by)^{2}}{4t} } & \text{(see \cite[\S~VI]{Chavel})}.
 \\ &\label{eq:heat-kernel:bound:lower} g^{\Lambda}_{t}(\bx,\by) \geq \frac{1}{C} t^{-\frac{\nn|\Lambda|}{2}} \Exp*{ - C' \frac{\di_\La(\bx, \by)^{2}}{4t} } & \text{(see \cite[Thm.~4.2]{SaloffCoste})}.
  \\&\label{eq:heat-kernel:bound:derivative} |\nabla^{k}_{\bx} g^{\Lambda}_{t}(\bx,\by)| \leq C_{k} t^{-\frac{k}{2}} \left(\frac{\di_\La(\bx,\by)}{\sqrt{t}} + 1 \right)^{k} g^{\Lambda}_{t}(\bx,\by) & \text{(see \cite{Hsu})}.
  \end{align}
\end{lemma}

\begin{claim}
For all $T > 0$ fixed, all $t \in (0,T)$, all $k \geq 0$, and all $y \in M^{\Lambda}$
\begin{equation}\label{eq:heat-kernel:bound:derivative:lp}
  \norm{\nabla^{k} g^{\Lambda}_{t}(\cdot,y)}_{\mathscr{L}^{p}} \leq \Cc(T, k, |\La|) \times t^{-\frac{\nn|\Lambda|}{2}\paren*{1-\frac{1}{p}}} \paren*{1 + t^{-\frac{1}{2}}}^{k}.
\end{equation}
\end{claim}
\begin{proof}
Using the log-convexity of $p \mapsto \norm{f}_{\mathscr{L}^{p}}$ for all $f$, and \eqref{eq:heat-kernel:bound:upper}, we get for $t > 0$ and $y \in M^{\Lambda}$:
\begin{equation}
\label{eq:heat-kernel:bound:lp}
  \norm{g^{\Lambda}_{t}(\cdot,y)}_{\mathscr{L}^{p}} \leq \norm*{g^{\Lambda}_{t}(\cdot,y)}_{\mathscr{L}^{1}}^{\frac{1}{p}} \norm*{g^{\Lambda}_{t}(\cdot,y)}_{\mathscr{L}^{\infty}}^{1-\frac{1}{p}} \leq \Cc(|\La|) \times t^{-\frac{\nn|\Lambda|}{2}\paren*{1-\frac{1}{p}}}.
\end{equation}
Moreover, by \eqref{eq:heat-kernel:bound:upper} and \eqref{eq:brownian:radial}, we find for all $t \in (0,T)$, $y \in M^{\Lambda}$
\begin{equation}
\label{eq:heat-kernel:distance:lp}
  \norm{\di_\La(\cdot,y)^{k} g^{\Lambda}_{t}(\cdot,y)}_{\mathscr{L}^{p}} \leq \Cc(T, |\La|) \times t^{-\frac{\nn|\Lambda|}{2}\paren*{1-\frac{1}{p}}} t^{k}.
\end{equation}
Combining \eqref{eq:heat-kernel:bound:lp} and \eqref{eq:heat-kernel:distance:lp} with \eqref{eq:heat-kernel:bound:derivative}, we get \eqref{eq:heat-kernel:bound:derivative:lp}.
\end{proof}

\subsection{Useful heat semigroup inequalities}
The heat kernel bounds listed above can be translated into quantitative Sobolev regularization inequalities for the heat semi-group.
\begin{theorem}[Young's inequality]
\label{YoungHeat}
  Let $s \in \mathbb{R}_{+}$.
  Take $p$, $q$, and $r \in [1,\infty]$ such that $\frac{1}{p} + \frac{1}{q} = 1 + \frac{1}{r}$, then
  \begin{equation}\label{eq:heat-semigroup:young}
    \norm{\mathsf{G}_{t} f}_{\mathscr{W}^{s,r}(\oL)} \leq \Cc(s, p, q, |\La|) \times \norm{f}_{\mathscr{L}^{q}(\oL)} t^{-\frac{\nn|\Lambda|}{2}\paren*{1-\frac{1}{p}}} \paren*{1 + t^{-\frac{1}{2}}}^{s}.
  \end{equation}
  The constant $\Cc$ is locally uniform in time (it also depends implicitely on the parameters of the model).
\end{theorem}

\begin{proof}
  We first establish the result for $s \in \mathbb{N}$ and then use an interpolation argument to conclude.
  \paragraph{Integer case.} We first treat the case $s = k \in \mathbb{N}$. By the Riesz--Bakry inequalities \cite{BakryRieszTransform}, it is then sufficient to show that:
  \begin{equation*}
    \norm{\nabla^{k} \mathsf{G}_{t} f}_{\mathscr{L}^{r}(\oL)} \leq  \Cc(k, p, q, |\La|) \times \norm{f}_{\mathscr{L}^{q}(\oL)} \times t^{-\frac{\nn|\Lambda|}{2}\paren*{1-\frac{1}{p}}} \paren*{1 + t^{-\frac{1}{2}}}^{k}.
  \end{equation*}
  We first prove an intermediary inequality.
  Let $p$ and $q \in [1,\infty]$, and $p'$ the Hölder conjugate of $p$.
  We claim that for every $f \in \mathscr{L}^{p'}(\oL)$
  \begin{equation}
    \label{eq:heat-semigroup:sobolev} \norm{\nabla^{k} \mathsf{G}_{t}f}_{\mathscr{L}^{q}(\oL)} \leq \Cc(k, p, q, |\La|)\times \norm{f}_{\mathscr{L}^{p'}(\oL)} \times t^{-\frac{\nn|\Lambda|}{2} \paren*{1-\frac{1}{p}}\paren*{1-\frac{1}{q}}} \paren*{1 + t^{-\frac{1}{2}}}^{k},
  \end{equation}
  with a constant $\Cc$ locally uniform in time. Using interpolation of $\mathscr{L}^{q}(\oL)$-norms, it is sufficient to prove the inequality \eqref{eq:heat-semigroup:sobolev} for $q \in \{1, \infty\}$.

  Let us start with the case $q = \infty$. By definition of the heat semi-group, we have:
  \begin{equation*}
    \nabla^{k} \mathsf{G}_{t}f(x) = \int f(y) \nabla^{k} g_{t}(x,y) \dd \oL(y),
  \end{equation*}
  and thus, applying Hölder's inequality:
  \begin{equation*}
    \sup_{x \in M} \abs*{ \nabla^{k} \mathsf{G}_{t} f(x) } \leq \norm{f}_{\mathscr{L}^{p'}(\oL)} \norm{\nabla^{k} g_{t}(x,\cdot)}_{\mathscr{L}^{p}(\oL)},
  \end{equation*}
  which, combined with \eqref{eq:heat-kernel:bound:derivative:lp}, proves the case $q = \infty$.
  
  For the case $q = 1$, still by \eqref{eq:heat-kernel:bound:derivative:lp} we have:
  \begin{equation*}
    \norm*{\nabla^{k} \mathsf{G}_{t}f}_{\mathscr{L}^{1}(\oL)} \leq \int_{M^\La} \abs{f(y)} \norm{\nabla^{k} g_{t}(\cdot,y)}_{\mathscr{L}^{1}(\oL)} \dd \oL(y) \leq \norm{f}_{\mathscr{L}^{1}(\oL)} \paren*{1+t^{-\frac{1}{2}}}^{k},
  \end{equation*}
  and since $\mathscr{L}^{p}(\oL)$-norms are non-decreasing with $p$, the result follows.

 We have thus proven \eqref{eq:heat-semigroup:sobolev}, which implies that for all $p \in [1,\infty]$,
  \begin{equation*}
    \norm{\nabla^{k}\mathsf{G}_{t}}_{\mathscr{L}^{1}(\oL) \to \mathscr{L}^{p}(\oL)} \vee \norm{\nabla^{k} \mathsf{G}_{t}}_{\mathscr{L}^{p'}(\oL) \to \mathscr{L}^{\infty}(\oL)} \leq \Cc(k, p, q, |\La|) \times t^{-\frac{\nn|\Lambda|}{2} \paren*{1-\frac{1}{p}}} \paren*{1+t^{-\frac{1}{2}}}^{k}.
  \end{equation*}
The condition on $p$, $q$, and $r$ ensures that $\frac{1}{r} \in [0,\frac{1}{p}]$, thus there exists $\theta \in [0,1]$ such that $\frac{1}{r} = \frac{1-\theta}{p}$. For such a $\theta$, we have $\frac{1}{q} = (1-\theta) + \frac{\theta}{p}$, and we conclude by the Riesz--Thorin interpolation theorem \cite[Thm.~7.1.12]{HormaderLinearOperator}.

  \paragraph{General case.} Fpr arbitrary values of $s \geq 0$, we use another interpolation result.
  Let $\varepsilon \in (0,1)$, and consider the real interpolation method.
  We obviously have $(\mathscr{L}^{q}(\oL), \mathscr{L}^{q}(\oL))_{\varepsilon,2} = \mathscr{L}^{q}(\oL)$, while by \cite[Thms.~4 \& 5]{Triebel}, $(\mathscr{W}^{k,p}(\oL), \mathscr{W}^{k+1,p}(\oL))_{\varepsilon,2} = \mathscr{W}^{k+\varepsilon,p}(\oL)$.
  We conclude by the version of Riesz--Thorin theorem for interpolated spaces \cite{LionsPeetre}.
\end{proof}

\bibliographystyle{alpha}
\bibliography{IGF}

\end{document}